\documentclass[draft,12pt,reqno,a4paper]{amsart}

\usepackage[margin=3cm,footskip=1cm]{geometry}

\usepackage{amssymb} 
\usepackage{amsmath} 
\usepackage{amsfonts}
\usepackage{amsthm} 
\usepackage{mathtools}
\usepackage{color}
\usepackage{enumerate} 
\usepackage[latin1]{inputenc}
\usepackage[shortlabels]{enumitem}
\usepackage{color}
\usepackage{graphicx} 
\usepackage{verbatim}

\newcommand{\supp}{\operatorname{supp}}

\newcommand{\dist}{\operatorname{dist}}

\newcommand{\dbar}{\bar{\partial}}

\newcommand{\grad}{\operatorname{grad}}

\DeclareMathOperator\ran{ran}

\newcommand{\ad}{{\operatorname{ad}}}
\newcommand{\N}{{\mathbb{N}}}

\newcommand{\R}{{\mathbb{R}}}
\newcommand{\Z}{{\mathbb{Z}}}
\newcommand{\C}{{\mathbb{C}}}

\newcommand{\T}{{\mathbb{T}}}

\DeclareFontFamily{U}{mathx}{\hyphenchar\font45}
\DeclareFontShape{U}{mathx}{m}{n}{
      <5> <6> <7> <8> <9> <10>
      <10.95> <12> <14.4> <17.28> <20.74> <24.88>
      mathx10
      }{}
\DeclareSymbolFont{mathx}{U}{mathx}{m}{n}
\DeclareFontSubstitution{U}{mathx}{m}{n}
\DeclareMathAccent{\widecheck}{0}{mathx}{"71}

\renewcommand\i{\mathrm{i}}

\newcommand{\p}{{\mathrm p}}

\renewcommand{\c}{{\mathrm c}}

\newcommand{\s}{{\mathrm s}}
\newcommand{\e}{{\mathrm e}}

\renewcommand{\d}{{\mathrm d}}

\newcommand{\unif}{{\mathrm {unif}}}

\newcommand{\pupo}{{\mathrm {pp}}}

\newcommand{\ener}{{\mathrm {ener}}}
\newcommand{\wave}{{\mathrm {wave}}}
\newcommand{\ac}{{\mathrm {ac}}}
\renewcommand{\sc}{{\mathrm {sc}}}

\renewcommand{\Re}{\operatorname{Re}}
\renewcommand{\Im}{\operatorname{Im}}

\DeclarePairedDelimiter\inp\langle\rangle

\newcommand\parb[2][]{#1 \big ( #2#1\big )}

\newcommand\parbb[2][]{#1 \Big ( #2#1\Big )}
 
 \newcommand{\pp}{{\mathrm {pp}}}
 \renewcommand{\exp}{{\mathrm
    {exp}}} 
\newcommand{\mand}{\text{ \,and\, }}

\newcommand{\cas}{{\textrm {the Cauchy--Schwarz inequality }}}
\newcommand{\caS}{{\textrm {the Cauchy--Schwarz inequality}}}

\DeclarePairedDelimiter\ket{\lvert}{\rangle}
\DeclarePairedDelimiter\bra{\langle}{\rvert}

\DeclareMathOperator*{\slim}{s-lim}

\DeclareMathOperator*{\vHlim}{{\mathcal  H}-lim}

\DeclarePairedDelimiter\abs\lvert\rvert
\DeclarePairedDelimiter\norm\lVert\rVert
\DeclarePairedDelimiter\set{\{}{\}}
\DeclarePairedDelimiter\comm{[}{]}

\newcommand\Step[1]{ 
  \par\bigskip
  \noindent
  \textbf{#1}.\enspace
}

\newcommand{\bD}{{\mathbf D}}

\newcommand{\bM}{{\mathbf M}}

\newcommand{\bZ}{{\mathbf Z}}
\newcommand{\bY}{{\mathbf Y}}
\newcommand{\bX}{{\mathbf X}}

\def\sq{\sqrt}
\def\pa{\partial}

\def\q{\quad}
\def\qq{\qquad}
\def\wt{\widetilde}
\def\wc{\widecheck}
\def\wcH{\widecheck H}
\def\t{\tilde}

\def\br{\breve}

\newcommand{\vA}{{\mathcal A}}

\newcommand{\vD}{{\mathcal D}}
\newcommand{\vE}{{\mathcal E}}
\newcommand{\vF}{{\mathcal F}}
\newcommand{\vG}{{\mathcal G}}

\newcommand{\vH}{{\mathcal H}}

\newcommand{\vL}{{\mathcal L}}

\newcommand{\vO}{{\mathcal O}}

\newcommand{\vT}{{\mathcal T}}

\newcommand{\vS}{{\mathcal S}}

\theoremstyle{plain}
\newtheorem{thm}{Theorem}[section]
 
\newtheorem{defns}[thm]{Definitions} 
\newtheorem{proposition}[thm]{Proposition}
\newtheorem{lemma}[thm]{Lemma} \newtheorem{corollary}[thm]{Corollary}
\newtheorem{cond}[thm]{Condition}
\theoremstyle{definition}

 \newtheorem{remark}[thm]{Remark}
\newtheorem{remarks}[thm]{Remarks}
\newtheorem*{remarks*}{Remarks}
\newtheorem*{remark*}{Remark}

\numberwithin{equation}{section}

\title {Finite energy  subspace for time-periodic   Schr\"odinger operators}

\author{E. Skibsted} \address[E. Skibsted]{Institut for Matematik\\
Aarhus Universitet\\ Ny Munkegade 8000 Aarhus C, Denmark}
\email{skibsted@math.au.dk}

\begin{document}

\begin{abstract} For $N$-body  Schr\"o\-dinger operators with
  time-periodic  short-range pair-potentials we show by a time-dependent commutator method that all channel wave operators
  exist.
For $N=2$ we prove  asymptotic completeness by a simplified version of the  method, recovering Yajima's completeness result \cite{Yaj1} proven by a stationary method.
 We propose a definition of a \emph{finite energy 
    subspace}, intuitively consisting of states  with `finite asymptotic energy'. This 
  geometric notion is used to characterize the \emph{wave  operator 
    subspace} given as the direct sum of the ranges of the channel wave operators. Thus our main result states that the two subspaces coincide. In turn they coincide for $N= 2$  with the orthogonal subspace of the pure point subspace of
    the monodromy  operator (by asymptotic
  completeness). For $N\geq 3$  asymptotic
  completeness for time-periodic  short-range potentials  remains an open problem. The  main result of the paper may in this case potentially   serve as an
  intermediate step for proving (or possibly disproving) asymptotic completeness. Another potential ingredient could be a   key intermediate result of the paper,   asserting  that the states  orthogonal to the pure point subspace obey a (sharp)   \emph{minimal velocity bound}. Thus it remains an open problem possibly to derive a good maximal velocity bound for $N\geq 3$. Our results apply to $N$-body systems of particles in a time-periodic external electric field with time-mean equal to zero, for example the AC-Stark model.
  \end{abstract}

\allowdisplaybreaks

\maketitle


\tableofcontents

\section{Introduction}\label{sec:Introduction} The large time behaviour of a quantum-mechanical system of $N$ particles interacting by time-independent short-range pair-potentials is well understood: In the center of mass framework the channel wave operators exist and they are complete.  This  asymptotic completeness result constitutes  a satisfactory  large time description of  the evolution of the scattering states. In the present paper we will study scattering problems for time-dependent (more precisely) time-periodic short-range many-body Hamiltonians. Although  asymptotic completeness is not shown/known, we do obtain various new related results. Our principle tool is commutator methods.

\subsection{The  time-independent model}\label{subsec:TIMEin} We  briefly recall the
mentioned  completeness result for the  time-independent case,  partially to motivate the scattering problems to be studied in our paper. 
Thus let us consider $n$-dimensional  particles $ j=1, \dots,  N$ interacting by short-range
pair-potentials
\begin{equation}\label{eq:SHRT}
	V_{ij} (x_i -x_j)=\vO(\abs{x_i -x_j}^{-1-2\mu}), \quad \mu>0.
\end{equation}
 The total 
Hamiltonian  reads 
\begin{equation}\label{eq:HamNoTime}
H=H^{a_{\max}}=-\sum_{j = 1}^N \frac{1}{2m_j}\Delta_{x_j} + \sum_{1 \le i<j \le N} V_{ij} (x_i -x_j).
\end{equation} It acts  
on the Hilbert space $\vH=L^2(\bX)$, where the centre of mass space
\begin{equation*}
  \bX=\bX^{a_{\max}}  =\set[\Big]{ x=(x_1,\dots ,x_N)\in \R^n\times \dots \times  \R^n\
    \mid\sum_{j=1}^{N} m_j x_j =
0}. 
\end{equation*}
 For the cluster decompositions  $a\neq a_{\max}=\set{1, \dots,  N}$
there are  sub-Hamiltonians $H^a$  defined similarly
acting on   $L^2(\bX^a)$, where   $\bX^a\subsetneq \bX$ consists of the  vectors for which the centre  of mass of each of the clusters of $a$ vanishes.

For any  \emph{channel}
$\alpha =(a,\lambda^\alpha, u^\alpha)$, i.e.   $a$ be  a cluster
decomposition, $a\neq a_{\max}$, $\lambda^\alpha\in \R$ and  
$(H^a-\lambda^\alpha)u^\alpha=0$ for a normalized $u^\alpha\in
L^2(\bX^a)$,  the corresponding forward and backward \emph{channel wave operators} $W_\alpha^{\pm}$  are  given by
\begin{subequations}
 \begin{equation}\label{eq:wave_opq}
  W_{\alpha}^{\pm}\psi_a=\lim_{t\to \pm\infty}\e^{\i tH}\parbb{\parb{\e^{-\i t\lambda^\alpha }u^\alpha }\otimes
\parb{\e^{\i t \Delta_{x_a}}\psi_a}};\quad \psi_a\in L^2(\bX_a).
 \end{equation} Here we have decomposed orthogonally $\bX =\bX^a\oplus \bX_a$. The existence of the channel wave operators as well as
 their  \emph{completeness}, i.e. the property  that the
 absolutely continuous subspace  of $H$ is  spanned by their ranges
 \begin{equation}\label{eq:ACo}
   \sum_\alpha^\oplus \, \ran(W_\alpha^\pm)=\vH_{\rm ac}(H)=\vH_{\rm{ pp}}^{\perp}(H),
 \end{equation} 
\end{subequations}
is a well-studied   problem in mathematical physics and since decades ago  definitively solved,  see   \cite{SS, Gra, De, Ya1,   DG,  HS,  Is}. Some of the mentioned works (partly being  reviews) include scattering theory for  long-range potentials, however   all treatments are based on  commutator methods. The present paper  relies on commutator methods as well,  used somewhat in the same spirit as  in the  literature. More generally   our  paper  depends on ideas  from the mentioned literature  as well as   on ideas and results  from \cite {Yaj1, KY, MS}, specifically   devoted to time-dependent problems. 

\subsection{A physical   time-dependent model}\label{subsec:PhysTIME}
 Now let us add $-\sum_{j = 1}^N \,q_j\  E(t)\cdot x_j $ as a new  term on the right-hand side of (\ref{eq:HamNoTime}), where $ E(t)\in \R^n$ is any  time-periodic vector  with zero time-mean  and $q_j$ is the charge of the $j$th particle. This Hamiltonian models  a  quantum-mechanical system of charged particles in a constant time-periodic external electric field with time-mean equal to zero and interacting by short-range
pair-potentials. Again   we can trivially factor out the center of mass motion, leading us to  study the (reduced)   evolution in $L^2(\bX)$, which is now generated by a time-periodic Hamiltonian $H=H(t)$.
 It is   natural to  ask the same question, how do the `scattering states' in $L^2(\bX)$ evolve? This is the topic of the present paper.

 Let us for simplicity  assume that the period of $ E(t)$ is one. Hence we assume that
 \begin{equation}
    E(t+1)= E(t)\mand \int_0^1\,  E(t)\, \d t=0.
 \end{equation} Let us write the generator $H(t)$ in the following way. Using as before the notation  $x=(x_1, \dots, x_N)\in (\R^n)^N$ and $y=(y_1, \dots, y_N)\in (\R^n)^N$, the expression 
 \begin{equation*}
   x\cdot y= \sum_{j = 1}^N \,2m_jx_j\cdot y_j
 \end{equation*} defines an inner product   for which the total Hamiltonian splits into a direct sum,
 \begin{equation*}
   H(t)\otimes I +I\otimes H_{\rm CM}(t),
 \end{equation*}
  corresponding to the orthogonal splitting $(\R^n)^N=\bX\oplus \bX_{\rm CM}$. Let $Q=q_1+\cdots q_N$,  $M=m_1+\cdots m_N$ and
 \begin{equation*}
   \vE(t)= \tfrac 12 \parbb{\parb{\tfrac {q_j} {m_j}-\tfrac Q {M}}E(t), \dots, \parb{\tfrac {q_N} {m_N}-\tfrac Q {M}}E(t)}.
 \end{equation*} Then $ \vE(t)\in \bX$, and we can express the centre of mass Hamiltonian $H(t)$ in terms of the momentum operator  $p=-\i\nabla_x$ as
 \begin{equation*}
   H=H(t)= p^2-\vE(t)\cdot x+ \sum_{1 \le i<j \le N} V_{ij} (x_i -x_j).
 \end{equation*} Clearly, unless the charge-mass ratio $\tfrac {q_j} {m_j}=\tfrac Q {M}$ for all $j=1,\cdots, N$, this Hamiltonion  has a non-trivial time-dependence, although of course being periodic  with zero time-mean. It generates a dynamics,
 \begin{align*}
   \i\frac{\d}{\d t}{U}(t,s)\psi =
 H(t){U}(t,s)\psi\mand {U}(s,s)=I,
 \end{align*}
  and we are interesting in understanding the behaviour of the evolution  ${U}(t,0)\psi$ of  states $\psi $ being orthogonal to the pure point subspace     of the \emph{monodromy operator}  $U(1,0)$. (These states  are the scattering states of our model.)

  \subsubsection{A reduction of  the physical   model}\label{subsubsec:redu}
  Following \cite{MS} it is convenient to apply a Galilean transformation to simplify the evolution. Hence let
 \begin{align*}
   &b(t)=\int_0^t\, \vE(s)\,\d s-b_0, \q b_0=\int_0^1\, \int_0^t\,\vE(s)\,\d s\,\d t, \\
   &c(t)=2\int_0^t\, b(s)\,\d s-c_0, \q c_0=2\int_0^1\, \int_0^t\,b(s)\,\d s\,\d t, \\
   &\alpha(t)=\int_0^t\, |b(s)|^2\,\d s-t\alpha_0, \q \alpha_0=\int_0^1\, |b(s)|^2\,\d s.
 \end{align*} Here  $b(t),c(t)\in \bX$ and  these  vectors  are periodic with period one and  zero mean, and the scalar $\alpha(t)$ is periodic as well. Defining then
 \begin{subequations}
 \begin{equation}\label{eq:Sgal}
   S(t)= \e ^{-\i \alpha(t)}\exp ( \i b(t)\cdot x)\exp( -\i c(t)\cdot p),
 \end{equation} we  can record that
 \begin{align*}
   S(t)^*H(t)S(t)&=H_{\rm G}(t)-\i\parb{\tfrac \d{\d t}S(t)^*}S(t)+\alpha_0;\\
   H_{\rm G}(t)&=p^2 +\sum_{1 \le i<j \le N} V_{ij} (x_i -x_j+ c(t)_i-c(t)_j),
 \end{align*} Therefore  
 \begin{equation}\label{eq:conjug}
   S(t)^*U(t,s)S(s)= \e ^{\i (s-t) \alpha_0} U_{\rm G}(t,s),
 \end{equation} where $U_{\rm G}$ denotes  the  dynamics
 \begin{align*}
   \i\frac{\d}{\d t}{U_{\rm G}}(t,s)\psi =
 H_{\rm G}(t){U_{\rm G}}(t,s)\psi\mand {U_{\rm G}}(s,s)=I.
 \end{align*}
   \end{subequations} Clearly  the generator $H_{\rm G}(t)$ is a time-dependent Schr\"odinger operator with time-periodic short-range pair-potentials. The fundamental question is then:  How does the  evolution  ${U_{\rm G}}(t,0)\psi$ of  states $\psi $ being orthogonal to the pure point subspace  of the monodromy operator $U_{\rm G}(1,0)$, say denoted by $\vH_{{\rm G},\rm{ pp}}$,   behave?  

 Motivated by the formulation of completeness for the time-independent  short-range case we are lead to consider the following  conjecture on this problem.

For any  \emph{channel}
$\alpha =(a,\lambda^\alpha, u^\alpha)$, i.e.  $a$  be  a  cluster decomposition  $a\neq a_{\max}$, $\lambda^\alpha\in [0,2\pi)$ and  
$U^a_{\rm G}(1,0)u^\alpha = \e^{-\i \lambda^\alpha }u^\alpha$ for a normalized $u^\alpha\in
L^2(\bX^a)$, we introduce the corresponding \emph{channel wave
  operators}  $W_{{\rm G},\alpha}^{\pm}$  by 
\begin{subequations}
\begin{equation}\label{eq:wave_op0}
  W_{{\rm G},\alpha}^{\pm}\psi_a=\lim_{t\to \pm\infty}U_{\rm G}(0,t)\parbb{\parb{\e^{-\i \lambda^\alpha [t]}U_{\rm G}^a(t- [t])u^\alpha} \otimes
\parb{\e^{-\i tp_a^2} \psi_a}};\, \psi_a\in L^2(\bX_a).
 \end{equation} Here     $ [t]\in [0,1)$ denotes  the integer
 part of  $t\in \R$, $p_a=-\i \nabla_{x_a}$ and as before we have decomposed orthogonally $\bX =\bX^a\oplus \bX_a$. Note also that $U_{\rm G}^a(t)=U_{\rm G}^a(t,0)$
 denotes the dynamics generated by sub-Hamiltonians $H_G^a(t)$  defined similarly as $H_G(t)$ (but acting on $L^2(\bX^a)$ rather than on $L^2(\bX)$). 
 Since asymptotic completeness is not known in this  setting let us  state it as  a \underline{conjecture}:  
 \begin{equation}\label{eq:ACo2}
   \vH_{{\rm G},\rm{ wave}}^\pm:=\sum_\alpha^\oplus \, \ran( W_{{\rm G},\alpha}^{\pm})=\vH_{{\rm G},\rm{ pp}}^{\perp}.
 \end{equation}  
\end{subequations}
If (\ref{eq:ACo2}) was known,  we would then have a similar \emph{completeness} result  for the original time-dependent problem. This would  be  given be applying the Galilean transformation in the opposite direction. More precisely note for example the property  $\vH_{\rm{ pp}}^{\perp}=S(0)\vH_{{\rm G},\rm{ pp}}^{\perp} $, where $\vH_{\rm{ pp}}$ is the pure point subspace  of the monodromy operator $U(1,0)$. Similarly the operators $S(0)W_{{\rm G},\alpha}^{\pm}S(0)^*_a$ are nothing but the physical channel wave
  operators, see   the computation \eqref{eq:sWave} and (\ref{eq:wave_op022}) below. Hence a rather satisfactory  large time description of  the evolution of all  scattering states of the  physical model would follow.

Although (\ref{eq:ACo2}) remains  unproven  in general, let us mention that the result   is known for $N=2$,  as well as for $N=3$ or bigger under  strong additional conditions, see \cite{Yaj1, Yaj2, Na}. The known  conditions for $N\geq 3$ appear  very  strong, and  one could suspect  that in general (\ref{eq:ACo2}) might  not be valid (potentially rooted in examples from  Classical Mechanics?).  At the practical level  there are two  obstacles one (presumably) would  need to overcome for possibly proving  (\ref{eq:ACo2}):

1) Does the evolution $\psi_G(t)=U_{\rm G}(t,0)\psi$, $\psi\in\vH_{{\rm G},\rm{ pp}}^{\perp}$, concentrate in the  strictly outgoing/incoming phase space regions for $t\to \pm \infty$, respectively?

2)  Does the evolution  $\psi_G(t)$   concentrate in the finite kinetic energy region?

Anyway,   if (\ref{eq:ACo2}) was known, then   these problems  would   have confirmative answers.
They are non-trivial since  the Hamiltonian is time-dependent. Quite opposite  there are short direct confirmative answers in the 
time-independent case.

Let us elaborate on 2) by defining the \emph{finite energy   subspaces}
    $\vH^\pm_{{\rm G},\ener}$ as 
      \begin{equation*}
     \vH^\pm_{{\rm G},\ener}= \set[\big]{\psi\in \vH_{{\rm G},\rm{ pp}}^{\perp} \mid  \lim_{E\to \infty}\limsup_{t\to \pm\infty}\, \norm{F(p^2>E)\psi_G(t)}=0}.
   \end{equation*} 
     (Here and below $F(\cdot >E)=1_{(E,\infty)}$ and $F(\cdot <E)=1_{(-\infty,E)}$.)  Then the  precise confirmative answer to 2) would be the assertion  $\vH_{{\rm G},\rm{ pp}}^{\perp}=\vH^\pm_{{\rm G},\ener}$. This property   is not known.

     \subsection{Results for the physical and the reduced physical  models}\label{subsec:resultsOBT}
The main result of this paper stated for the reduced physical  model is (up to milder conditions supplementing \eqref{eq:SHRT} and proven in a generalized setting) the assertion
    \begin{equation}\label{eq:ACo23}
   \vH_{{\rm G},\rm{ wave}}^\pm=\sum_\alpha^\oplus \, \ran( W_{{\rm G},\alpha}^{\pm})=\vH^\pm_{{\rm G},\ener}.
 \end{equation}
 In addition we show that for $N=3$
 \begin{equation}
   \label{eq:ortFor}
   \vH^\pm_{{\rm G},\ener}=\vH_{{\rm G},\rm{ pp}}^{\perp}  \ominus \vH_{{\rm G},{\rm ie}}^\pm ,
 \end{equation}
   where the \emph{increasing energy  subspaces} are  given as
 \begin{equation*}
     \vH_{{\rm G},{\rm ie}}^\pm= \set[\big]{\psi\in \vH_{{\rm G},\rm{ pp}}^{\perp} \mid  \forall E>0:\q \lim_{t\to \pm\infty}\, \norm{F(p^2< E)\psi_G(t)}=0}.
   \end{equation*}

   While (\ref{eq:ortFor}) is not known for $N\geq 4$, we do show that if $\vH_{{\rm G},{\rm ie}}^\pm =\set{0}$ for $H_{\rm G}$ as well as for all sub-Hamiltonians then indeed asymptotic completeness $\vH_{{\rm G},\rm{ wave}}^\pm=\vH_{{\rm G},\rm{ pp}}^{\perp}$ holds (which for $N=3$ is a consequence of (\ref{eq:ACo23}) and (\ref{eq:ortFor})). 
   
  In the process of proving (\ref{eq:ACo23})  we prove the localization assertion in  1) for all states $\psi\in \vH_{{\rm G},\rm{ pp}}^{\perp}$. Hence for such states the evolution  $\psi_G(t)$  indeed   concentrates in the  strictly outgoing/incoming phase space regions for  $t\to \pm\infty$. In particular  $\psi_G(t)$  concentrates  in   the strictly positive  kinetic energy region (which already    is a non-trivial property due to  the time-dependence of the Hamiltonian).

The analogue of (\ref{eq:ACo23}) for the  physical model is
\begin{equation}\label{eq:ACo23P}
   \sum_\alpha^\oplus \, \ran( W_{\alpha}^{\pm})=\vH^\pm_{\ener},
 \end{equation}
 \begin{subequations} where (given in terms of  the  canonical splitting $S(t)=S(t)^a\otimes S(t)_a$) 
   \begin{align}
     \label{eq:sWave}
     \begin{split}
     W_{\alpha'}^{\pm}&=S(0)W_{{\rm G},\alpha}^{\pm}S(0)^*_a,\\  \alpha =(a,\lambda^\alpha, u^\alpha),\q
      \alpha'&= (a,\lambda^\alpha+\alpha_0^a, S(0)^au^\alpha),\q \alpha_0^a=\int_0^1\, |b(s)^a|^2\,\d s,
     \end{split}
   \end{align}
     are the physical channel wave
  operators (appearing here reparametrized) and 
   \begin{equation*}
     \vH^\pm_{\ener}= \set[\big]{\psi\in \vH_{\rm{ pp}}^{\perp} \mid  \lim_{E\to \infty} \limsup_{t\to \pm\infty}\, \norm{F(p^2>E)\psi(t)}=0};\q \psi(t)=U(t,0)\psi.
   \end{equation*}

   As for (\ref{eq:sWave}) we note that the expression 
  $p_a^2-\vE(t)_a\cdot x_a$  defines a Hamiltonian on $L^2(\bX_a)$ generating a dynamics $U_a$ for which there exist the limits   
 \begin{equation}\label{eq:wave_op022}
  W_{\alpha}^{\pm}\psi_a=\lim_{t\to \pm\infty}U(0,t)\parbb{\parb{\e^{-\i \lambda^\alpha [t]}U^a(t- [t])u^\alpha} \otimes
\parb{U_a(t)\psi_a}},\, \psi_a\in L^2(\bX_a).
\end{equation} Here  $U^a(t)=U^a(t,0)$
 denotes the dynamics generated by sub-Hamiltonians $H^a(t)$  defined similarly as $H(t)$. 
 In fact using \eqref{eq:Sgal} and \eqref{eq:conjug} we deduce that  the  limits   exist  by  applying $S(0)$ to both sides of \eqref{eq:wave_op0}. The right-hand side simplifies as $W_{\alpha'}^{\pm}S(0)_a\psi_a$. In particular $W_{\alpha'}^{\pm}$ exist and therefore in turn also $W_{\alpha}^{\pm}$. Obviously we obtained the formula (\ref{eq:sWave}) as well.
\end{subequations}

As for (\ref{eq:ACo23P})   we obtain the formula by  applying $S(0)$ to both sides of  (\ref{eq:ACo23}) and then  use  \eqref{eq:sWave},  (\ref{eq:wave_op022}) and the fact that  $\vH^\pm_{\ener}=S(0)\vH^\pm_{{\rm G},\ener}$. Of course completeness would amount to (\ref{eq:ACo23P}) with the right-hand side being replaced by $\vH_{\rm{ pp}}^{\perp} $.

For  $N=3$ completeness  is equivalent to 
 the property  $ \vH_{{\rm ie}}^\pm =\set{0}$,  where
 \begin{equation*}
     \vH_{\rm ie}^\pm= \set[\big]{\psi\in \vH_{\rm{ pp}}^{\perp} \mid  \forall E>0:\,\, \lim_{t\to \pm\infty}\, \norm{F(p^2< E)\psi(t)}=0},
   \end{equation*} since in this case  
   \begin{equation}
     \label{eq:ortForPhy}
     \vH^\pm_{\ener}=\vH_{\rm{ pp}}^{\perp}  \ominus \vH_{{\rm ie}}^\pm. 
   \end{equation}
   Note that (\ref{eq:ortForPhy}) for  $N=3$  follows similarly from the analogous assertion  \eqref{eq:ortFor}. For $N\geq 4$
   the  criterion for completeness (similarly derived) reads that $\vH_{{\rm ie}}^\pm =\set{0}$ for $H$ as well as for all sub-Hamiltonians.
   
The states in $\vH_{\rm ie}^\pm$ correspond in  Classical Mechanics to orbits  with increasing kinetic energy $p(t)^2\to \infty$. It is not known if such orbits exist (ideally  they are  absent).

Although our analysis is mainly concentrated on the study of $\vH^\pm_{{\rm G},\ener}$ (and hence including  $\vH^\pm_{\ener} $) we obtain a new propagation estimate     for states in 
$\vH_{{\rm G},\rm{ pp}}^{\perp}$. This is the following pointwise \emph{minimal velocity bound}. 
 \begin{equation}
\label{eq:decayInx22G}
  \forall \psi\in \vH_{{\rm G},\rm{ pp}}^{\perp} \,\forall \epsilon>0:\,\, \lim_{|t|\to \infty}\, \norm{F(|x|< |t|^{1-\epsilon})\psi_G(t)}=0.
\end{equation} Of course a similar assertion  is valid  for  states in    $\vH_{\rm{ pp}}^{\perp}$. This    minimal velocity bound is important  for our proof of (\ref{eq:ACo23}), however we believe it has   independent interest. Note in particular that the lower velocity bound is valid for states where  we do not know (and possibly  maybe not even expect to have) a finite upper velocity bound. For the time-independent model energy conservation yields lower as well as upper velocity bounds, see for example \cite[Proposition  4.11.1, Theorem 4.13.1]{DG} and \cite[Theorems 3.8, 3.11, 3.12]{Is} (including  integral versions of the
bounds).  

\subsubsection{Concluding remarks}\label{subsubsec:Concluding remarks}
Clearly  (\ref{eq:decayInx22G}) is a characterizing geometric property of  the space $\vH_{{\rm G},\rm{ pp}}^{\perp}$, and we also remark  that the integral  local decay estimates from \cite{MS} (see \eqref{eq:LocalDecay}) are related to (\ref{eq:decayInx22G}), however still rather different. These integral estimates are important for us throughout the paper, and in particular they are used to derive (\ref{eq:decayInx22G}) (explicitly we exploit in this case \eqref{eq:LocalDecay} with $r=0$ and $s> \tfrac 12$ arbitrary). The proof  is somewhat    complicated. Besides the mentioned  estimates from \cite{MS} it is based on commutator methods (as for all of the paper) and in particular on unbounded  time-dependent \emph{propagation observables}. For other purposes  we  shall frequently
use the concept of bounded propagation observables.

A  main goal  of our work is to see how far the techniques  of positive commutators can be pushed in the direction of proving (or possibly disproving) the completeness conjecture. These tools  are  rather different from   Kato's  perturbative smoothness technique   (originating in  \cite{Ka1}) on which the mentioned works \cite{Yaj1, Yaj2, Na} rely. Although we do not achieve the ultimate  completeness result, our efforts yield   new related results of geometric nature (as outlined above). Apart from a (quantum)   smoothing effect,    the positive commutator techniques are strongly related  to (and in fact possibly   limited of) Classical Mechanics. This is further emphasized in this paper in that we only consider sufficiently smooth potentials. However even with this assumption it is a major  challenge to see that various `quantum errors' of the
techniques are `negligible', for example and  typically  implemented by the integral  local decay estimates from \cite{MS}. As a general feature it follows that   the classical Poisson bracket determines  the desired content of  commutator positivity, as expected.

 \subsection{Related literature and organization  of  the paper }\label{subsec:contOrg}
 There are other works on  scattering for a many-body quantum-mechanical system of charged particles in a constant  external electric field,  time-independent or time-periodic, see \cite{AT, HMS, Ad}.  However for these works  the  time-mean is  nonzero, which dramatically amounts to different physics as well as  different  methods of proof. In particular asymptotic completeness is known  for the nonzero time-mean models. The present paper appears more premature in the sense that asymptotic completeness is not proven (or disproven).

 Let us also mention the works \cite{KY, BM} whose  results    are given in  the same spirit as (\ref{eq:ACo23}) (although being  different).    More precisely their main results    give  geometric characterizations  of states in the range of certain wave operators, and clearly (\ref{eq:ACo23}) has the same nature.

 The  derivation of (\ref{eq:decayInx22G}) is  given in Section \ref{sec:Preliminary localization results}.
 We do a further analysis in Section \ref{sec:Localization to  outgoing  region} and introduce an appropriate \emph{asymptotic observable} to show that the forward  evolution indeed concentrates  in  the strictly outgoing region.  The setup in Sections \ref{sec:Proof of {eq:Strength}} and \ref{sec:Asymptotic clustering}  is more `time-independent like', hence in some sense  simpler to deal with. On the other hand the analysis of  Sections \ref{sec:Proof of {eq:Strength}} and \ref{sec:Asymptotic clustering} involves a new complicated \emph{square root commutation problem}. We overcome it by invoking the  full strength of the integral  local decay estimates from \cite{MS} (which amounts to the  explicit content of  momentum  space decay given by \eqref{eq:LocalDecay} with  $r>0$). As a main result we prove an \emph{asymptotic clustering} property for states in $\vH^+_{{\rm G},\ener}$, which is an important ingredient of the proof of (\ref{eq:ACo23})  (given for simplicity only for the plus case). This concept can also be used to characterize the asymptotic completeness property  (\ref{eq:ACo2}) (this is done in Section \ref{sec:Proof
  of Theorem (N)}). In  Sections \ref{sec:Existence of
  channel wave operators}-\ref{sec:Proof
  of Theorem (N)} we prove various results  stated in Subsection \ref{subsec:Main result}. The  proofs are based on Yafaev's constructions/observables \cite{Ya1},  to be  recalled  in Section \ref{sec: Yafaev type
  constructions}. These  results along with the needed results from \cite{MS} are stated in the familiar setting of  generalized Schr\"odinger operators. We devote Subsection \ref{subsubsec::Short-range N-body Hamiltonians} to an account of  conditions, notation, extracts from  \cite{MS}  and other  prerequisites.

\section{$N$-body Hamiltonians, notation  
  and  results}\label{sec:body Hamiltonians, limiting absorption
  principle  and notation}

\subsection{Generalized $N$-body short-range  Hamiltonians}\label{subsubsec::Short-range N-body Hamiltonians}

Let $\bX$ be                    
a (nonzero) finite dimensional real inner product space,
equipped with a
finite family $\{\bX_a\}_{a\in \vA}$ of subspaces closed under intersection:
For any $a,b\in\mathcal A$ there exists $c\in\mathcal A$, denoted
$c=a\vee b$,  such that $\bX_a\cap\bX_b=\bX_c$.
 We
 order $\vA$ by  writing $a\leq b$ (or equivalently as $b\geq a$) if
$\bX_a\supseteq \bX_b$. 
It is assumed that there exist
$a_{\min},a_{\max}\in \vA$ such that 
$\bX_{a_{\min}}=\bX$ and 
$\bX_{a_{\max}}=\{0\}$. We will use the abbreviated notation
$a_0=a_{\max}$. Let $\vA_1=\vA\setminus
  \{a_{0}\} $,  $\vA_2=\vA\setminus
  \{a_{\min}\} $ and $\vA_3=\vA_1\cap
\vA_2$. For  $a\in\vA_2$ the subspaces $\bX_a$   are called  \emph{collision
 planes}. For    $a\in\vA_1$ we will  use the notation $d_a=\dim
\mathbf X_a$ and in particular  the abbreviated notation
$d=d_{a_{\min}}=\dim \mathbf X$.  
 
For any   \emph{chain}  $a_{\min}=a_1\lneq
\cdots \lneq a_n=a_{0}$ in $\vA$, the integer $n\geq 2$ is referred to as its
\emph{length}. The model is referred to as an $N$-body model  if    $N$ is 
 the biggest appearing  length in $\vA$. Hence the $2$-body model is
based on the structure 
 $\vA=\set{a_{\min},a_{0}}$. The   $3$-body model  is given  by the
 properties $ \vA_3 \neq \emptyset$  and 
 that
 $\bX_a\cap\bX_b=\set{ 0}$ for all  different $a,b\in \vA_3$. 
 The condition of this  paper, stated below as Condition
   \ref{cond:smooth2wea3n12}, is  a  short-range condition. In particular under this condition the $N$-body scattering  theory has a  canonical
   meaning (as to be demonstrated). While the  short-range scattering
   theory for the $2$-body   model is fairly well understood \cite{Ho,
     Yaj1}, our knowledge on   many-body scattering  theory   is rather limited.

Let $\bX^a\subseteq\bX$ be the orthogonal complement of $\bX_a\subseteq \bX$,
and denote the associated orthogonal decomposition of $x\in\bX$ by 
$$x=x^a\oplus x_a=\pi^ax\oplus \pi_ax\in\bX^a\oplus \bX_a.$$ 
The vectors $x^a$ and $x_a$ may be  called the \emph{internal
  component}  and the 
\emph{inter-cluster component}  of $x$, respectively. The momentum operator
$p=-\i \nabla$ decomposes similarly, $p=p^a\oplus p_a$. For $a\neq a_{\min} $ the  unit sphere  in $\mathbf X^{a}$ is denoted
by $\mathbf{S}^{a}$, while for $a\neq a_{0} $ the  unit sphere  in $\mathbf X_{a}$ is denoted
by $\mathbf{S}_{a}$. We shall use the notation 
$r(x)=\inp{x}=\parb{1+\abs{x}^2}^{1/2}$ for  $x\in \mathbf X$ (or
possibly for another normed space).  For $x\in \mathbf X\setminus\set{0}$ we
write $\hat x=x/\abs{x}\in \mathbf{S}^{a_0}$.

A real-valued measurable function $V\colon\bX\to\mathbb R$ is 
a \textit{potential of many-body type} 
if there exist real-valued measurable functions
$V_a\colon\bX^a\to\mathbb R$ such that 
\begin{equation*}
V(t,x)=\sum_{a\in\mathcal A}V_a(t,x^a)\ \ \text{for }x\in\mathbf X.
\end{equation*} We take $V_{a_{\min}}(t, \cdot)=0$ and  impose throughout the paper the
following condition for    $a\neq a_{\min}$. 
\begin{cond}\label{cond:smooth2wea3n12}
    There exists $\mu\in (0,1/2)$  such that for all $a\in \vA_2$ the 
    potential $V_a=V_a(t,x^a)$ fulfils:
    \begin{enumerate}[(1)]
    \item \label{item:shortr}$V_a(t+1,x^a)=V_a(t,x^a)$, $t\in\R$.

    \item\label{item:Regularity} $V_a\in C^2(\R\times X^a)$.
\item\label{item:BoundsatInfinity}$V_a(t,y)=o(1)$,  $y\cdot\nabla_y
  V_a(t,y)=o(1) $ and  
  \begin{equation*}
    |y|^{|\gamma|}\partial_y^\gamma \partial_t^k V_a(t,y) = \vO(1)
\text{ for }0\leq k\leq 1 \mand k+|\gamma|\leq 2.
  \end{equation*}
 \item \label{item:shortl}$\abs{y}^{1+2\mu}
      \pa_y^\gamma V_a(t,y)= \vO(1)$ for $ |\gamma|\leq1$.
    \end{enumerate} Here $o(1)$ means $o(1)\to 0$
    for $|y|\to\infty$, uniformly in
$t$, and similarly for $\vO(1)$. 
\end{cond}

For some  intermediate
results the condition \ref{item:shortl} is not needed.  With  the short-range bound   \ref{item:shortl} included,  the scattering theory has a unique meaning.    Apart from the fact that local singularities are not included  in Condition \ref{cond:smooth2wea3n12}, it defines  a rather general  class of time one periodic short-range potentials.

 For any $a\in\vA$  we  introduce   associated  {Hamiltonians} $H^a$
 and $H_a$   as follows. 
For $a=a_{\min}$  we define
$H^{a_{\min}}=0$ on $\mathcal H^{a_{\min}}=L^2(\set{0})=\mathbb C $
and $H_{a_{\min}} =
p^2$ on $L^2(\bX)$, respectively. 
For $a\neq a_{\min}$ 
we let 
\begin{equation*}
V^a=V^a(t,x^a)=\sum_{b\leq a} V_{b}(t,x^b)
,\quad
x^a\in \bX^a,
\end{equation*} 
and  define  then 
\begin{equation*}
 H^a=H^a(t)=-
\Delta_{x^a} +V^a\ \ 
\text{on }\mathcal H^a=L^2(\bX^a)\mand H_a=H^a \otimes I +I\otimes
p^2_a\text{ on }L^2(\bX).
\end{equation*} 
We  abbreviate 
\begin{align*}
V^{a_0}=V=V(t),\quad
 H^{a_0}=H= H(t),\quad 
 \mathcal H^{a_0}=\mathcal H=L^2(\bX).
\end{align*}
 The operators  $H$,  $H^a$ and $H_a$, $a\neq a_0$, generate dynamics, i.e. families of
 unitary operators $U(t,s)$, $U^a(t,s)$
 and $U_a(t,s)$ fulfilling certain requirements. We specify those for
 $U(t,s)$ for arbitrary $r,s,t\in\R $ (the
 conditions are similar  for $U^a(t,s)$ and $U_a(t,s)$):
 \begin{subequations}
 \begin{align}\label{eq:Iden} {U}(s,s)&=I,\\
{U}(s,r){U}(r,t)&={U}(s,t),\qquad  \label{eq:ChapmanKolm}\\
{U}(t+1,s+1) &= {U}(t,s), \label{eq:periodicU}\\
\i\frac{\d}{\d t}{U}(t,s) &= H(t){U}(t,s)
 \quad \text{(in the strong sense)}. \label{eq:solSchr}
\end{align}  
 \end{subequations} The meaning of \eqref{eq:solSchr} is that the
 formula 
   $\i\frac{\d}{\d t}{U}(t,s)\psi =
 H(t){U}(t,s)\psi$  holds for any $ \psi\in\inp{p}^{-2}L^2(\bX)$.  In addition it holds that for   any $\kappa \in [0,2]$ and $ \psi\in\inp{p}^{-\kappa}L^2(\bX)$ the map $\R^2\ni (t,s)\to \inp{p}^{\kappa}{U}(t,s)\psi\in L^2(\bX)$ is continuous. By Yosida's theorem indeed such dynamics exists, see \cite[Section XIV.4]{Yo}.

Abbreviating  ${U}(t)={U}(t,0)$, ${U^a}(t)={U^a}(t,0)$ and
${U_a}(t)={U_a}(t,0)$, the \textit{thresholds} of the \emph{monodromy  operator}  $U(1)$ are by
definition the
eigenvalues of the    sub-monodromy  operators $U^a(1)$, 
$a\in  \vA_1$,
  and hence they constitute the   set 
\begin{equation*}
 \vT=\vT (U(1))= \bigcup_{a\in\vA_1} \sigma_{\pupo}( U^a(1)).
\end{equation*}
 Let $ [t]\in \Z$ denote   the integer
 part of any given $t\in \R$ (characterized by  $ t-[t]\in [0,1)$). We can then record that 
\begin{equation}
  \label{eq:groupLike} U(t)=U(t-[t])U([t])=U(t-[t])U(1)^{[t]},  
\end{equation} and similarly for $U^a(t)$ and $U_a(t)$.

 It is known from  \cite{MS} that the threshold set $\vT$  
is closed and countable, and that the singular continuous spectrum $\sigma_{\sc}(U(1)) = \emptyset$. Moreover the set of non-threshold eigenvalues
  is discrete in $\T\setminus \vT $. 
We introduce the notation
\begin{equation*}
  \vT_{\p}=\vT_{\p}(U(1))=\sigma_{\pp}(U(1))\cup
  \vT(U(1)).
\end{equation*}

Away from $\vT_{\p}$ the following class of integral local decay estimates (of local Kato smoothness type, \cite{Ka2}) is known. 

  \begin{thm}(\cite{MS})\label{thm:LocalDecay}
 Suppose  $V=\Sigma_{a\in\mathcal A}\,V_a$ satisfies Condition \ref{cond:smooth2wea3n12}
\ref{item:shortr}--\ref{item:BoundsatInfinity}.  Let $s>\frac12>r\geq 0$ and
let $g$ be a bounded Borel-measurable function on the unit-circle
 supported  away from $\vT_{\p}(U(1))$. Then
there exists $C>0$ such that
\begin{equation}\label{eq:LocalDecay}
 \forall \,\psi\in \vH:\quad \int_{-\infty}^\infty
 \norm[\big]{\inp{p}^r\inp{x}^{-s}U(t)g(U(1))\psi }^2\,\d t \leq C\|\psi\|^2. 
\end{equation}
 \end{thm}

The above local decay estimates will be vital for us. For completeness of presentation let us state the following result on  the decay of eigenfunctions.

\begin{thm}(\cite{MS})\label{thm:DecayEigenfct} Suppose
  $V=\Sigma_{a\in\mathcal A}\,V_a$ 
satisfies Condition \ref{cond:smooth2wea3n12}
\ref{item:shortr}--\ref{item:BoundsatInfinity}. Let $\e^{-\i \lambda}\in
\sigma_{\pp}(U(1))$ and $u\in \vH$ satisfy $U(1)u = \e^{-\i \lambda}u$. Then
$u\in\vD(p)$,  and if furthermore $\e^{-\i \lambda}\notin\vT(U(1))$ then for every $\kappa>0$ with
$$
\lambda+\kappa^2 < \inf\left\{E>\lambda : \e^{-\i E}\in\vT(U(1))\right\}
$$
 the function  $\e^{\kappa |x|}u\in \vH$.
\end{thm}

\subsubsection{General 
  notation}\label{subsubsec:General 
  notation}

 Any  function $f\in C^\infty_\c(\R)$ taking values in $[0,1]$ is
referred to as a  
\emph{support function}. Given  such
function  $f$ and a function  $f_1\in C^\infty_\c(\R)$  we write
$f\succ f_1$, if $f=1$ in a neighbourhood of $\supp f_1$. We will use
similar notation $g\succ g_1$ for smooth functions on $\T$.

If $T$ is a self-adjoint  operator on a
Hilbert space and $f$ is a support function we can represent the
operator $f(T)$ by the well-known Helffer--Sj\"ostrand formula
\begin{align}\label{82a0}
  \begin{split}
 f(T) 
&=
\int _{\C}(T -z)^{-1}\,\mathrm d\mu_f(z)\, \text{ with}\\
&\mathrm d\mu_f(z)=\pi^{-1}(\bar\partial\tilde f)(z)\,\mathrm du\mathrm dv;\quad 
z=u+\i v.   
  \end{split}
\end{align} Here $\tilde{f}$ is an `almost analytic' extension of
$f$, which in this case may be taken compactly supported (see also \cite[(3.1.11)]{H1}). The
formula \eqref{82a0} extends to more general  classes of  functions, cf. \cite[Proposition C.2.2]{DG},
and  serves in particular as a standard tool for computing commutators. 
Since  it
will be used rather freely  in this paper  the interested  reader might benefit from
consulting 
 \cite[Section 6]{Sk1},  which  is devoted to various  
applications of \eqref{82a0} to  $N$-body Schr\"odinger operators.

\begin{defns}\label{def:chi}
  \begin{enumerate}[1)]
  \item \label{item:Fplus}
  The class $\vF_+$ denotes the set of non-negative 
  functions $\chi\in C^\infty(\R)$  supported in $\R_+=(0,\infty)$,  with $\sq\chi\in C^\infty(\R)$ 
   and with a non-negative 
   derivative obeying $\sqrt{\chi'} \in C_\c^\infty(\R)$.
 \item \label{chi} Let $\chi_+$ denote any $\chi\in\vF_+$ obeying the following  properties:
\begin{subequations}
\begin{align}\label{eq:chi1}
  \begin{split}
  \chi(t)
=\left\{\begin{array}{ll}
0 &\mbox{ for } t \le 4/3 \\
1 &\mbox{ for } t \ge 5/3
\end{array},
\right.
\quad
\chi'\geq  0,  
  \end{split}
\end{align}
\begin{align}\label{eq:chi2}
  &\q \q\q \sqrt{\chi}, \sqrt{\chi'}, (1-\chi^2)^{1/4} ,
  \sqrt{-\parb{(1-\chi^2)^{1/2} }'}\in C^\infty\\
  &\text{and the square root of  the  function }\R\ni t \to \int_{-\infty}^t\,\chi^2(s)\,\d s \in\R\text{ is smooth}. \label{eq:chi2b}
\end{align}  
   \item \label{chi_minus} Let $\chi_-=(1-\chi_+^2)^{1/2} $.
   
 \item\label{chikK} For arbitrary $\kappa, K\in \R_+$ with  $
2\kappa<K$  the  notation $\chi_{\kappa,K}$ denotes  the function
$\chi_{\kappa,K}(t)=
\chi_+(t/\kappa)\chi_-(t/K)$.
\end{subequations} 
   \end{enumerate}
  \end{defns}
 The functions $\chi_+$ and $\chi_-$ are considered as fixed. Unless otherwise stated, we will henceforth  abbreviate $\chi_+=\chi$. We record that
\begin{align*}
  \chi_+^2+\chi_-^2=1, \q\sqrt{\chi_+},\, \sqrt{\chi_+'}, \sqrt{\chi_-}, \sqrt{-\chi_-'}\in C^\infty
\end{align*} and that the functions  $\R\ni t \to \int_{-\infty}^t\,\chi_{\kappa{,K}}^2(s)\,\d s \in\R$, $0<
2\kappa<K$,  are  in $\vF_+$.
\begin{remarks}\label{remarks:chiSMooTH}
  \begin{enumerate}[i)]
  \item \label{item:invarWavesmooth} We can construct the function
    $\chi=\chi_+$ in \ref{chi}  obeying $\chi(t)=
    (t-4/3)^{-1}\exp\parb{1/ (4/3-t)}$ for all $t$ slightly to the
    right of $4/3$, obviously conforming with parts of the
    requirements of \ref{chi}.
  \item \label{item:invarScatsmooth} For any  $\epsilon>0$  and any given real-valued $h\in
    C^\infty (\R)$ supported in $\R_+$, and with compactly supported derivative,  there exist
    $\chi_1,\chi_2\in \vF_+$ such that
    \begin{equation*}
      \sup_{
      \R}\,|h-\chi_1+\chi_2|\leq \epsilon.
    \end{equation*}
This  approximation result follows  easily by  first 
    decomposing  the derivative  of $h$ in terms of its positive and
    negative parts, and  second by smoothing out the square roots of the
    latter functions by the standard convolution technique. Proper anti-derivatives of the smooth squares then represent the desired  functions $\chi_1$ and $\chi_2$.
\end{enumerate}
\end{remarks}

If $T$ is a self-adjoint  operator on a
Hilbert space $\vG$ and $\varphi\in \vG$, then
$\inp{T}_\varphi:=\inp{\varphi,T\varphi}$ and $F(T<C):=1_{(-\infty,
  C)}(T)$ (and similarly for $F(T>C)$ and $F(-C<T<C)$). If $S$ is another
operator on  $\vG$ we introduce (formally) the iterated commutators $\ad^k _{T}(S)$ for  $k\in \N_0
=\N\cup\set{0}$ as follows. Let $\ad^0 _{T}(S)=S$ and 
(recursively) $\ad^k _{T}(S)=\big [\ad^{k-1} _{T}(S),T \big ]$
for 
$k\in \N$. Then   we may  formally  use \eqref{82a0}  expanding 
\begin{equation*}
  f(T)S=\sum ^\infty_{k=0}\, \tfrac{(-1)^k} {k!}\ad ^k_{T}(S)
  f^{(k)}(T) .
\end{equation*} We will  concretely use truncations of this formula
 with errors controlled in terms of  the concept of order introduced below.

We denote the space of  bounded operators 
from one  (abstract) Banach space $X$ to another one $Y$ by $\vL(X,Y)$ 
and abbreviate $\mathcal L(X,X)=\mathcal L(X)$. The dual space of $X$
is denoted by $X^*$.

Recall  the standard \emph{weighted $L^2$ spaces}  
$$
L_s^2=L_s^2(\mathbf X)=\inp{x}^{-s}L^2(\mathbf X)\ \ \text{for }s\in\mathbb R ,\quad
L_{-\infty}^2=\bigcup_{s\in\R}L^2_s,\quad
L^2_\infty=\bigcap_{s\in\mathbb R}L_s^2.
$$

Let $S$ be any  operator on $\mathcal H=L^2(\mathbf X)$ such that
$S,S^*:L^2_\infty\to L^2_\infty$, and let $\sigma\in\mathbb R$.  Then we
say that   $S$ is  {\emph{of order $\sigma$}}, if 
 for each $s\in\mathbb R$  the restriction  $S_{|L^2_\infty}$ extends to
 an operator $S_s\in\vL(L^2_{s}, L^2_{s-\sigma})$. Alternatively stated,
 \begin{subequations}
  \begin{equation}\label{eq:defOrder}
\|\inp{x}^{s-\sigma}S\inp{x}^{-s}\psi\|\le C_s\|\psi\| \text{ for all }\psi\in
L^2_\infty.
\end{equation} 
 If   $S$ is of {order $\sigma$}, we write 
\begin{equation}
  S=\vO(\inp{x}^\sigma)=\vO(r^\sigma);\quad r=\inp{x}.
\label{eq:1712022}
\end{equation}  The class of ordered operators forms  a  graded
algebra. In particular, if $S$ and   $T$  are  of {order $\sigma$} and $\tau$, respectively,
then $ST$ is of {order $\sigma+\tau$}. For any given
$\sigma\in\mathbb R$  the notation $S=\vO(r^{-\sigma_+})$ means that
$S=\vO(r^{-\br\sigma})$ for some $\br \sigma>\sigma$.
\end{subequations}  If $(S_\theta)_{\theta\in \Theta}$ is a family of operators of
order $\sigma$  for which \eqref{eq:defOrder} holds uniformly in
$\theta\in \Theta$ for all  $s\in\mathbb R$,  we say that $S_\theta$ is $\theta$-\emph{uniformly of
  order $\sigma$}, written $S=\vO_{\unif}(r^\sigma)$.
 For any given
$\sigma\in\mathbb R$   we write
$S=\vO_{\unif}(r^{-\sigma_+})$ if $S=\vO_{\unif}(r^{-\br\sigma})$ for
some $\br \sigma>\sigma$. We write
$S=\vO(r^{-\infty })$, if $S=\vO(r^\sigma)$  for all $\sigma\in \R$, and
similarly for the uniform setting,   $S=\vO_{\unif}(r^{-\infty})$ means
that $S=\vO_{\unif}(r^\sigma)$  for all $\sigma \in \R$.

We use  the convention
  $S=\vO(t^\sigma)$ for a given 
  family of 
   operators  $S(t)\in \vL(\vH)$, $t\geq 1$,  to mean that
  the  operator norm of $S(t)$  is
   bounded as $\norm{S(t)}\leq C t^\sigma$. The notations 
  $S=\vO(t^{-\sigma_+})$ and  $S=\vO(t^{-\infty
})$ signify that 
$S=\vO(t^{-\br\sigma})$ for some $\br \sigma>\sigma$ and that
$S=\vO(t^\sigma)$  for all $\sigma\in \R$, respectively.

\subsection{Results}\label{subsec:Main result}

For any  \emph{channel}
$\alpha =(a,\lambda^\alpha, u^\alpha)$, i.e.   $a\in\vA_1$, $\lambda^\alpha\in [0,2\pi)$ and  
$U^a(1)u^\alpha = \e^{-\i \lambda^\alpha }u^\alpha$ for a normalized $u^\alpha\in
\mathcal H^a$, we introduce the corresponding \emph{channel wave
  operators} $W_\alpha^{\pm}$ as the strong limits
\begin{equation}\label{eq:wave_op}
  W_\alpha^{\pm}\psi_a=\lim_{t\to \pm\infty}U(0,t)\parbb{\parb{\e^{-\i \lambda^\alpha [t]}U^a(t- [t])u^\alpha} \otimes
\parb{\e^{-\i tp_a^2} \psi_a}};\quad \psi_a\in L^2(\bX_a).
 \end{equation} Here     $ [t]\in [0,1)$ denote  the integer
 part of  $t\in \R$, and we recall that $p_a=-\i \nabla_{x_a}$. Note
 also (cf.  \eqref{eq:groupLike}) that
 \begin{equation*}
 \e^{-\i \lambda^\alpha [t]} U^a(t- [t])u^\alpha=U^a(t)u^\alpha, 
 \end{equation*}
and hence (since $U_a(t)=U^a(t)\otimes \e^{-\i tp_a^2}$) that 
\begin{equation*}
  W_\alpha^{\pm}\psi_a=\lim_{t\to \pm\infty}U(0,t)U_a(t)\parb{u^\alpha \otimes
\psi_a};\quad \psi_a\in L^2(\bX_a).
 \end{equation*} 
For the channel $\alpha_{\min}:=(a_{\min}, 0,1)$ the corresponding  wave
  operators simplify as 
  \begin{equation*}
     W_{\alpha_{\min}}^{\pm} \psi=\lim_{t\to \pm\infty} U(0,t)\e^{-\i
       tp^2} \psi;\quad \psi\in \vH.
  \end{equation*}

A basic   result of this  paper is the following version  of 
\cite[Theorems XI.35 and  XI.36]{RS}, which concerns  the
time-independent case. Our proof of the existence of the channel  wave
operators
does not use the Cook--Kuroda argument (which is used in \cite{RS}), it 
does not  need decay assumptions on channel  eigenfunctions (possibly not being known  from Theorem
 \ref{thm:DecayEigenfct})   and it works
for `pair-potentials' of any dimension.  (As far as we know the proof of
\cite[Theorems XI.34 and  XI.35]{RS} requires   three or higher dimensional pair-potentials.)
\begin{thm}\label{thm:ACn-body-short1} Under Condition
  \ref{cond:smooth2wea3n12} all of the channel wave operators
  $W_\alpha^\pm$  exist. For any two  different channels  $\alpha\neq \beta$,
  it holds that $ \ran(W_\alpha^\pm)\perp \ran (W_\beta^\pm)$ (i.e. orthogonality
  of   the forward channels as well as of   the backward channels hold).
 \end{thm}
Thanks to   \eqref{eq:groupLike} the   intertwining property for our 
$1$-periodic model reads (see \cite[Theorem  XI.36]{RS} for the
analogous version  for the
time-independent case)
\begin{equation}\label{eq:intertwin}
  U(1)W_\alpha^{\pm}=W_\alpha^{\pm}\e^{-\i k_\alpha};\quad
  k_\alpha:=p_a^2+\lambda^\alpha.
\end{equation}

We will focus on studying the forward channel wave operators $W_\alpha^+$ (i.e. scattering as $t\to
+\infty$). Of course there are similar results for the backward channel wave operators $W_\alpha^-$.
\begin{defns}\label{defn:scattering subspace} 
  \begin{enumerate}[1)]
  \item \label{item:ppSubspace} The subspace 
    $\vH_{\rm ac}\subseteq \vH$ denotes the absolutely continuous subspace of $U(1)$. (It is given as the orthogonal complement of the set
    of bound states of $U(1)$.)
\item \label{item:energySubspace} The  finite energy   subspace
    $\vH^+_{\ener}\subseteq \vH$ consists of the vectors  $\psi\in \vH_{\rm ac}$ for
    which the following bound  holds for $\psi(t)=U(t)\psi$.
    \begin{align}\label{eq:BndEnergy} \quad \quad \lim_{E\to \infty}\,
      \limsup_{t\to \infty}\, \norm{F(p^2> E)\psi(t)}=0.
    \end{align}
  \item \label{item:excSubspace} The {increasing energy subspace} $\vH^+_{\rm ie}$ consists of the vectors  $\psi\in \vH_{\rm ac}$ for
    which the following bounds  hold for  $\psi(t)=U(t)\psi$.
 \begin{equation}\label{eq:BndEnergyEXC}
     \forall E>0:\q \lim_{t\to \infty}\, \norm{F(p^2< E)\psi(t)}=0.
   \end{equation}
  \item \label{item:WaveOper} The   wave operator  subspace 
    $\vH^+_{\wave}\subseteq \vH$ is given as the direct sum  of the ranges of
    the   forward channel wave operators, i.e.  
    \begin{equation*}
      \vH^+_{\wave}=\sum_\alpha^\oplus \, \ran(W_\alpha^+).
    \end{equation*}
\end{enumerate}
\end{defns}

The above spaces are  defined for the family $H=H(\cdot)$ and its dynamics $U(\cdot)$. To emphasize this dependence we
may write $\vH_{\rm ac}=\vH_{\rm ac}(H)$,
$\vH^+_{\ener}=\vH^+_{\ener}(H)$, $\vH^+_{\rm ie}=\vH^+_{\rm ie}(H)$  and $\vH^+_{\wave}=\vH^+_{\wave}(H)$. The analogous spaces  for $a\in \vA_3$, which amounts to replacing $H$
 by $H^a$ and the dynamics $U$ by $U^a$,  are 
denoted by $\vH_{\rm ac}(H^a)$,
$\vH^+_{\ener}(H^a)$, $\vH^+_{\rm ie}(H^a)$ and $\vH^+_{\wave}(H^a)$.

\begin{remarks}\label{remarks:n-body-short}
  \begin{enumerate}[i)]
  \item \label{item:invarScat} It is obvious from 
    \eqref{eq:groupLike} that  $\vH^+_{\ener}$ and $\vH^+_{\rm ie}$ 
    are  invariant invariant subspaces  of   the mo\-nodromy operator   
    $U(1)$ and its inverse $U(0,1)$.
    \item \label{item:invarWave} It follows from \eqref{eq:intertwin}
    that also $\vH^+_{\wave}$ is invariant for 
    $U(1)$ and its inverse. It also follows that $\vH^+_{\wave}\subseteq \vH_{\rm ac}$.
  \item \label{item:Inclusions} It is easy to see using  \ref{item:invarWave} 
    and Yosida's theorem that $\vH^+_{\wave}\subseteq\vH^+_{\ener}$, and obviously $\vH^+_{\ener }$ and $\vH^+_{\rm ie}$ are othogonal subspaces. Hence 
    \begin{equation}\label{eq:Inclusions}
      \vH^+_{\wave}\subseteq\vH^+_{\ener }\subseteq \vH_{\rm ac} \ominus \vH^+_{\rm ie}\subseteq \vH_{\rm ac}.
    \end{equation}
\item \label{item:AsCom} Asymptotic completeness is  the
  assertion that $\vH^+_{\wave}=\vH_{\rm ac}$.
\end{enumerate}
\end{remarks}

For $N=2$  asymptotic completeness it well-known, cf. \cite{Yaj1}. We provide in this case a fast and  purely time-dependent  proof of this  property.
\begin{thm}\label{thm:ACn-body-short2} 
  Asymptotic completeness is valid for the  two-body model, i.e. it
    holds  that  $\vH^+_{\wave}(H)=\vH_{\rm ac}(H)$  for $N=2$.
\end{thm}
 Our main result is for the  general many-body model. It reads as follows.

\begin{thm}\label{thm:ACn-body-short3} For
  arbitrary  $N $
  \begin{equation}\label{eq:mainR}
    \vH^+_{\wave}(H)=\vH^+_{\ener}(H).
  \end{equation}
\end{thm}
One can view  (\ref{eq:mainR}) as  a geometric  characterization of the wave operator subspace.
Moreover it follows from (\ref{eq:mainR}) that asymptotic completeness for the
 many-body model is equivalent to the property that  $\vH^+_{\ener }=\vH_{\ac}$, which consequently is a  criterion  for completeness. In this paper
 we do not show the applicability of the criterion (nor prove asymptotic completeness) for $N\geq  3$.  We remark that Nakamura \cite{Na} showed completeness for  $N= 3$ under   strong conditions (including   exponential decay of pair-potentials). Recently Yajima \cite{Yaj2} showed it for arbitrary $N$,  however again   under  strong conditions (including a smallness condition on the pair-potentials).
  
 In
    Proposition \ref{prop:asympt-compl-anoth} we give a different
    criterion  for asymptotic completeness  than $\vH^+_{\ener }=\vH_{\ac}$. This is in terms of a notion of 
    {`asymptotic clustering'}, which actually is also a key ingredient in our proof of the above more geometrically flavoured  theorem.

    For  $N  =3$ yet another criterion  for  asymptotic completeness  is $\vH^+_{\rm ie}=\set{0}$ thanks to the following strenghtening of (\ref{eq:mainR}).
\begin{proposition}\label{thm:ACn-body-short3EXC} For  $N  =3$  
  \begin{equation}\label{eq:strength}
    \vH_{\rm ac}=  \vH^+_{\ener} \oplus  \vH^+_{\rm ie} =\vH^+_{\wave} \oplus  \vH^+_{\rm ie} .
  \end{equation} 
\end{proposition}

In general this geometric criterion reads.
 \begin{proposition}\label{prop:asympt-compl-anoth9} For
  arbitrary  $N $  the following assertions  $i)$ and $ii)$ are equivalent.
  \begin{enumerate}[i)]
  \item\label{item:1asymp49} $\forall  a\in \vA_2$: \q $\vH^+_{\rm ie }(H^a)=\set{0}$. 
     \item \label{item:2equi59} $\forall  a\in \vA_2$: \q $\vH^+_{\wave }(H^a)=\vH_{\ac}(H^a)$. 
  \end{enumerate}
     \end{proposition}

We record the following pointwise  decay result (for 
  arbitrary  $N $). 
 \begin{lemma}\label{lemma:scatDef} For any  $\psi\in \vH^+_{\ener}$ and
   any $R>0$
   \begin{equation}
\label{eq:decayInx}
   \lim_{t\to \infty}\, \norm{F(|x|< R)\psi(t)}=0;\q \psi(t)=U(t)\psi.
   \end{equation} 
 \end{lemma}
 \begin{proof}  Consider  $\phi(t)=\norm
   {\inp{p}^{-1}r^{-1}\psi(t)}^2$, where $\psi=g(U(1))\psi$ for some
 $g\in C^\infty_{\c}(\T\setminus \vT_\p)$. By density  and  commutation it suffices to show that
 $\phi(t)\to 0$ for $t\to \infty$. Thanks to \eqref{eq:LocalDecay} the
 function $\phi$
 is integrable at infinity. Since furthermore it  has a
 bounded derivative,  indeed $\lim_{t\to \infty}\, \phi(t)=0$.
\end{proof}

 With a considerably more complicated proof (given in Section \ref{sec:Preliminary localization results})  this pointwise  decay result is shown to be valid for any $\psi\in \vH_{\ac}$.  The further generalization,  stated below as \eqref{eq:decayInx22} for $\lambda<1$,  is sharp in the sense that  it is incorrect for $\lambda=1$. It may be referred to as a \emph{minimal velocity bound}, which is consistent with the use of  this terminology  in   time-independent $N$-body 
scattering theory. 
 
 \begin{proposition}\label{lemma:scatDef22}
    For any  $\psi\in \vH_{\ac}$ and
    any $R>0$
    \begin{subequations}
      \begin{equation}
\label{eq:decayInx2}
   \lim_{t\to \infty}\, \norm{F(|x|< R)\psi(t)}=0;\q \psi(t)=U(t)\psi.
 \end{equation}

 More generally, for any  $\psi\in \vH_{\ac}$ and any $\lambda<1$
 \begin{equation}
\label{eq:decayInx22}
   \lim_{t\to \infty}\, \norm{F(|x|< t^\lambda)\psi(t)}=0;\q \psi(t)=U(t)\psi.
 \end{equation}\end{subequations}
 \end{proposition}
  
\section{Yafaev's  constructions}\label{sec: Yafaev type
  constructions}

In \cite{Ya1} various  real functions $m_a\in C^\infty(\mathbf X)$,
$a\in\vA_1$,  and $m_{a_{0}}\in C^\infty(\mathbf X)$  are 
constructed. They  are  homogeneous of degree $1$ for $\abs{x}> 5/6$,
and in addition $ m_{a_0}$ is (locally) convex in that region. These
functions are constructed as depending  of a  small parameter
$\epsilon>0$, and once given, they  are  used in the constructions
\begin{equation}\label{eq:M_a0}
 M_a=2\Re(w_a\cdot p)=-\i\sum_{j\leq d}\parb{(w_a)_j\partial_{x_j}+\partial_{x_j}(w_a)_j};\quad  w_a=\mathop{\mathrm{grad}} m_a.   
\end{equation}
The operators $M_a$, $a\in\vA_1$, may  be considered as `channel
localization operators', while  operators of the form  $M_{a_0}$  partially enter as  auxiliary  quantities  controlling  commutators of  
the Hamiltonian and the channel
localization operators. (For a more preliminary application, see \eqref{eq:DECO}.)  This  scheme is useful for  proving  asymptotic completeness  for the time-independent case, as demonstrated in \cite{Ya1}.

We  consider various conical
subsets of  $\mathbf X\setminus\set{0}$, defined for 
  $a\in\vA_1$ and $\varepsilon,\delta\in (0,1)$,
\begin{align}\label{eq:primes}
  \begin{split}
    \mathbf X'_a&={\mathbf X_a }\setminus\cup_{{b\not\leq a,\,b\in
                     \vA}\,
                 }\mathbf X_b,\\
\mathbf X_a(\varepsilon)&=\set{x\in
  \bX\mid\abs{x_a}>(1-\varepsilon)\abs{x}},\\\quad
&\quad\quad\overline{\mathbf X_a(\varepsilon)}=\set{x\in \bX\setminus\set{0}\mid\abs{x_a}\geq(1-\varepsilon)\abs{x}},\\
\mathbf Y_a(\varepsilon)&=\parb{\mathbf X \setminus\set{0}}\setminus\cup_{{b\not\leq a,\,b\in \vA_1}
}\,\mathbf X_b(\varepsilon),\\
\mathbf Z_a(\delta)&=\mathbf X_a(\delta)\setminus\cup_{b\not\leq a,\,b\in \vA_1}
\,\overline{\mathbf X_b(3\delta^{1/d_a})};\quad\quad d_a=\dim
\mathbf X_a.
\end{split}
\end{align}
The structure of the sets $\mathbf X_a(\varepsilon)$, $\mathbf
Y_a(\varepsilon)$ and $\mathbf Z_a(\delta)$ is
${\R_+ V}$, where  $V$ is a
subset of the unit sphere $\mathbf{S}^{a_0}\subseteq\mathbf X$. For $\mathbf
X_a(\varepsilon)$ and $\mathbf Z_a(\delta)$ the set $V$ 
 is relatively open, while for  $\mathbf Y_a(\varepsilon)$ the  set is  compact.
We also record that $\mathbf X_a\setminus\set{0}\subseteq\mathbf
  X_a(\varepsilon)\subseteq \set{\abs{x^{a}}< \sqrt{2\ \varepsilon}\abs{x}}$,
and that $\mathbf Z_a(\delta)\subseteq \mathbf X_a(\delta)\cap\mathbf
Y_a(3\delta^{1/d_a})$. The following inclusion is less obvious (see
\cite[Lemma 3.10]{Sk2}):
\begin{equation}\label{eq:cover}
  \forall a\in\vA_1\, \forall \varepsilon\in (0,1):\quad \mathbf Y_a(\varepsilon)\subseteq\cup_{b\leq a}  \,\mathbf X'_b.
\end{equation}

The Yafaev constructions involve families of  cones $\mathbf X_a(\varepsilon)$ roughly with    `width'
$\varepsilon\approx\epsilon^{d_a}$ for a 
parameter $\epsilon>0$. This parameter is taken sufficiently small   as primarily 
determined  by the following
geometric property: 
There exists $C>0$ such that for all $a,b \in \vA_1$ and $x\in\mathbf X$
\begin{equation}\label{eq:3.2}
  \abs{x^c}\leq C\parb{\abs{x^a}+\abs{x^b}};\quad c=a\vee b.
  \end{equation}
We note that \eqref{eq:3.2} is  also vital in Graf's constructions
\cite{Gra, De}. For a self-contained account of Yafaev's constructions recalled in this section we refer to \cite{Sk2}.

\subsection{Yafaev's  homogeneous functions}\label{subsec: Homogeneous Yafaev type functions} 

For small $\epsilon>0$ and any $a\in \vA_1$ we define
\begin{equation*}
  \varepsilon^a_k= k\epsilon^{d_a},\quad k=1,2,3,
\end{equation*} and view the numbers  $\varepsilon_a$ in the
interval 
$(\varepsilon^a_2, \varepsilon^a_3)$  as 
`admissible'. More precisely we view $\varepsilon_a$  as a free  real parameter constrainted by 
\begin{equation}\label{eq:admis}
  \varepsilon^a_2=2\epsilon^{d_a}<\varepsilon_a<3\epsilon^{d_a}=\varepsilon^a_3.
\end{equation} Fixing an  arbitrary
ordering of $\vA_1$ we then introduce in terms of  these  parameters  the `admissible' vector 
\begin{equation*}
  \bar\varepsilon=(\varepsilon_{a_1}, \dots
\varepsilon_{a_n}),\quad n= \# \vA_1,
\end{equation*} and denote by $\d \bar\varepsilon$ the corresponding
Lebesgue measure. 

Let for $\varepsilon>0$ and $a\in \vA_1$
\begin{equation*}
  h_{a,\varepsilon}(x)=(1+\varepsilon)\abs{x_a};\quad
  x\in \mathbf X\setminus \set{0}.
\end{equation*} 
 Letting  $\Theta= 1_{[0,\infty)}$,  we define for  any $a\in \vA_1$ and any admissible vector $\bar\varepsilon$
  \begin{equation*}
    m_a(x, \bar \varepsilon)= h_{a,\varepsilon_a}(x)\Theta\parb{
      h_{a,\varepsilon_a}(x)-\max_{\vA_1 \ni c\neq a } h_{c,\varepsilon_c}(x)};\quad
  x\in\mathbf X\setminus \set{0}.
  \end{equation*} 
For each   $c\in \vA_1$ we fix an arbitrary 
non-negative function $\varphi_c\in C^\infty_\c(\R)$ with
\begin{equation*}
  \supp \varphi_c\subseteq (\varepsilon^c_2,\varepsilon^c_3)
\mand \int_\R \varphi_c(\varepsilon)\,\d \varepsilon=1,
\end{equation*} and then   average  the functions $ m_a(x,
\bar \varepsilon)$,     $a\in \vA_1$, over $\bar
\varepsilon$ as    
\begin{equation*}
  m_a(x)=\int_{\R^n} m_a(x, \bar \varepsilon)\,\prod_{c\in \vA_1}\, \varphi_c(\varepsilon_c)\,\d \bar\varepsilon;\quad x\in\mathbf X\setminus \set{0}.
\end{equation*} 
\begin{lemma}\label{lemma:ma1}
For each  $a\in\vA_1$ the  function 
 $m_a:\mathbf X\setminus \set{0}\to [0,\infty)$ fulfils  the following
properties for any sufficiently small  $\epsilon>0$ and any $b\in\vA_1$: 
\begin{enumerate}[1)]
\item\label{item:10a} $m_a$ is homogeneous of degree $1$.
\item\label{item:11a} $m_a\in C^\infty (\bX\setminus\set{0})$.
\item\label{item:12a} If $b\leq a$  and  $x\in \mathbf X_b(\varepsilon^b_1)$, then
  $m_a(x)=m_a(x_b)$.
\item\label{item:13a} If ${b\not\leq a}$ and $x\in \mathbf
  X_b(\varepsilon^b_1)$, then   $m_a(x)=0$.
\item\label{item:14a} If $a\neq {a_{\min}}$,  $x\neq 0$  and $\abs{x^{a}}\geq \sqrt{2\ \varepsilon^a_3}\abs{x}$, then   $m_a(x)=0$.
\end{enumerate} 
\end{lemma}

Following \cite{Ya1} we introduce for  $ x\in\mathbf X\setminus \set{0}$
and admissible vectors $\bar\varepsilon$
\begin{equation*}
   m_{a_0}(x,\bar\varepsilon)=\max_{a\in \vA_1} h_{a,\varepsilon_a}(x),  
\end{equation*} and let then
\begin{equation*}
  m_{a_0}(x)=\int_{\R^n} m_{a_0}(x, \bar \varepsilon)\,\prod_{c\in
    \vA_1}\, \varphi_c(\varepsilon_c)\,\d \bar\varepsilon.
\end{equation*}
\begin{lemma}\label{lemma:m1} The  function 
 $m_{a_0}:\mathbf X\setminus \set{0}\to (0,\infty)$ fulfils  the following
properties for any sufficiently small  $\epsilon>0$: 
\begin{enumerate}[i)]
\item\label{item:10b} $m_{a_0}$ is convex and homogeneous of degree $1$.
\item\label{item:11b} $m_{a_0}\in C^\infty (\bX\setminus \set{0})$.
\item\label{item:12b} If $b\in\vA_1$  and  $x\in \mathbf X_b(\varepsilon^b_1)$, then
  $m_{a_0}(x)=m_{a_0}(x_b)$.
\item\label{item:9b} $m_{a_0}=\sum_{a\in \vA_1}\,m_a$  (with all functions defined for the same $\epsilon$).
\item\label{item:13b} Let $a\in \vA_1$ and suppose $x\in\mathbf
  X_a(\varepsilon^a_1)\setminus \cup_{b\gneq
  a}\, \mathbf
  X_b(\varepsilon^b_3)$.  Then 
  \begin{equation*}
    m_{a_0}(x)=m_{a}(x)=C_a
    \abs{x_a};\quad C_a=\int_\R (1+\varepsilon)\,\varphi_a(\varepsilon)\,\d \varepsilon.
  \end{equation*}
\item\label{item:15b} The derivative of $m_{a_0}\in C^\infty
  (\bX\setminus \set{0})$ obeys 
\begin{align*}
   \quad \quad\nabla \parb{m_{a_0}(x)-\abs{x}}&\\=\sum_{a\in \vA_1}\,\int
   \,
 \d \varepsilon_a&
   \nabla \parb{{h_{a,\varepsilon_a}(x)}-\abs{x}}\,\varphi_a(\varepsilon_a)\\  &\prod_{\vA_1\ni b\neq a    }\,\parbb{\int\,\Theta\parb{ h_{a,\varepsilon_a}(x)-h_{b,\varepsilon_b}(x)}\varphi_b(\varepsilon_b)\, \d\varepsilon_b}. 
\end{align*} In particular there exists $C>0$ such that for all $x\in
\bX\setminus \set{0}$ 
\begin{equation*}
  \abs{\nabla \parb{m_{a_0}(x)-\abs{x}}}\leq C\sqrt{\epsilon}.
\end{equation*}
\end{enumerate} 
\end{lemma}

\subsection{Yafaev's  observables and some  commutator  properties}\label{subsec: Geometric considerations for a0=a{max}}
 We introduce smooth modifications of the functions $m_a$ from Lemma \ref{lemma:ma1}
and $m_{a_0}$ from Lemma \ref{lemma:m1}  by
multiplying them by a suitable factor, say specifically by the factor
$\chi_+(2|x|)$. We adapt these modifications and will (slightly
abusively) use  the same notation $m_a$ and $m_{a_0}$ for the smoothed out
versions.  We may then consider the corresponding first order
operators $M_a$ and $M_{a_0}$ from \eqref{eq:M_a0} as realized as self-adjoint
operators. 

In  our applications    the above Yafaev  observables  will be used to construct a suitable `phase-space 
 partition of unity'. In the present subsection we provide preliminary  
 computational  details on how to control their commutator with the
Hamiltonian $H$. Although not being treated explicitly, similar
arguments work for the  Hamiltonians $H_a$, $a\in\vA_1$.

 Thanks to   \eqref{eq:cover}  we can for  any $a\in\vA_1$ and  any  $\varepsilon, \delta_0\in (0,1)$  write
  \begin{equation*}
    \mathbf Y_a(\varepsilon)\subseteq\cup_{b\leq a} \cup_{\delta\in (0,\delta_0]} \,\mathbf Z_b(\delta).
  \end{equation*} 
  By compactness this   leads for
  any fixed $\varepsilon,\delta_0\in(0,1)$ to the existence of a finite
  covering
\begin{equation} \label{eq:deltaNBHb}
     \mathbf Y_a(\varepsilon)\subseteq\cup_{j\leq J} \,\,\mathbf Z_{b_j}(\delta_j),
  \end{equation} where  $\delta_1,\dots, \delta_J\in
(0,\delta_0]$ and  $b_1,\dots, b_J\leq a$; $J=J(a)$. Here and in the
following we prefer to  suppress
 the dependence of quantities on the given $a\in\vA_1$.

 We recall   that  the
functions $m_a$, $a\in \vA_1$, depend on a
positive parameter $\epsilon$ and  by Lemma  \ref{lemma:ma1}
\ref{item:13a}  fulfil 
\begin{equation}\label{eq:suppa}
  \supp m_a\subseteq 
\mathbf Y_a(\epsilon^{d});\quad d=\dim
\mathbf X.
\end{equation}
    Having  Lemmas \ref{lemma:ma1} and \ref{lemma:m1} at our
    disposal we also observe   that the (small) positive  parameter
    $\epsilon$ appears independently for  the two lemmas.  First we
    choose and consider it fixed by  the \emph{same}   small $\epsilon$ for the  lemmas,
    yielding in this way  the whole family  $(M_a)_{a\in \vA}$
    uniquely defined. Thanks to Lemma
\ref{lemma:m1} \ref{item:9b}  we can  then record the  identity  
     $M_{a_0}=\sum_{a\in \vA_1}\,M_a$,  which will  be  vital   in
     Subsection \ref{subsec:A prior  localization}. Prior to the implementation of the  identity we will demonstrate localization to the outgoing region, stated in terms of  $M_{a_0}$ (slightly modified). We devote most of Sections \ref{sec:Preliminary localization results} and \ref{sec:Localization to  outgoing  region} to this problem, while Section \ref{sec:Proof of {eq:Strength}} concerns related propagation estimates.
\subsubsection{Representation  of the commutators $\i[
p^2,M_a]$}\label{subsubsec: Controlling commutators}
We need addditional   applications of Lemma \ref{lemma:m1}, more
precisely for certain inputs $\epsilon_1,\dots,\epsilon_J\leq
\epsilon$ (with $J=J(a)$, 
$a\in\vA_1$)
     defined  as follows. 
We apply \eqref{eq:deltaNBHb} with
    $\varepsilon=\delta_0=\epsilon^{d}$
    (cf. \eqref{eq:suppa}). Since $\delta_j\leq \delta_0$, we can
    introduce positive $\epsilon_1,\dots,\epsilon_J\leq \epsilon$ by  the
    requirement $\epsilon_j^{d_{b_j}}=\delta_j$. The 
    new inputs $\epsilon=\epsilon_j$ in Lemma \ref{lemma:m1}  yield
    corresponding   functions, say denoted   $m_j$. In
    particular in the region $\bZ_{b_j}(\delta_j)$ the function $m_a$
    from Lemma \ref{lemma:ma1} only
    depends on $x_{b_j}$ (thanks to   Lemma
    \ref{lemma:ma1} \ref{item:12a} and the  property
    $\bX_{b_j}(\delta_j)\subseteq \bX_{b_j}(\epsilon^{d_{b_j}})$),
    while (thanks to Lemma \ref{lemma:m1} \ref{item:13b})
    \begin{equation}\label{eq:good0}
      m_j(x)=C_j\abs{x_{b_j}},\quad C_j=\int_\R (1+\varepsilon)\,\varphi_{b_j}(\varepsilon)\,\d
      \varepsilon.
    \end{equation}
    
Next  we introduce   
the vector-valued first order operators 
\begin{equation}\label{eq:Ffield}
  G_b=\chi_+(2\abs{x_b})\abs{x_b}^{-1/2} P(x_b)\cdot
  p_b,\text{ where } P(x)=I-\abs{x}^{-2}\ket{ x}\bra{ x}.
\end{equation}   (Alternatively introduced,   $P(x_b)$  projects orthogonally to $x_b$ in  $\bX_b$.)
 We  choose for the considered   $a\in\vA_1$ 
   a quadratic partition $\hat\xi_1,\dots, \hat\xi_J\in
C^\infty(\mathbf{S}^{a_0})$ (viz $\sum_j \,{\hat\xi_j}^2=1$)
subordinate to the covering \eqref{eq:deltaNBHb} (recalling the discussion before
    \eqref{eq:cover}). Then
we can write, using  the support properties   \eqref{eq:deltaNBHb} and \eqref{eq:suppa}, 
\begin{align*}
  m_a(x)=\sum_{j\leq J}\,\,m_{a,j}(x);\quad m_{a,j}(x)=\hat\xi^2_j( \hat x) m_a(x),\quad \hat x=x/\abs{x},
\end{align*} and from the previous discussion it follows  that 
\begin{align*}
  \chi^2_+(\abs{x})m_{a,j}(x)=\xi^2_j
  m_{a}(x_{b_j}),\quad \xi_j=\xi_j(x):=\hat\xi_j( \hat x)\chi_+(\abs{x}),
\end{align*} as well as 
\begin{align}\label{eq:Hes0}
  \begin{split}
 p\cdot\parb{\chi^2_+(\abs{x})\nabla^2m_a(x)} p&=\sum_{j\leq J}\,\,
                                 p\cdot\parb{\xi^2_j\nabla^2m_a(x_{b_j})}p\\&
=\sum_{j\leq J}\, \,G^*_{b_j}\parb{\xi^2_j\vG_j}G_{b_j};\quad 
                                                                  \vG_j=\vG_j(x_{b_j})\text{ 
                                                                  bounded}.   
  \end{split}
\end{align}

Used with appropriate momentum-localization,  the left-hand side and 
$\tfrac 14 \i[
p^2,M_a]$ coincide
up to order $-3$. However such
momentum-localization is not  provided by energy-localization (since our Hamiltonian is  time-depending),  which
will cause a complication in our applications.

By the convexity property of $m_j$ and \eqref{eq:good0} we can record the estimate  
\begin{align}\label{eq:Hes_est}
   Q(a,j)^*Q(a,j)\leq   p\cdot \parb{\chi^2_+(\abs{x})\nabla^2m_j(x)} p;\q Q(a,j):=\xi_jG_{b_j}.
\end{align} 
Here the right-hand side  should be  the `leading term' of  the commutator
$\tfrac 14 \i[
p^2,M_j]$, where $M_j$ is given by \eqref{eq:M_a0} for  the
modification of 
$m_j$ given by $\chi_+(2|x|)m_j(x)$. The estimate  will be  applied in Subsection \ref{subsec:Radiation condition integral estimates} to derive certain radiation condition bounds (although using there a slightly different regularization), which in turn will be vital for proving the asymptotic clustering result Proposition \ref{prop:negB}.

\section{Preliminary localization results}\label{sec:Preliminary localization results}

We consider the operator $M=M_{a_0}$ from
  Section \ref{sec: Yafaev type
  constructions} (given in terms of a fixed small $\epsilon>0$). It is convenient here and henceforth to  modify and  abbreviate the underlying  function
$m_{a_0}$  of  its  construction \eqref{eq:M_a0}. In the following
definition we use the strict homogeneous function $m_{a_0}$  from Lemma
\ref{lemma:m1}, introducing the  smooth modification 
\begin{subequations}
\begin{equation}\label{eq:ml}
  m(x):= f_{\rm cvx}(m_{a_0}(x)),
\end{equation} where   $f_{\rm cvx}\in C^\infty(\R_+)$ is taken  convex such that $f_{\rm cvx}(s)=1$ for $s\leq  1/4$ and $f_{\rm cvx}(s)=s+2/3$ for $s\geq
1/2$. Note that $m(x)\geq 1$  is a smooth convex function on $\bX$,  fixed as $m(x)=m_{a_0}(x)+2/3$ for $|x|\geq 1/2$.

Hence, more precisely and consistently with  \eqref{eq:M_a0},  we shall henceforth reserve the notation $M$  to mean the self-adjoint operator
\begin{equation}\label{eq:M0l}
  M=M_0:= 2\Re (\grad m\cdot p).
\end{equation} 
\end{subequations}
  Since  $m_{a_0}$ depends  on a small parameter
  $\epsilon>0$, also  the above quantities $m$ and $M$ do. However we suppress this dependence   for example to  avoid   below cumbersome notation like  $m_{\epsilon,t}$ or $(m_\epsilon)_t$, preferring  the simplified notation $m_t$.

   With  $\lambda\in (0,1)$ (soon to be taken  slightly smaller than $1$)   we then introduce 
\begin{equation}
  \label{eq:mtime}
  m_t(x)=t^{\lambda}m\parb{x/t^{\lambda}}\mand M_t=
  2\Re \parbb{\parb{\nabla  m}\parb{x/t^{\lambda}}\cdot p};\q t\geq1.
\end{equation} Clearly the quantities of (\ref{eq:mtime})  for $t=1$ agree with (\ref{eq:ml}) and (\ref{eq:M0l}) (and hence $M_0=M_1=(M_t)_{|t=1}$). We will only need the time-scaled quantities in  the present section and  in  Section \ref{sec:Localization to  outgoing  region}. 
\begin{subequations}

  Note (recalling $r:=\inp{x}$) that
  \begin{align}\label{eq:mbndss}
  m_t(x)&\geq \max \parb{m_{a_0}(x)+ 2/3,t^{\lambda}}\geq  2^{-1}\max \parb{r,t^{\lambda}},\\
   m_t(x)&\leq 3\max \parb{r,t^{\lambda}}, \label{eq:mbndss2} 
  \end{align} while the $x$-derivatives $\pa^\gamma_x m_t(x)$ obey
  \begin{equation}\label{eq:derx}
    \forall |\gamma|\geq 1 \,\exists C>0 \,\forall t\geq 1\,\forall x\in \bX:\q\abs[\big]{\pa^\gamma_x m_t(x)}\leq Cr^{1-|\gamma|}.
  \end{equation}
\end{subequations}

  Introducing  the Heisenberg derivative $\bD=\pa_t+\i [H(t),\cdot]$,  we record that 
\begin{align}\label{eq:Heis1}   
  \begin{split}
  \bD m_t&=M_t+\pa_t m_t=M_t+ \lambda t^{\lambda-1}\parb{m\parb{x/t^{\lambda}}- x/t^{\lambda}\cdot \parb{\nabla m}\parb{x/t^{\lambda}}}\\ &=
          M_t+\vO\parb{t^{\lambda-1} }.
          \end{split}                                                                                            
\end{align}
\begin{subequations} In turn  
\begin{align}
  \begin{split}
    \bD M_t&=4 t^{-\lambda}\parb{p-\tfrac  \lambda 4
      x/t}\cdot\parb{\nabla^2
      m}(x/t^{\lambda}) \parb{p-\tfrac  \lambda 4 x/t}\\&
    -\tfrac {\lambda^2}4 t^{\lambda-2} x/t^\lambda\cdot\parb{\nabla^2 m}(x/t^{\lambda})  x/t^\lambda\\&- t^{-3\lambda}\parb{\Delta^2 m}(x/t^{\lambda}) -
2\parb{\nabla  m}(x/t^{\lambda})\cdot \nabla  V
    \\ &=
         B_t^*B_t+R_t, \q \text{where}
    \\ & \q B_t=2t^{-\lambda/2}\sqrt{\parb{\nabla^2
      m}(x/t^{\lambda}) }\parb{p-\tfrac  \lambda 4 x/t},
    \label{eq:Heis2}                                                                                               \end{split}
\end{align} 
and the remainder term $R_t$ obeys  the bounds
 \begin{align}\label{eq:Rbnd1}  R_t&=\vO\parb{t^{\lambda-2} }+\vO\parb{t^{-3\lambda} }+\vO\parb{t^{-\lambda(1+2\mu)} },\\
  \label{eq:Rbnd2}
   R_t&=\vO\parb{r^{1-2/\lambda} }+\vO\parb{r^{-3} }+\vO\parb{r^{-1-2\mu} }\q (\text{uniformly in }t\geq 1).
\end{align} 
\end{subequations}
The middle  term $- t^{-3\lambda}\parb{\Delta^2 m}(x/t^{\lambda})$ of  $R(t)$ is a `quantum correction term', while the two other terms `agree with'  Classical Mechanics. In Subsection \ref{subsec:Bound in incoming region} we compute and bound  the  Heisenberg derivative $\bD\Psi$ of  an observable $\Psi$ built from  the quantities $m_t$ and $M_t$, see (\ref{eq:PS2}) and (\ref{eq:negHeis}). The computation  will demonstrate that the `leading order' of   $\bD\Psi$  agrees with the time-derivative of the analogous classical observable, as expected. Hence  for example the above middle  term will  only contribute by a sufficiently small term. 

  \subsection{An auxiliary Hamiltonian}\label{subsec:An auxiliary Hamiltonian}

We consider for $t\geq 1$ the  Hamiltonian
\begin{equation}\label{eq:widecheckH2}
  \wcH_t=\wcH_{t,\delta,\lambda} =m_t^{-\delta}H(t)m_t^{-\delta},\q \delta\in (1/2,3/5), \,2\lambda\delta>1,
\end{equation}
realized as a  self-adjoint operator on $\vH$ by the Friedrich extension
procedure.
More precisely we consider the closed quadratic form
\begin{align*}
  \inp{\psi, \wcH_t \psi}&= \sum_{j\leq d}\, \norm [\big]{p_j m_t^{-\delta} \psi }^2+\inp{\psi, V(t)\psi} \text{ on }\vH_{\rm qf};\\&\q \vH_{\rm qf}=\set{\psi\in L^2\mid
  p_j\psi \in L_{-\delta }^2 \text{ for all }j\leq  d}.
\end{align*}
Its operator domain is explicitly given as
\begin{align*}
  \vD( \wcH_t )=\set{\psi\in \vH_{\rm qf}\mid
  p^2\psi \in L_{-2\delta }^2}=\set{\psi\in L^2\mid
  p^2m_t^{-\delta} \psi \in L_{-\delta }^2} ,
\end{align*}
  and it follows by the standard mollifier technique  that the Schwartz space $\vS(\bX)$ is dense in $\vH_{\rm qf}$ and as well as in $\vD( \wcH_t )$. Moreover we may  consider the unbounded term  $m_t^{-\delta}p^2m_t^{-\delta}$  as a
  PsDO (i.e. a pseudodifferential operator, see \cite[Chapter XVIII]{H1} for a general account of PsDOs). The above assertions may be proved by commutation. Although we skip the details, note for example  the elliptic bound 
  \begin{align*}
    \sum_{j\leq d}\, \norm [\big]{p_j\parb{1+\kappa p^2}^{-1/2} m_t^{-\delta} \psi }^2\leq C \parb{\norm{\psi}^2+ {\norm{ m_t^{-2\delta}p^2\psi}^2}}.
  \end{align*} Here $C$ is independent of $\kappa>0$,  allowing us to let $\kappa\to 0$. 

  The parameter $\delta$ can  throughout the paper  be considered as 
fixed, possibly   for example as $\delta=4/7$. In the present section actually any $\delta\in (1/2,1)$ can be used (with $\lambda<1$ obeying   $\lambda> 1/(2\delta)$), see Remark \ref{remark:bound-incom-regi} \ref{item:2psibb} for a discussion.

The operator $M_t$ ia a  first order differential operator, and
consequently it has an  explicit flow. It is easily checked  (for
example by using the Gronwall inequality) that this 
flow  $\e^{\i \sigma M_t}$ leaves $\vS(\bX)$,  $\vD(\wcH_t)$  and the
spaces $L^2_s$ invariant. In particular $M_t$ is essentially self-adjoint   on $\vS(\bX)$ and
\begin{align*}
  M_t\psi&= \parb{2p\cdot \parb{\nabla  m}\parb{x/t^{\lambda}}+\i t^{-\lambda}(\Delta m)\parb{x/t^{\lambda}}}\psi \text{ on }\vD( M_t);\\
  &\vD( M_t)=\set{\psi\in L^2\mid p\cdot \parb{\nabla  m}\parb{x/t^{\lambda}}\psi\in L^2}.
  \end{align*} (Since $ M_t$ has a distributional action, say also denoted $ M_t$, alternatively written  $\vD( M_t)=\set{\psi\in L^2\mid M_t \psi\in L^2}$.)

We can  interpret the commutator 
$\i [\wcH_t, M_t]$ as an operator with domain $\vD(\wcH_t)$ and
specified as  the (strong) limit
\begin{equation}
  \label{eq:flow}
  -\i \ad ^1_{\wcH_t}(M_t)=\i \big [\wcH_t, M_t \big ]=\slim_{\sigma\to 0}\, \sigma^{-1}\parb{\e^{-\i \sigma M_t}\wcH_t \e^{\i \sigma M_t}-\wcH_t },
\end{equation} and then in turn by its formal expression (as computed
on $\vS(\bX)$). In particular $\i [\wcH_t, M_t]$  may be considered as
an $\wcH_t$-bounded operator.  (See \cite{GGM} for an abstract account
of  computing commutators this way.) In particular $(\wcH_t-z)^{-1}\in
\vL (\vD( M_t))$ for $\Im z \neq 0$. 

Furthermore we can record the  bounds, cf.  the convention 
\eqref{eq:1712022},
  \begin{equation}\label{eq:bndM}
  (\wcH_t-z)^{-1},\q p_jm_t^{-\delta}(\wcH_t-z)^{-1}=\vO(r^0);\q \Im z
  \neq 0,\q1\leq j\leq d.
\end{equation}  The boundedness on $L^2$ follows from the very definition of
$\wcH_t$. To show  the invariance  of $L^\infty$ we need to commute any
big power  $\inp{x}^s$ through the  resolvent $ (\wcH_t-z)^{-1}$. This
may  conveniently be done by first repeatingly commuting the bounded
approximation  $(\chi_-(|x|/n)\inp{x})^s$ through  resolvents to lower
the exponent step by step, and
then by taking  $n\to \infty $ using the observed boundedness on $L^2$. The resulting formula shows that $
\inp{x}^s(\wcH_t-z)^{-1}\inp{x}^{-s}$ is bounded. We can argue
similarly for the remaining cases of \eqref{eq:bndM}. 

In particular  the operator $M_t m_t^{-\delta}(\wcH_t-z)^{-1}$, 
viewed   as an
element of  $\vL (\vD( M_t),
\vH)$  for any given $\Im z \neq 0$,   extends to a bounded operator on $\vH$, and in fact $M_t (\wcH_t-z)^{-1}$ can unambiguously be viewed as  an operator of   order $\delta$.

 Very frequently the abstract parameter $\theta$ used below
\eqref{eq:1712022} will be  the time $t\geq 1$. To simplify the presentation, often  we  prefer  to  use  the
unified notation $\vO(r^s)$ in this case, even though the bounds
 are uniform in $t\geq 1$. This convention is already encountered above.

\begin{lemma}\label{lemma:expansion}
  \begin{subequations}
 Let  $\chi\in C^\infty(\R)$ be given with  compactly supported
  derivative and let  $f\in C_\c^\infty(\R)$. Then $f\parb{\wcH_t}$
  and $\chi(M_t)$  are of order $0$ (uniformly in $t\geq 1$) and the following
  expansions are valid   (with  remainder estimates  being uniform
  in $t\geq 1$).
  \begin{enumerate}[1)]
\item \label{item:co1}  For all $s\in \R$:
\begin{align}\label{eq:ExpsBreveH}
\comm[\big]{f\parb{\wcH_t},m_t^s}&=-\ad ^1_{\wcH_t}(m_t^s)
  f^{(1)}\parb{\wcH_t} + \vO(r^{s-2-2\delta}),\\
M_tf\parb{\wcH_t}-f\parb{\wcH_t}M_t&= \ad ^1_{\wcH_t}(M_t)
  f^{(1)}\parb{\wcH_t} + \vO(r^{-2-\delta}).\label{eq:ExpsBreveHb}
\end{align}
  \item \label{item:co2} For all $s\in \R$ and $K\in \N :$
 \begin{equation}\label{eq:Exps} 
  \comm[\big]{\chi(M_t),m_t^s}=\sum ^K_{k=1}\, \tfrac{(-1)^k} {k!}\ad ^k_{M_t}(m_t^s)
  \chi^{(k)}(M_t) + \vO(r^{s-K-1}).
\end{equation}
\item \label{item:co3} For $ k=0,1,2$ the operator  $\ad
  ^k_{M_t}(f\parb{\wcH_t}) =  \vO(r^{-k})$, and
\begin{align}\label{eq:ExpBreveH}
 \chi(M_t) f\parb{\wcH_t}=\sum ^2_{k=0}\, \tfrac{(-1)^k} {k!}\ad ^k_{M_t}(f\parb{\wcH_t})
  \chi^{(k)}(M_t) + \vO(r^{-2-\delta}).
\end{align}
 \end{enumerate}
   \end{subequations}
\end{lemma}
\begin{remarks*}
  The first  term   to
  the right in \eqref{eq:ExpsBreveH} is  of order $s-1-\delta$. The first term  to
  the right in \eqref{eq:ExpsBreveHb}  is  of
  order $-1$, cf. (\ref{eq:flow}). The left-hand side of  \eqref{eq:ExpsBreveHb} is a priori given as a quadratic form on $\vD(M_t)$, which thanks to (\ref{eq:flow}) extends to a bounded operator as demonstrated in   the below computation (\ref{eq:long}).
     The operators in
  \eqref{eq:Exps} and \eqref{eq:ExpBreveH} are of order $s-1$ and $0$, respectively. Note also that although the appearing factor $\chi(M_t)$ in
  \eqref{eq:Exps} and \eqref{eq:ExpBreveH} is  of order $0$, and the
  stated expansions are valid, this operator is not a good PsDO in any
  conventional sense.
\end{remarks*}
\begin{proof}[Proof of Lemma \ref{lemma:expansion}]
  By  repeated commutation using the 
  representation \eqref{82a0} and bounded approximations of $\inp{x}$,
  for example given by $\chi_-(|x|/n)\inp{x}$ for large  $n\in \N$ as
  used for  \eqref{eq:bndM}, it is easily    proven  by
  \eqref{eq:bndM} and its proof that $f\parb{\wcH_t}$ has order $0$.

  For the  assertion \ref{item:co1}   we similarly     compute 
  \begin{align}\label{eq:T99}
    \begin{split}
    \comm[\big]{f\parb{\wcH_t}&,m_t^s}+\ad ^1_{\wcH_t}(m_t^s)
  f^{(1)}\parb{\wcH_t} \\&= 
   \int _{\C}(\wcH_t -z)^{-1} \ad ^2_{\wcH_t}(m_t^s)(\wcH_t -z)^{-2}\,\mathrm d\mu_{f}(z),  
    \end{split}
 \end{align} where $\ad ^2_{\wcH_t}(m_t^s) (\wcH_t -\i)^{-1}$ has order $s-2-2\delta$. This shows \eqref{eq:ExpsBreveH}.

 As quadratic forms  on $\vD(M_t)$ we have (with limits taken in the weak sense) 
 \begin{align*}
   M_t f\parb{\wcH_t}&=\lim_{\sigma\to 0}\, \i\sigma^{-1}\parb{\e^{-\i \sigma M_t}-I}f\parb{\wcH_t},\\
   f\parb{\wcH_t}
   M_t &=\lim_{\sigma\to 0}\, f\parb{\wcH_t}\i \sigma^{-1}\parb{\e^{-\i \sigma M_t}-I},
 \end{align*} respectively. In combination with (\ref{eq:flow}) we  compute as in (\ref{eq:T99}), using below the notation $T(\sigma)=\e^{-\i \sigma M_t }T\e^{\i \sigma M_t}$,
 \begin{align}\label{eq:long}
   \begin{split}
    &\comm[\big]{f\parb{\wcH_t},M_t } =
      -\i \slim_{\sigma\to 0}\, \sigma^{-1}\parb{f\parb{\wcH_t(\sigma)}-f\parb{\wcH_t}}\\
&=\i \slim_{\sigma\to 0}\, \int _{\C}(\wcH_t(\sigma) -z)^{-1}  \sigma^{-1} \parb{\wcH_t(\sigma)-\wcH_t}(\wcH_t -z)^{-1}\,\mathrm d\mu_{f}(z) \\
&=  \int _{\C}(\wcH_t -z)^{-1}  \ad ^1_{\wcH_t}(M_t)(\wcH_t -z)^{-1}\,\mathrm d\mu_{f}(z) \\&=-\ad ^1_{\wcH_t}(M_t)
      f^{(1)}\parb{\wcH_t}+\int _{\C}(\wcH_t -z)^{-1} \ad ^2_{\wcH_t}(M_t)(\wcH_t -z)^{-2}\,\mathrm d\mu_{f}(z),  
   \end{split}
 \end{align}  Since  $(\wcH_t -\i)^{-1}\ad ^2_{\wcH_t}(M_t) (\wcH_t -\i)^{-1}$ has order $-2-\delta$,  (\ref{eq:ExpsBreveHb})  follows.

  For the assertion \ref{item:co2} it is convenient to use \eqref{82a0} to compactly
  supported approximations of $\chi$  like $\chi_n=\chi_-(|\cdot|/n)\chi$. It follows that
  $\chi_n(M_t)$ has order $0$. Moreover the  validity of \ref{item:co2}  is   easily checked  in this case. Note that $\ad ^k_{M_t}(m_t^s)$ is a multiplication operator of order $s-k$.  To finish the proof of \ref{item:co2} we take almost analytic extensions of the approximating functions with uniform bounds on the densities of the corresponding measures $d\mu_{\chi_n}(z)=\pi^{-1} \bar\pa \wt {\chi_n}(z) \d u\d v$. More precisely this means  a construction such that pointwise $\bar\pa \wt {\chi_n}(z)\to \bar\pa \wt {\chi}(z)$ and that the bounds 
  \begin{align}\label{eq:Alm0}
    \forall N\in \N:\q\abs{\bar\pa \tilde{f}(z)}\leq C_N \inp{z}^{-1-N}\abs{\Im z}^N
  \end{align} hold uniformly for the family of extended functions $\tilde{f}=\wt {\chi_n}$,  $n\in \N$. 
  This  allows us to let $n\to \infty$ using Lebesgue's
  dominated convergence theorem. In particular the  limiting procedure argument shows  that $\chi(M_t)$  has  order $0$ as well as \ref{item:co2} for $\chi(M_t)$. 

  For the assertion 
  \ref{item:co3} we proceed similarly and compute
  \begin{align*}
    &\comm[\big]{\chi(M_t) ,f\parb{\wcH_t}} 
      \\
    &=  \int _{\C}(M_t -z)^{-1}  \ad ^1_{M_t}(f\parb{\wcH_t} )(M_t -z)^{-1}\,\mathrm d\mu_{\chi}(z) \\
    &= -\ad ^1_{M_t}(f\parb{\wcH_t})
  \chi'(M_t) +\int _{\C}(M_t -z)^{-1}  \ad ^2_{M_t}(f\parb{\wcH_t} )(M_t -z)^{-2}\,\mathrm d\mu_{\chi}(z). 
 \end{align*}  It follows from \eqref{eq:long} and its proof that $\ad ^1_{M_t}(f\parb{\wcH_t})$ and $\ad ^2_{M_t}(f\parb{\wcH_t})$ have  order $-1$  and  $-2$, respectively.  This  suggests that with one more commutation to the left in the integral,  the expansion (\ref{eq:ExpBreveH}) would  hold  with an error of order $-3$.  However, since we do not  have assumptions on derivatives of the potential of  order bigger than two, this  commutation   only yields   an error of  order $-2-\delta$, as stated in (\ref{eq:ExpBreveH}).
\end{proof}

     We need  the following version of
the representation formula \eqref{82a0} now for the unitary operator
$U=U(1)$.
\begin{lemma}\label{lemma:Helffer}      Let $g\in C^\infty_{\c}(\T\setminus
  \set {1})$. Let $U$ be a unitary operator on $\vH$. Then
  \begin{align}\label{82a02}
  \begin{split}
 g(U) 
&=
\int _{\C}(U -z)^{-1}\,\mathrm d\mu_g(z)\, \text{ with}\\
&\mathrm d\mu_g(z)=\pi^{-1}(\bar\partial\tilde g)(z)\,\mathrm du\mathrm dv;\quad 
z=u+\i v.   
  \end{split}
\end{align} Here $\tilde{g}$ denotes  a  compactly supported almost analytic extension of
$g$. 
      \end{lemma}
   \begin{proof} 
 The function $\tilde g$  can  be constructed by  the Cayley transform as
follows. Consider $h\in C^\infty_{\c}(\R)$ given by $h(x)
=g\parb{(x-\i)/(x+\i)}$. Take a compactly supported   almost analytic
extension $\t h$ of $h$, and define then  $\tilde{g}(z)=\t
h(\eta(z))$, where $\eta(z)=-\i(z+1)/(z-1)$.  Since $\eta$ is analytic 
it follows that $\dbar\eta=0$,   and we can then compute 
\begin{equation*}
 \dbar\tilde{g}(z)=\pa_{\bar z} \t h\parb{\eta(z)}=\parb{\dbar \t h}\parb{ \eta(z)}\overline{\eta'(z)}.
\end{equation*}
 In particular it follows that for any $N\in\N$
 \begin{equation*}
   \abs{\dbar\tilde{g}(z)}\leq C_N \abs{\dist(z, \T)}^N.
 \end{equation*} Hence indeed $\tilde{g}$ is a  compactly supported almost analytic extension of
 $g$, and \eqref{82a02} follows (cf. \cite[(3.1.11)]{H1}). 
\end{proof}
\begin{corollary} 
  \label{cor:an-auxil-hamilt} Let $g\in C^\infty_{\c}(\T\setminus \vT_\p)$ and $\psi\in \vS(\bX)$. Then 
  \begin{equation*}
    [1,\infty) \ni t\to x^\alpha \pa_x^\beta U(t) g(U(1))\psi \in \vH
  \end{equation*} are  well-defined continuous maps for any  multiindices $\alpha$ and $\beta$  with $|\alpha+\beta|\leq 2$.
  \end{corollary}
  \begin{proof} For $\alpha=0$ the result follows from \eqref{82a02}, Yosida's theorem and the familiar Duhamel 
    formula (for example applied to the approximations of the momentum components $p_j $ given as  $T=p_j /(1+\kappa p^2)$ with  $\kappa>0$ small)
    \begin{equation}
      \label{eq:Duhamel}
      [T,U(\sigma)]=U(\sigma)\parb{U(\sigma)^{-1}TU(\sigma)-T}=U(\sigma)\int_0^\sigma\, U(s)^{-1}\i[H(s),T]U(s)\,\d s.
      \end{equation} By using  bounded approximations of $x^\alpha $ the general case can be treated similarly. The details are omitted.
  \end{proof}

\subsection{Bound in the very high energy region}\label{subsec:Bound in high energy region}
\begin{lemma}
  \label{lemmaHighEnergy} For any  $\psi\in\vH $ the evolution $\psi(t)=U(t)\psi$ obeys
  \begin{subequations}
  \begin{align}\label{eq:energyBND}
    \lim_{t\to \infty}\, \norm{\chi_+(\wcH_t)\psi(t)}&=0,\\
    \lim_{t\to \infty}\, \norm{\parb{I-\chi_-(\wcH_t)}\psi(t)}&=0.\label{eq:CONs}
\end{align}  
  \end{subequations}
\end{lemma}
\begin{proof} 
  Note that \eqref{eq:CONs} is a consequence of \eqref{eq:energyBND}. To prove \eqref{eq:energyBND} it suffices (by density) to check the estimate for any fixed  $\psi\in \vD(p^2)$.
  \Step {I} Introducing the  Hamiltonian  
\begin{equation*}
  \widehat H_t=2^{2\delta+1}t^{-2\lambda\delta}H(t),
\end{equation*} we will show that
\begin{equation}\label{eq:widehatBig}
  \norm{\chi_+(\widehat H_t)\psi(t)}\to 0.
\end{equation}

We introduce
$f\in C^\infty(\R)$  by the recipe    $f(s)=(s-1)\chi_+^2(s)$ and then the observable 
\begin{equation*}
  \Psi=t^{2\lambda\delta}f(\widehat H_t).
\end{equation*} The Heisenberg derivative $\bD=\pa_t+\i [H(t),\cdot]$ of this observable  is approximate negative. More precisely, 
 recalling    the
generic notation
$\inp{T}_{\varphi}=\inp{\varphi,T\varphi}$, we  can compute $\tfrac{\d}{\d t}\inp{\Psi}_{\psi(t)}=\inp{\bD \Psi}_{\psi(t)}$  with 
 \begin{equation}\label{eq:NegHat}
   \bD \Psi\leq Ct^{-\lambda\delta};\q t\geq 1.
 \end{equation} 
 To prove (\ref{eq:NegHat}) we first note that
 \begin{equation*}
   \bD \widehat H_t=-2\lambda\delta t^{-1 }\widehat H_t +2^{2\delta+1}t^{-2\lambda\delta}(\pa_ t V) .
 \end{equation*} 
 Taking  a suitable almost analytic  extension $\t f$  of $f$ obeying 
 \begin{align}\label{eq:Almost}
    \forall N\in \N:\q\abs{\bar\pa \tilde{f}(z)}\leq C_N \inp{z}^{-N}\abs{\Im z}^N
 \end{align} we then use  (\ref{82a0}) and obtain  to leading order
 \begin{align*}
   \bD \Psi &\approx 2\lambda\delta t^{-1 } \Psi- t^{2\lambda\delta-1 }\sq {f'}(\widehat H_t)\parbb{2\lambda\delta \widehat H_t -2^{2\delta+1}t^{1-2\lambda\delta}(\pa_ t V) }\sq {f'}(\widehat H_t)\\
            &\leq  - t^{2\lambda\delta-1 }\sq {f'}(\widehat H_t)\parb{2\lambda\delta  -2^{2\delta+1}t^{1-2\lambda\delta}(\pa_ t V) }\sq {f'}(\widehat H_t)\\
   &\leq  - t^{2\lambda\delta-1 } \parb{2\lambda\delta  -Ct^{1-2\lambda\delta}}f'(\widehat H_t).
 \end{align*} Since  $2\lambda\delta>1$ the right-hand side is
 eventually  non-positive.  The error in the  above approximation (indicated by $\approx$) is a sum of two terms. The first one is  given by 
 \begin{align*}
   &2^{2\delta+1}\int _{\C}(\widehat H_t -z)^{-1}  [\widehat H_t ,(\pa_ t V)](\widehat H_t -z)^{-2}\,\mathrm d\mu_{f}(z)\\
   =-\i \tfrac {2^{4\delta+3}}{t^{2\lambda\delta}}&\int _{\C}(\widehat H_t -z)^{-1}  \Re\parb {p\cdot(\nabla\pa_ t V)}(\widehat H_t -z)^{-2}\,\mathrm d\mu_{f}(z)= \vO( t^{-\lambda\delta}).
 \end{align*} Here we used the uniform bound $ t^{-\lambda\delta}\norm{p(\widehat H_t -z)^{-1}  }\leq C |\Im z|^{-1}\inp{ z}^{1/2}$ and  (\ref{eq:Almost}).
 The second one is given by
 \begin{align*}
   &2^{2\delta+1}[(\pa_ t V), \sq {f'}(\widehat H_t)]\sq {f'}(\widehat H_t)\\
   =-\i \tfrac {2^{4\delta+3}}{t^{2\lambda\delta}}&\int _{\C}(\widehat H_t -z)^{-1}  \Re\parb {p\cdot(\nabla\pa_ t V)}(\widehat H_t -z)^{-1}\,\mathrm d\mu_{\sq {f'}}(z)\,\sq {f'}(\widehat H_t)= \vO( t^{-\lambda\delta}).
 \end{align*} Here we estimated similarly using a representation of $\sq {f'}$ in terms of an  almost analytic  extension obeying \eqref{eq:Alm0}. Hence the total error is on the form $\vO( t^{-\lambda\delta})$, and (\ref{eq:NegHat}) follows.

 Noting  that $\Psi\geq 3^{-1}t^{2\lambda\delta} \chi^2_+(\widehat H_t)$, clearly (\ref{eq:widehatBig})   follows from  (\ref{eq:NegHat}) by integration.
\Step {II} 
Thanks to (\ref{eq:widehatBig})  it remains to show that
\begin{equation*}
  \norm{\chi_+(\wcH_t)\chi^2_-(\widehat H_t)\psi(t)}\to 0.
\end{equation*} This will be done by showing the `energy bound'
\begin{equation}\label{eq:energyBND2}
  \norm{\chi_+(\wcH_t)\chi_-(\widehat H_t)}\leq C t^{-\lambda\delta}.
\end{equation} 
 We first use  (\ref{eq:mbndss}) and  (\ref{eq:derx})
estimating  as
\begin{subequations}
\begin{align}\label{eq:wc1}
  \wcH_t\leq 2^{-1} \widehat H_t+ &C_1t^{-2\lambda\delta},\\
  \wcH_t(\widehat H_t-\i)^{-1}&=\vO(t^0),\label{eq:wc2}\\
                                           (\widehat H_t-\i)^{-1} [\wcH_t,\widehat H_t] &(\widehat H_t-\i)^{-1}=\vO(t^{-2\lambda\delta- \lambda(1-\delta)}).\label{eq:wc3}
\end{align}
\end{subequations}
Then we apply   these bounds, \eqref{82a0}  and repeated commutation,     estimating  as follows for any $\varphi \in \vH$:
  \begin{align*} 
  &\tfrac{1}3 \norm {\chi_+(\wcH_t)\chi_-(\widehat H_t)\varphi}^2\\
&\leq \inp{\wcH_t-I
  }_{\chi_+(\wcH_t)\chi_-(\widehat H_t)\varphi}\\
&\leq \inp{\wcH_t-I}_{ \chi_-(\widehat H_t)\chi_+(\wcH_t)\varphi}+C_2t^{-2\lambda\delta}\norm{\varphi}^2 \q(\text{by }(\ref{eq:wc2})\mand (\ref{eq:wc3})\text{ applied twice})\\&
\leq (C_1+C_2)t^{-2\lambda\delta} \norm{\varphi}^2 \qq(\text{by }(\ref{eq:wc1}).
 \end{align*}
 Hence  (\ref{eq:energyBND2}) follows with $C^2=3(C_1+C_2)$, and   \eqref{eq:energyBND} is proven.
\end{proof}

\subsection{Bound in the incoming region}\label{subsec:Bound in incoming region}
We  derive a preliminary propagation estimate for  $M_t$, stated as  (\ref{eq:PS30}). (This study will be completed in  Section \ref{sec:Localization to  outgoing  region}.) Motivated by
\eqref{eq:CONs} and Lemma \ref{lemma:expansion} we introduce 
  real functions $f,\tilde f, \wc f \in C_\c^\infty (\R)$   such that
  \begin{subequations}
\begin{equation}\label{eq:fes}
  f\parb{\wcH_t}=\chi_-\parb{\wcH_t},\q \tilde{ f}\parb{\wcH_t}=\chi_-\parb{\wcH_t/2},\q \wc f\parb{\wcH_t}=
\tilde{ f}\parb{\wcH_t} \wcH_t \tilde{ f}\parb{\wcH_t}.
\end{equation} These functions are henceforth considered as  fixed and
independent of $t\geq 1$. (Note that  $C+\wcH_t\geq 1$ for some $C\geq
1$  being independent of $t\geq 1$.)   Obviously we can assume that
$\tilde{ f}\succ f$, and in the  presence of a  factor  of  $f(\wcH_t)$
we are   lead to   consider 
 the following 'partial regularization' of $M_t$. We introduce  for
 each $t\geq 1$
\begin{align*}
  \wt M_t\psi&= M_t \tilde{ f}\parb{\wcH_t}^2\psi-\big [M_t ,\tilde{ f}\parb{\wcH_t}\big  ]\tilde{ f}\parb{\wcH_t} \psi\text{ on }\vD(\wt  M_t);\\
  &\vD( \wt M_t)=\set{\psi\in L^2\mid M_t \tilde{ f}\parb{\wcH_t}^2\psi \in L^2}.
\end{align*} By Lemma \ref{lemma:expansion} \ref{item:co1} the second term is bounded.  We can  (unambiguously)  view $\wt M_t$ as  an operator of  (uniform) order $\delta$, cf. \eqref{eq:bndM}. We show below   that  $\wt M_t$ is self-adjointness with core $\vS(\bX)$.
 Note that formally  
\begin{equation}\label{eq:Mtilde}
  \wt M_t=\tilde{ f}\parb{\wcH_t}M_t  \tilde{ f}\parb{\wcH_t}.
\end{equation}
Thanks to  Lemma \ref{lemma:expansion} \ref{item:co1}  we can record that $\vD(  M_t)\subseteq \vD( \wt M_t)$   and that for any  $\psi\in \vD(M_t)$,  indeed $\wt M_t\psi= \tilde{ f}\parb{\wcH_t}M_t  \tilde{ f}\parb{\wcH_t}\psi$.

The minimum action of $\wt M_t$, say denoted by $\wt M_{t,\min}$ or for convenience below just by $M_{\min}$, has domain given by $\vS(\bX)$. Note that $\wt M_t =M^*_{\min}$. Since $\vS(\bX)$ is dense in $\vD(  M_t)$ and $\wt M_t$ is closed, it follows that $\vD(  M_t)\subseteq \vD(  \overline{M_{\min}})\subseteq \vD( \wt M_t)$. We claim that in fact the closure
\end{subequations}
\begin{align}
  \label{eq:closure} \overline{M_{\min}}=\wt M_t.
\end{align} 
 To see this  we compute for any  $\psi\in \vD(\wt M_t)$
 \begin{align*}
  \wt M_t\psi=\lim_{n\to \infty}\, \chi_-(r/n)\wt M_t\psi=\lim_{n\to \infty}\, \wt M_t\chi_-(r/n) \psi.
\end{align*} By  the standard mollifier technique (applied to $\chi_-(r/n) \psi$) we  then obtain
that $\vS(\bX)$ is a core for $\wt M_t$, as wanted.

Finally using (\ref{eq:closure}) it follows that $\wt M_t=\overline{M_{\min}}=(M^*_{\min})^*=\parb{\wt M_t}^*$. We   conclude     that  indeed $\wt M_t$ is essentially self-adjoint  on $\vS(\bX)$.

    \begin{remarks} \label{remarks:tilde}
      \begin{enumerate}[i)]
  \item\label{item:tilde1}
    We need $\wt M_t$ rather than just $ M_t$ to prove
    the below assertion (\ref{eq:PS30}). The  reason  is  that there is no  way of
    bounding $ (t^{\epsilon_1} M_t -z)^{-1}(\pa_t M_t)(t^{\epsilon_1}
    M_t -z)^{-1}$, not even with an external momentum cut-off like
    $\tilde{ f}\parb{\wcH_t}$. Hence we can not use  the same method
    as the one applied  below for  (\ref{eq:PS3}).
    \item\label{item:tilde2}
      It is easily checked that the operator  $\chi_-(t^{\epsilon_1}\wt M_t)$ in   (\ref{eq:PS3}) (with   $\epsilon_1>0$ small) is an operator of zero order, cf. Lemma \ref{lemma:expansion}. In fact repeated commutation shows   that for any $s\in \R$ 
      \begin{equation}\label{eq:ExpsB} 
  \comm[\big]{\chi_-(t^{\epsilon_1}\wt M_t),m_t^s}= \vO(r^{s+\epsilon_1/\lambda-1}).
\end{equation}
      \end{enumerate}
  \end{remarks}
  \begin{lemma}
  \label{lemmaNegForbidden} For all sufficiently small  $\epsilon_1>0$ (determined only by $\lambda$ and $\delta$) and for all    $\psi\in \vH_{\ac}$ the following bounds hold for $\psi(t)=U(t)\psi$.
  \begin{subequations}
   \begin{align}\label{eq:PS30}
  &\lim_{t\to \infty}\,\norm{\chi_-(t^{\epsilon_1}M_t)\psi(t)}= 0,\\
\label{eq:PS3}
  &\lim_{t\to \infty}\,\norm{\chi_-(t^{\epsilon_1}\wt M_t)\psi(t)}= 0.
\end{align}
  \end{subequations}
\end{lemma}
\begin{proof}  The assertions \eqref{eq:PS30} and  \eqref{eq:PS3} are equivalent. Indeed by \eqref{eq:CONs} be can freely replace $\psi(t)$ by
  $f(\wcH_t)\psi(t)$.
Note  then (here assuming that $\epsilon_1<\lambda$) that
\begin{subequations}
 \begin{equation}\label{eq:T1I}
   T:=\parb{\chi_-(t^{\epsilon_1}\wt M_t)-\chi_-(t^{\epsilon_1} M_t)}f(\wcH_t)=\vO(t^{\epsilon_1-\lambda})=\vO(t^{-0_+}),
 \end{equation} which by \eqref{82a0} is an easy consequence of the represention
 \begin{equation}\label{eq:T1II}
   T=\int _{\C}(t^{\epsilon_1} M_t -z)^{-1} t^{\epsilon_1}\parb{M_t-\wt M_t}(t^{\epsilon_1}\wt M_t -z)^{-1}f(\wcH_t)\,\mathrm d\mu_{\chi_{-}}(z)
 \end{equation} and the facts that $f=\tilde{ f}f$ and $[M_t, \t f(\wcH_t)]=\vO(t^{-\lambda})$ (cf. (\ref{eq:ExpsBreveHb})).
  \end{subequations}
  Thus it suffices  to   show \eqref{eq:PS3}.

  For \eqref{eq:PS3} we can assume (by density) that
$\psi=g(U(1))\br\psi$  with  
$\br\psi\in \vS(\bX)$ and   $g\in C^\infty_{\c}(\T\setminus \vT_\p)$.

Consider for small $\epsilon_1> \epsilon_2>0$
the time-dependent observable
\begin{subequations}
\begin{equation}
  \label{eq:PS2}
  \Psi=f(\wcH_t)\chi_-(t^{\epsilon_1}\wt M_t)m_t^{\epsilon_2}\chi_-(t^{\epsilon_1}\wt M_t)f(\wcH_t);\q t\geq 1.
\end{equation}  More precisely we consider  in  Steps I-IV below this  quantity $\Psi$ under the 
conditions $2\lambda\delta>1$, $\epsilon_1< 2(1-\lambda)$ and
$\epsilon_2$ taken sufficiently small, cf. \eqref{eq:Young} and
(\ref{eq:smalle1}). The complete  proof of \eqref{eq:PS3} (given in Steps  V and  VI) requires   additional smallness of  $\epsilon_1$, see for example \eqref{eq:S2} and  (\ref{eq:mainS3}). We record that $m_t^{-\epsilon_2/2}\Psi m_t^{-\epsilon_2/2}$ is of order zero uniformly in $t\geq 1$.

Now by using \eqref{eq:mbndss2}  and  Corollary
\ref{cor:an-auxil-hamilt} we can compute and estimate (to be elaborated on in several steps below) $\tfrac{\d}{\d t}\inp{\Psi}_{\psi(t)}=\inp{\bD \Psi}_{\psi(t)}$  with
 \begin{equation}\label{eq:negHeis}
   \bD \Psi=\pa_t \Psi+\i [H(t),\Psi]\leq  \vO(t^{-1_+})+\vO_{\unif}(r^{-1_+}) .
 \end{equation}
\end{subequations} 
  Thanks to \eqref{eq:LocalDecay} and integration,  \eqref{eq:negHeis}  yields that  $\inp{\Psi}_{\psi(t)}$ is bounded in time. Using then 
  \eqref{eq:mbndss} and \eqref{eq:CONs}, clearly (\ref{eq:PS3}) follows.

  To establish \eqref{eq:negHeis}  let us as a motivation first
  extract the `leading order'  contribution from $\bD \Psi$ (agreeing
  with Classical Mechanics) and  disregard  `quantum commutation
  errors'. We will in Steps I-IV see that the leading order term
  agrees with \eqref{eq:negHeis}, devoting  Steps  V and  VI  to a complete proof of \eqref{eq:negHeis} using in particular various notation from the previous steps.
 \Step {I} We write (using here and throughout Steps I-IV  the approximately  equal symbol $\approx$ intuitively)
 \begin{align}\label{eq:Leib}
   \begin{split}
   \bD \Psi&\approx 2 \Re\parbb{\parb{\bD \wcH_t}f'(\wcH_t)\chi_-(t^{\epsilon_1}\wt M_t)m_t^{\epsilon_2}\chi_-(t^{\epsilon_1}\wt M_t)f(\wcH_t)}\\
           &+  2 \Re\parbb{f(\wcH_t)\parb{\bD (t^{\epsilon_1}\wt M_t)}\chi'_-(t^{\epsilon_1}\wt M_t)m_t^{\epsilon_2}\chi_-(t^{\epsilon_1}\wt M_t)f(\wcH_t)}\\
           &+   f(\wcH_t)\chi_-(t^{\epsilon_1}\wt M_t)\parb{\bD m_t^{\epsilon_2}}\chi_-(t^{\epsilon_1}\wt M_t)f(\wcH_t)=:T_1+T_2+T_3,  
   \end{split}
 \end{align} and then in turn, using \eqref{eq:Heis1} and \eqref{eq:Heis2},
 \begin{subequations}
 \begin{align}\label{eq:1Heis}\bD \wcH_t&\approx -2\delta m_t^{ -1}\wcH_t\parb{\wt M_t+ \vO(t^{\lambda-1})}+m_t^{-2\delta}\vO(r^{0}),\\\label{eq:2Heis}
     \bD (t^{\epsilon_1}\wt M_t)&\approx
     \epsilon_1t^{-1}\parb{t^{\epsilon_1}\wt M_t}+
     t^{\epsilon_1}\parb{B_t^*B_t+R_t},\\
   \bD m_t^{\epsilon_2}&= \epsilon_2 m_t^{(\epsilon_2-1)/2}\parb{\wt M_t+ \vO(t^{\lambda-1})}m_t^{(\epsilon_2-1)/2} .\label{eq:3Heis}
 \end{align}  
\end{subequations} Note that in (\ref{eq:1Heis})-(\ref{eq:3Heis}) we do not distinguish between $ M_t$ and $\wt M_t$, which of course is motivated by the presence of factors of  $f(\wcH_t)$ in (\ref{eq:Leib}) and the property $\tilde{ f}\succ f$.
 Let us in the following Steps II-IV  identify the terms
obtained by substituting (\ref{eq:1Heis})-(\ref{eq:3Heis}) into the respective terms $T_1$-$T_3$ of \eqref{eq:Leib}. As to be demonstrated they are all  essentially negative in the precise sense of  \eqref{eq:negHeis}.

 \Step {II}
 For $T_1$ we introduce 
 \begin{align}\label{eq:barf}
   \bar f(s)=\sq{-\chi_-'(s)s\chi_-(s)},
 \end{align} leading with  (\ref{eq:1Heis}) to 
\begin{align}\label{eq:Leib1}
   \begin{split}
     T_1&\approx 4\delta {\bar
       f(\wcH_t)m_t^{(\epsilon_2-1)/2}\chi_-(t^{\epsilon_1}\wt M_t)\parbb{\wt M_t+
         \vO(t^{\lambda-1}}\chi_-(t^{\epsilon_1}\wt M_t)m_t^{(\epsilon_2-1)/2}\bar
       f(\wcH_t)}\\ & \qq \qq + \vO(r^{\epsilon_2-2\delta})\\
        &\leq 8\delta t^{-\epsilon_1}\bar f(\wcH_t)m_t^{\epsilon_2-1}\bar f(\wcH_t)+ t^{\lambda-1}\vO(r^{
       \epsilon_2-1})+ \vO(r^{\epsilon_2-2\delta})\\
     &=  t^{-\epsilon_1} \vO(r^{\epsilon_2-1})+ t^{\lambda-1}\vO(r^{
       \epsilon_2-1})+
     \vO(r^{\epsilon_2-2\delta})\\
 &=  \vO(t^{-1_+})+\vO(r^{-1_+}).
    \end{split}
 \end{align} In the last step we used the familar estimate  $ab\leq p^{-1}a^p+q^{-1} b^q$, $ p^{-1}+q^{-1}=1$.  For the middle term this leads to  
 \begin{equation*}
   t^{\lambda-1}r^{ \epsilon_2-1} \leq  p^{-1}t^{(\lambda-1)p} +q^{-1}r^{
     (\epsilon_2-1)q} 
 \end{equation*} with $p>1$ chosen so  big and $\epsilon_2>0$ so small that 
 \begin{equation}\label{eq:Young}
     (1-\lambda)p>1 \mand (1-\epsilon_2)q>1.
 \end{equation}  We argue similarly for the first term.

 \Step {III} For $T_2$ we introduce 
 \begin{align*}
   \bar \chi_1(s)=\bar f(s)=\sq{-\chi_-'(s)s\chi_-(s)}\mand \bar \chi_2(s)=\sq{-\chi_-'(s)\chi_-(s)},
 \end{align*} leading with  (\ref{eq:2Heis})  and \eqref {eq:Rbnd2} to 
\begin{align}\label{eq:Leib2}
   \begin{split}
     T_2&\approx -2\epsilon_1t^{-1}f(\wcH_t)\bar \chi_1(t^{\epsilon_1}\wt M_t) 
       m_t^{\epsilon_2} \bar \chi_1(t^{\epsilon_1}\wt M_t) f(\wcH_t)\\
&-2t^{\epsilon_1}f(\wcH_t)\bar \chi_2(t^{\epsilon_1}\wt M_t) 
       m_t^{\epsilon_2/2}\parbb{B_t^*B_t+R_t}m_t^{\epsilon_2/2}
       \bar \chi_2(t^{\epsilon_1}\wt M_t) f(\wcH_t)\\
        &\leq r^{\epsilon_1/\lambda}\parbb{\vO\parb{r^{1-2/\lambda+\epsilon_2} }+\vO\parb{r^{-1-2\mu+\epsilon_2} }}\\
        &= \vO(r^{-1_+}).
    \end{split}
\end{align} In the last step we needed  that
\begin{equation}\label{eq:smalle1}
  \epsilon_1< 2(1-\lambda),
\end{equation} and also that $\epsilon_2>0$ is taken sufficiently small. (We omitted the quantum correction term $\vO(r^{-3})$  of $R_t$.)
 \Step {IV} For $T_3$ we introduce 
 \begin{align*}
   \bar \chi_3(s)={s\chi^2_-(s)},
 \end{align*} leading with  (\ref{eq:3Heis})  to 
 \begin{align}
     T_3&= \epsilon_2 f(\wcH_t)\chi_-(t^{\epsilon_1}\wt M_t){m_t^{(\epsilon_2-1)/2}\parbb{\wt M_t+ \vO(t^{\lambda-1})}m_t^{(\epsilon_2-1)/2}}\chi_-(t^{\epsilon_1}\wt M_t)f(\wcH_t)\nonumber\\&=  \epsilon_2 t^{-\epsilon_1}f(\wcH_t)m_t^{(\epsilon_2-1)/2}\bar\chi_3(t^{\epsilon_1}\wt M_t)m_t^{(\epsilon_2-1)/2}f(\wcH_t)+ \vO(r^{ \epsilon_2-2})+t^{\lambda-1} \vO(r^{ \epsilon_2-1})\nonumber\\
&\leq t^{-\epsilon_1}\vO(r^{\epsilon_2-1})+ \vO(r^{ \epsilon_2-2}) + t^{\lambda-1} \vO(r^{ \epsilon_2-1}) \label{eq:Leib3}\\
 &=  \vO(t^{-1_+})+\vO(r^{-1_+}).\nonumber
 \end{align} Here  the term $\vO(r^{ \epsilon_2-2})$ is a  quantum correction related to \eqref{eq:Exps} and Remark \ref{remarks:tilde}
 \ref{item:tilde2}  (which presently could be omitted).

 By combining the bounds from  Steps II-IV we obtain the estimate \eqref{eq:negHeis} up to `quantum correction terms', as desired.

 To proceed and   give a complete proof of \eqref{eq:negHeis}  we
 divide the problem into two parts,  analyzing the terms $\pa_t \Psi$ and $\i [H(t),\Psi]$  separately. 
 The principal tool is   \eqref{82a0}, to be used tacitly.
 \Step {V} 
 We treat  the `easy case' of $\pa_t \Psi$, more precisely we show that this term conforms with \eqref{eq:negHeis} up to a single term that  will be included and ultimately treated in Step VI. First we split
\begin{align}\label{eq:Leib8}
   \begin{split}
   \pa_t  \Psi&= 2 \Re\parbb{\parb{\pa_t  f(\wcH_t)}\chi_-(t^{\epsilon_1}\wt M_t)m_t^{\epsilon_2}\chi_-(t^{\epsilon_1}\wt M_t)f(\wcH_t)}\\
           &+  2 \Re\parbb{f(\wcH_t)\parb{\pa_t  \chi_-(t^{\epsilon_1}\wt M_t) }m_t^{\epsilon_2}\chi_-(t^{\epsilon_1}\wt M_t)f(\wcH_t)}\\
           &+   f(\wcH_t)\chi_-(t^{\epsilon_1}\wt M_t)\parb{\pa_t m_t^{\epsilon_2}}\chi_-(t^{\epsilon_1}\wt M_t)f(\wcH_t)=:S_1+S_2+S_3.  
   \end{split}
\end{align} We can compute $S_1$,  $S_2$ and $S_3$ separately and compare with  the corresponding `classical expressions' (defined explicitly below, although partially encountered in the previous steps), say denoted by $S^{\rm clas}_1$, $S^{\rm clas}_2$ and $S^{\rm clas}_3$, respectively.

Motivated by  Step II  we let
\begin{align*}
   \begin{split}
     S^{\rm clas}_1&= 2\bar
       f(\wcH_t)m_t^{\epsilon_2/2}\chi_-(t^{\epsilon_1}\wt M_t)2\delta (\pa_t m_t)m_t^{-1}\chi_-(t^{\epsilon_1}\wt M_t) m_t^{\epsilon_2/2}\bar
                     f(\wcH_t)\\&\qq + 2\Re\parb{f'(\wcH_t)m_t^{\epsilon_2-2\delta} (\pa_t V(t))\chi^2_-(t^{\epsilon_1}\wt M_t)f(\wcH_t)}.
     \end{split}
 \end{align*} Thanks to Lemma \ref{lemma:expansion} and Remark \ref{remarks:tilde}
 \ref{item:tilde2} we can record that
 \begin{align*}
   \begin{split}
     S^{\rm clas}_1= t^{\lambda-1}\vO(r^{
       \epsilon_2-1})+
     \vO(r^{\epsilon_2-2\delta}).
   \end{split}
 \end{align*}

We compute by commutation
\begin{align*}
  &\pa_t  f(\wcH_t)\\&=-\int _{\C}(\wcH_t -z)^{-1}\Re\parbb{ 2{(\pa_t m_t^{-\delta})H(t)m_t^{-\delta}+m_t^{-\delta} (\pa_t V(t))m_t^{-\delta}}}(\wcH_t -z)^{-1}\,\mathrm d\mu_f(z)\\
  &=\int _{\C}(\wcH_t -z)^{-1}\Re
    \parbb{2\delta(\pa_t m_t)m_t^{-1}\wcH_t+m_t^{-2\delta}\vO(r^0)}(\wcH_t -z)^{-1}\,\mathrm d\mu_f(z)\\
  &=-2\delta\Re\parbb{(\pa_t m_t)m_t^{-1}\wcH_t f'(\wcH_t)}+\vO(r^{-2\delta})\\
  &-\delta \int _{\C}(\wcH_t -z)^{-1}m_t^{-\delta}\comm[\big]{p^2, (\pa_t m_t)m_t^{-1}}m_t^{-\delta}\wcH_t(\wcH_t -z)^{-2}\,\mathrm d\mu_f(z)\\
  &+\delta \int _{\C}(\wcH_t -z)^{-2}\wcH_tm_t^{-\delta}\comm[\big]{p^2, (\pa_t m_t)m_t^{-1}}m_t^{-\delta}(\wcH_t -z)^{-1}\,\mathrm d\mu_f(z)\\
  &=-2\delta\Re\parbb{(\pa_t m_t)m_t^{-1}\wcH_t f'(\wcH_t)}+\vO(r^{-2\delta}).
\end{align*} Here we used the  expression for $\pa_t m_t$ from  \eqref{eq:Heis1}, calculating 
\begin{equation*}
 \comm[\big]{p^2, (\pa_t m_t)m_t^{-1}}=-\i2 \Re\parb{p\cdot \nabla \parb{(\pa_t m_t)m_t^{-1}}}=\Re\parb{p\cdot \vO(r^{-2})}, 
\end{equation*} and then bounding by \eqref{eq:bndM}.

 When  substituted in $S_1$   various commutation yields, cf.  Lemma \ref{lemma:expansion} and Remark \ref{remarks:tilde}
      \ref{item:tilde2},  that
\begin{align*}
  S_1-S^{\rm clas}_1=\vO(r^{\epsilon_1/\lambda+\epsilon_2-2})+
    \vO(r^{\epsilon_2-2\delta})=\vO(r^{-1_+}).
\end{align*} Note for example that
\begin{align*}
  &\Re\parb{(\pa_t m_t)m_t^{-1}\chi_-(t^{\epsilon_1}\wt M_t)  m_t^{\epsilon_2}\chi_-(t^{\epsilon_1}\wt M_t)}\\&= m_t^{\epsilon_2/2}\chi_-(t^{\epsilon_1}\wt M_t)(\pa_t m_t)m_t^{-1}\chi_-(t^{\epsilon_1}\wt M_t) m_t^{\epsilon_2/2}+\vO(r^{\epsilon_1/\lambda+\epsilon_2-2}).
\end{align*}
In particular it follows that  $S_1=  \vO(t^{-1_+})+\vO(r^{-1_+})$, conforming  with \eqref{eq:negHeis}.

Motivated by Step III we let
\begin{align}\label{eq:ST}
  \begin{split}
  S^{\rm clas}_2&=-2\epsilon_1t^{-1}f(\wcH_t)\bar \chi_1(t^{\epsilon_1}\wt M_t) 
       m_t^{\epsilon_2} \bar \chi_1(t^{\epsilon_1}\wt M_t) f(\wcH_t)+4\lambda 
                 t^{\epsilon_1-1}S;\\
                S&=T\Re \parbb{x/t^\lambda\cdot (\nabla^2
                   m )(x/t^{\lambda})  p} T^*, \\& \q T=f(\wcH_t)\bar \chi_2(t^{\epsilon_1}\wt M_t) 
       m_t^{\epsilon_2/2} \t f(\wcH_t),\q T^*=
  \t f(\wcH_t)m_t^{\epsilon_2/2}
    \bar \chi_2(t^{\epsilon_1}\wt M_t) f(\wcH_t).
      \end{split}
\end{align} Clearly the first term of $S^{\rm clas}_2$ is negative, while the second does not have a sign and it is `too big' when  comparing with \eqref{eq:negHeis} (it will be  treated in Step VI). We claim that  
\begin{align}\label{eq:S2}
  \begin{split}
&S_2-S^{\rm clas}_2\\&=t^{\epsilon_1-1}\vO\parb{r^{\epsilon_2-1}}
                            +
t^{2\epsilon_1-1}\vO\parb{r^{\epsilon_2+\delta
                            -1}}\\&\qq \qq + \vO(r^{-2\delta +\epsilon_1/\lambda+\epsilon_2})+\vO(r^{-2+\delta +2\epsilon_1/\lambda+\epsilon_2})\\
                       &=  \vO(t^{-1_+})+\vO(r^{-1_+}).    
  \end{split}
\end{align} Obviously the second bound  is a consequence of the first one  for sufficiently small $\epsilon_1>\epsilon_2>0$. In turn the first bound   is roughly obtained by arguing as  above. In particular  the first  and second  terms essentially  arise from commutation and symmetrization as done above.  However   there is an additional issue
(reflected by the presence of the third and fourth terms) in that the operator
$\wt M_t$ contains factors  of $\t f(\wcH_t)$ and the time-derivative of
those are not included  in $S^{\rm clas}_2$. (We included only the
contribution from $\t f(\wcH_t)\parb{\pa_t   (t^{\epsilon_1}M_t )}\t f(\wcH_t)$.) Hence we
need to check that  the contributions  from  $\parb{\pa_t  \t
  f(\wcH_t)}M_t\t f(\wcH_t)$ and its adjoint 
expression conform with the error terms of \eqref{eq:negHeis}. To do this we  need the  following    slight
refinement of the above calculation (used with $f$ replaced by $\t f$) 
\begin{align}\label{eq:tildeDeriv}
  \begin{split}
  \pa_t  \t f(\wcH_t)&=\Re\parbb{\parbb{-2\delta(\pa_t
      m_t)m_t^{-1}\wcH_t +(\pa_t V(t))m_t^{-2\delta}}\t
    f'(\wcH_t)}\\&\qq\qq+\vO(r^{0})m_t^{-\delta -2} +\vO(r^0)m_t^{-3\delta}. 
  \end{split}
\end{align} Inserted into $\parb{\pa_t  \t f(\wcH_t)}M_t\t f(\wcH_t)$,
and then in turn into $S_2$,  the second and the third terms   of
(\ref{eq:tildeDeriv}) contribute with a term on the form
$\vO(r^{-2\delta +\epsilon_1/\lambda+\epsilon_2})$. Here we use that
$m_t^{-\delta}M_t\t f(\wcH_t)=\vO(r^0) $, cf. \eqref{eq:bndM}. For the
first term we use after commutation that $f\t f'=0$. Hence the total
contribution to  $S_2$ from this term involves commutators improving
its decay. More precisely the total additional error term arising this way is on the form
$\vO(r^{-2+\delta +2\epsilon_1/\lambda+\epsilon_2})$. We have justified
\eqref{eq:S2} and  conclude that   $S_2= S^{\rm clas}_2+\vO(t^{-1_+})+\vO(r^{-1_+})$. Hence, comparing with our goal of proving \eqref{eq:negHeis},  we can record that
\begin{align}\label{eq:TIME}
  \begin{split}
 & S_2\leq  4\lambda 
                 t^{\epsilon_1-1}S +\vO(t^{-1_+})+\vO(r^{-1_+}).  
  \end{split}
\end{align}

 Finally we  let
 \begin{align*}
   S^{\rm clas}_3=S_3=f(\wcH_t)\chi_-(t^{\epsilon_1}\wt M_t)\parb{\pa_t m_t^{\epsilon_2}}\chi_-(t^{\epsilon_1}\wt M_t)f(\wcH_t).
 \end{align*} As in Step IV we can bound $S_3= t^{\lambda-1} \vO(r^{ \epsilon_2-1}) =  \vO(t^{-1_+})+\vO(r^{-1_+})$.

 In conclusion we  record  that except for the first term to the right in
 \eqref{eq:TIME},  the exact contribution from $\pa_t \Psi$  in \eqref{eq:negHeis}  conforms with the estimate.

 \Step {VI} It remains to see  how  the term $\i
 [H(t),\Psi]$ in \eqref{eq:negHeis}  in combination with (\ref{eq:TIME}) conforms with the
 estimate. This will finish the proof of  \eqref{eq:negHeis}.

Our first task is to justify the computation
\begin{equation}\label{eq:posComm2}
  \i[H(t),\Psi]=\i \big [m_t^\delta\wc f\parb{\wcH_t} m_t^\delta,\Psi\big ]+\vO(r^{-1_+}).
\end{equation} It is convenient in the remaining part of the proof to abbreviate  as 
\begin{equation*}
  m_t=m, \q M_t=M, \q \wt M_t=\wt M , \q \wcH_t=\wcH\,\mand \,H(t)=H,
\end{equation*}
although  this  notational  simplification  is somewhat
abusive (the previous meaning was  $m=m_{t=1}$ and $M=M_{t=1}$). All estimates below will be uniform in $t\geq 1$.

To obtain  \eqref{eq:posComm2}  we claim that  
\begin{equation}\label{eq:comH}
  m^\delta\wc f\parb{\wcH}  m^\delta\Psi=H \tilde{ f}^2\parb{\wcH}\Psi+\vO(r^{-1_+})=H\Psi+\vO(r^{-1_+}).
\end{equation} In fact using \eqref{eq:ExpsBreveH} and \eqref{eq:ExpsBreveHb}
\begin{align*}
  Hm^{-\delta}[&\tilde{f}^2\parb{\wcH},m^\delta]f\parb{\wcH}=Hm^{-\delta}
  (\tilde{f}^2)^{(1)}\parb{\wcH}\big [{\wcH},m^\delta\big ]f\parb{\wcH} +\vO(r^{-2})\\
&=Hm^{-\delta}
   (\tilde{f}^2)^{(1)}\parb{\wcH}\big [ -\i \delta m^{-(\delta+1)/2} \wt M m^{-(\delta+1)/2},f\parb{\wcH}\big ]
  +\vO(r^{-2})\\&= \vO(r^{-2})= \vO(r^{-(1+\epsilon_2)_+}).
\end{align*} This shows \eqref{eq:comH},  and we conclude  \eqref{eq:posComm2}.

It remains to compute and estimate the first term to the right in \eqref{eq:posComm2}. It may be  expanded as   a sum of terms  $T_1$,  $T_2$ and
$T_3$  (being different from those of \eqref{eq:Leib})
\begin{align*}
  \i\big[m^\delta\wc f\parb{\wcH} m^\delta,\Psi\big]&= T_1+T_2+T_3;\\
T_1&= \i \big[m^\delta,\Psi\big]\wc f\parb{\wcH} m^\delta,\\
T_2&=T_1^*= m^\delta\wc f\parb{\wcH} \i \big [m^\delta,\Psi\big],\\
T_3&=  m^\delta  \i\big [\wc f\parb{\wcH} ,\Psi\big ]  m^\delta. 
\end{align*}
For $T_1$ this leads similarly to  
  computing 
  \begin{align*}
    \i \big[m^\delta,\Psi\big]&=\i
                        \big[m^\delta,f(\wcH)\chi_-(t^{\epsilon_1}\wt M)m^{\epsilon_2}\chi_-(t^{\epsilon_1}\wt M)f(\wcH) \big]=S_1+S_2+S_3+S_4+S_5;\\
&S_1=\i \big[m^\delta,f\parb{\wcH}\big]\chi_-(t^{\epsilon_1 }\wt M)m^{\epsilon_2}\chi_-(t^{\epsilon_1}\wt M)f\parb{\wcH},\\
&S_2=S_1^*=f\parb{\wcH}\chi_-(t^{\epsilon_1}\wt M)m^{\epsilon_2}\chi_-( t^{\epsilon_1 }\wt M)\i
                        \big [m^\delta,f\parb{\wcH} \big ],\\
&S_3=f\parb{\wcH}\Re{\parb{\i
  \big[m^\delta,\chi^2_-(t^{\epsilon_1 }\wt M)\big]m^{\epsilon_2}}}f\parb{\wcH},\\
      &S_4=f\parb{\wcH}\Re{\parb{\chi_-(t^{\epsilon_1}\wt M)\big[m^{\epsilon_2},\i
        \big[m^\delta,\chi_-(t^{\epsilon_1 }\wt M)\big]\big]}}f\parb{\wcH},\\
    &S_5=f\parb{\wcH}\Re{\parb{\i
  \big[m^\delta,\chi_-(t^{\epsilon_1 }\wt M)\big]\big[m^{\epsilon_2},\chi_-(t^{\epsilon_1}\wt M)\big ]}}f\parb{\wcH},
  \end{align*} By \eqref{eq:ExpsBreveH}
  \begin{align*}
    S_1=
  -\delta f'(\wcH) m^{-(\delta+1)/2}\wt M m^{-(\delta+1)/2}\chi_-(t^{\epsilon_1 }\wt M)m^{\epsilon_2}\chi_-(t^{\epsilon_1}\wt M)f\parb{\wcH}+\vO(r^{\epsilon_2-\delta-2}).
  \end{align*} Thanks to \eqref{eq:ExpsBreveH} and  \eqref{eq:ExpsB} we see that  its contribution to $T_1$ is on the form 
  \begin{align*}
    -\delta f'\big(\wcH\big) m^{\epsilon_2-1}\wt M
    \chi^2_-(t^{\epsilon_1 }\wt M)f\parb{\wcH}\wc f\parb{\wcH}  +\vO(r^{-1_+}).
  \end{align*} Similarly $S_2$  contributes  to $T_1$ by
 \begin{align*}
    -\delta f\parb{\wcH}\chi^2_-(t^{\epsilon_1 }\wt M)\wt M
    m^{\epsilon_2-1}f'\big(\wcH\big) \wc f\parb{\wcH}  +\vO(r^{-1_+}).
 \end{align*}

 By a slight modification of \eqref{eq:Exps}, applied
  with $s=\delta$ and  $K=2$, the operator  $S_3$  is the real part of the expression 
\begin{align}\label{eq:mainS3}
  \begin{split}
    -2{\delta} t^{\epsilon_1} &f\parb{\wcH}|\nabla m|^2
    m^{\delta-1}(\chi^2_-)'(t^{\epsilon_1 }\wt M)
    m^{\epsilon_2}f\parb{\wcH} \\& \,-\tfrac {\i}2 t^{2\epsilon_1} f\parb{\wcH}\ad
  ^2_{\wt M}(m^\delta)(\chi^2_-)^{(2)}(t^{\epsilon_1 }\wt M) m^{\epsilon_2}f\parb{\wcH} \\&\q\q\q\q   +\vO(r^{3\epsilon_1/\lambda+ \epsilon_2+\delta-3})+\vO(r^{2\epsilon_1/\lambda+ \epsilon_2-2}).  
  \end{split}
\end{align}  The last term needs explanation. It arises from ${\i[m^\delta,  \t
  f(\wcH)]}M\t f(\wcH)$  and its adjoint 
expression  whose contributions are  handled by    \eqref{eq:ExpsBreveH} and \eqref{eq:ExpsBreveHb}. Our application is similar to how we above used \eqref{eq:tildeDeriv}.  Hence after commutation (with an error specified by the  last term) we use  the fact that $f\t f'=0$. Note also that for deriving    the expression (\ref{eq:mainS3})  the related property $f\t f=f$ is used. In conclusion our  calculus shows that indeed  the real part of \eqref{eq:mainS3} is a valid formula for  $S_3$.

Substituting into $T_1$ amounts to multiplying  by the factor $\wc f\parb{\wcH}m^{\delta}$ from the right,  which raises the order of the error by $\delta$. The third and fourth terms of \eqref{eq:mainS3} conform with the error terms of \eqref{eq:negHeis}. Since  the second  term of (\ref{eq:mainS3}) tends to be purely imaginary it also  only contributes with  such error terms. Hence it is only the real part of the first   term of (\ref{eq:mainS3}) 
 that needs further examination.

For the same reason the terms $S_4$  and $S_5$    contribute to $T_1$  with the errors of  \eqref{eq:negHeis}.   Taking also  into account that $T_2=T_1^*$  we finally   conclude  after  symmetrization (using Lemma \ref{lemma:expansion} and \eqref{eq:ExpsB}) that
\begin{align*}
  T_1+&T_2=-4\delta\Re \parbb{ f'\big(\wcH\big) m^{\epsilon_2-1}\wt M \chi^2_-(t^{\epsilon_1 }\wt M)f\parb{\wcH}\wc f\parb{\wcH}\\
&\quad\quad\quad+t^{\epsilon_1} f\parb{\wcH}\Re \parb{|\nabla m|^2
    m^{\delta-1}(\chi^2_-)'(t^{\epsilon_1 }\wt M)
    m^{\epsilon_2}}f\parb{\wcH} \wc f\parb{\wcH}m^\delta}   +\vO(r^{-1_+})\\
&=4\delta m^{(\epsilon_2-1)/2}\bar f(\wcH) \,
  \wt M\chi^2_-(t^{\epsilon_1 }\wt M)\,\bar f(\wcH)
  m^{(\epsilon_2-1)/2}\\
&\quad\quad\quad -4\delta t^{\epsilon_1} \Re \parbb{ f\parb{\wcH}(\chi^2_-)'(t^{\epsilon_1 }\wt M)|\nabla m|^2
    m^{2\delta+\epsilon_2-1}f\parb{\wcH} \wcH} +\vO(r^{-1_+}).
\end{align*}
The first term is bounded from above by $C t^{-\epsilon_1}m^{\epsilon_2-1}$, which conforms with \eqref{eq:negHeis}.

The second term will cancel
with a term in $T_3$ computed using the following ingredients.
\begin{align*}
   \i \big[\wc f\parb{\wcH} ,\Psi\big]&=f\parb{\wcH}\i
                        \big[\wc f\parb{\wcH} ,\chi_-(t^{\epsilon_1}\wt M)m^{\epsilon_2}\chi_-(t^{\epsilon_1}\wt M) \big]f\parb{\wcH}\\
&=-2t^{\epsilon_1 }f\parb{\wcH}\bar \chi_2(t^{\epsilon_1}\wt M_t) m^{\epsilon_2/2} {\i
                        \big[\wc f\parb{\wcH},
  \wt M\big]m^{\epsilon_2/2}}\bar \chi_2(t^{\epsilon_1}\wt M_t) f\parb{\wcH}\\& \qq  +f\parb{\wcH}\chi_-(t^{\epsilon_1}\wt M)\i
                        \big[\wc f\parb{\wcH} ,m^{\epsilon_2}\big]\chi_-(t^{\epsilon_1}\wt M) f\parb{\wcH}+\vO(r^{-(1+2\delta)_+}).
\end{align*} Here we used \eqref{eq:ExpBreveH} for  $\wc f$ (replacing $f$ there by $\wc f$)  and with
 $\chi(M_t)$ there replaced by $\chi_-(t^{\epsilon_1}\wt M) $. We also used that $\i\,\ad ^2_{\wt M}(\wc f\parb{\wcH})=\vO(r^{-2 })$ 
 is purely imaginary (resulting in improved decay when symmetrizing).
 
In turn 
\begin{align*}
\i
                        \big[\wc f\parb{\wcH},
  \wt M\big]&=R_1+R_2 +R_3+\vO(r^{-(1+2\delta+\epsilon_2)_+});\\
&R_1= \tilde{f}'(\wcH) \i\big[\wcH,\wt M\big]\wcH\tilde{f}(\wcH) +\tilde{f}(\wcH) \wcH\i\big[\wcH,\wt M\big]\tilde{f}'(\wcH) , \\
& R_2=   \tilde{f}^2(\wcH)
  \i[m^{-\delta}, M]H m^{-\delta}\tilde{f}^2(\wcH) +\tilde{f}^2(\wcH)
  m^{-\delta}H \i[m^{-\delta}, M]\tilde{f}^2(\wcH) ,\\
&R_3=   \tilde{f}^2(\wcH)  m^{-\delta}\i[H, M]m^{-\delta}\tilde{f}^2(\wcH),
\end{align*} and
\begin{align*}
\i
                        \big[\wc f\parb{\wcH},
  m^{\epsilon_2}\big]&=R_4+R_5 +\vO(r^{-(1+2\delta)_+});\\
&R_4= \tilde{f}'(\wcH) \i\big[\wcH,m^{\epsilon_2} \big]\wcH\tilde{f}(\wcH) +\tilde{f}(\wcH) \wcH\i\big[\wcH,m^{\epsilon_2}\big]\tilde{f}'(\wcH) , \\
&R_5=   \tilde{f}(\wcH)  m^{-\delta}\i[H,m^{\epsilon_2}]m^{-\delta}\tilde{f}(\wcH).
\end{align*}

The contributions  from $R_1$ and $R_4$  to $T_3$ are on the form  $\vO(r^{-1_+})$
(since $f\tilde{f}'=0$), and  the contribution from $R_2$ to $T_3$ 
indeed cancels with the above second term of  $T_1+T_2$. 

After a
symmetrization the   contribution  from $R_3+R_5$ to $T_3$
is  on the form 
\begin{align*}
  &-2t^{\epsilon_1 }{f}\big(\wcH\big)\bar \chi_2(t^{\epsilon_1}\wt M_t)\,
    m^{\epsilon_2/2}\,\tilde{f}^2\big(\wcH\big)\i [p^2, M]\tilde{f}^2\big(\wcH\big)\,m^{\epsilon_2/2} \,\bar \chi_2(t^{\epsilon_1}\wt M_t)f\parb{\wcH}\\&+
  \epsilon_2m^{(\epsilon_2-1)/2}{f}\big(\wcH\big)\,
    \chi_-(t^{\epsilon_1}\wt M)  \wt M\chi_-(t^{\epsilon_1}\wt M) f\parb{\wcH}m^{(\epsilon_2-1)/2}+\vO(r^{-1_+}).
\end{align*}
The middle  term is bounded from above by $C
t^{-\epsilon_1}m^{\epsilon_2-1}$ (which conforms with \eqref{eq:negHeis}). 

For the first term we can  replace $\i [p^2, M]$ with $4t^{-\lambda}{p\cdot (\nabla^2
  m )(x/t^{\lambda}) p}$ and replace each  factor of $\tilde{f}^2\big(\wcH\big)$ by a (single)  factor
of $\tilde{f}\big(\wcH\big)$. When combining the resulting expression with 
                \eqref{eq:TIME}, recalling  the operators $S$ and $T$ from \eqref{eq:ST}  and using  as well  the operator $B=B_t$ in \eqref{eq:Heis2}, we finally  obtain (by completing the square)  the upper bounds
                \begin{align}\label{eq:FinalNeg}
                  \begin{split}
                  &\bD \Psi \leq -2t^{\epsilon_1}TB^*BT^*+\vO(t^{-1_+})+\vO_{\unif}(r^{-1_+})\\
                  &\qq \leq \vO(t^{-1_+})+\vO_{\unif}(r^{-1_+}).   
                  \end{split}
                \end{align}
                We have justified
 \eqref{eq:negHeis}, and  therefore   \eqref{eq:PS3}
 is proven.
\end{proof}
\begin{remarks}\label{remark:bound-incom-regi}
  \begin{enumerate}[i)]
  \item\label{item:1psi} In comparison with (\ref{eq:Leib2}) note the
    appearance of   factors of $\tilde{f}\big(\wcH_t\big)$ in \eqref{eq:FinalNeg}. Locally uniformly in $t\geq 1$  we can bound
    \begin{equation*}
      A:=m_t^{\epsilon_2/2}\tilde{f}\big(\wcH_t\big)B_t^*B_t\tilde{f}\big(\wcH_t\big)m_t^{\epsilon_2/2}=\vO(r^{\epsilon_2+2\delta-1}),
    \end{equation*}
    and no better. In particular $A$  is not a bounded operator. We
    can approximate it  by a  sequence of operators
    $A_n\nearrow A$, for example  taken as  
    \begin{equation*}
      A_n=m_t^{\epsilon_2/2}\tilde{f}\big(\wcH_t\big)B_t^*\chi^2_-(r/n)B_t\tilde{f}\big(\wcH_t\big)m_t^{\epsilon_2/2};\q n\in \N.
    \end{equation*} This kind of approximation will be   useful for deriving integral propagation estimates in the bounded case  (which  roughly corresponds  to letting   $\epsilon_2=0$ in the above proof).  See Corollary \ref{cor:smoothEst9} and its proof, and  see also Remarks \ref{remarks:intepretation} and the proof of Corollary \ref{cor:smoothEst} for a closely related procedure.

Although it is possible to extract integral propagation estimates in the unbounded case considered  in the above  somewhat complicated proof of Lemma \ref{lemmaNegForbidden} (in the proof we just disregarded negative terms), we are not going to need them.   Rather we are in the coming sections (in particular from Section \ref{sec:Proof of {eq:Strength}} and onwards)  going to need  various related estimates with simplified proofs. Indeed the  scheme of proof is essentially the same. To avoid boring repetition of arguments we will prefer only to elaborate on    different aspects of the   proofs.   
  \item \label{item:2psi} It is easily checked that the right-hand side of (\ref{eq:FinalNeg}) can be replaced by
    \begin{equation*}
      \vO(t^{-1-\varepsilon})+\vO(t^{-2\varepsilon})\vO_{\unif}(r^{-1})
    \end{equation*}
    for some $\varepsilon>0$ (depending on the involved parameters).
    \item \label{item:2psibb} Although we (possibly) fixed $\delta=4/7$, and hence   close to $1/2$, this is not needed for  the above proof of Lemma \ref{lemmaNegForbidden} (it  is needed later). In fact the natural limit for the above proof is $\delta< 1$ (and not $\delta<3/5$). This is thanks to a quantum symmetrization effect. Hence formally calculating the error term of
      \begin{align*}
        \i[p^2,\chi_-(M_t)]+ \sq{-\chi'_-(M_t)}\i[p^2,M_t]\sq{-\chi'_-(M_t)}\approx 0,
      \end{align*}
  it is  expressed in terms of $\ad^3_{M_t}(p^2) \approx p\cdot \vO(r^{-3})p$. Writing
\begin{align*}
  p\cdot \vO(r^{-3})p=\parb{p m_t^{-\delta}}\cdot m_t^{2\delta}\vO(r^{-3})\parb{m_t^{-\delta}p},
\end{align*} we see that when sandwiched between factors of $f(\wcH_t)$ the effective order of the error is $2\delta-3<-1$, which counts for  `integrability'. Note that without the symmetrization the error is expressed  in terms of $\ad^2_{M_t}(p^2)\approx \vO(r^{2\delta-2})$ which is insufficient for integrability  for any  $ \delta >1/2$.
\end{enumerate} 
\end{remarks}

\subsection{Proof of Proposition  \ref{lemma:scatDef22}}\label{subsec:Proof of ref{lemma:scatDef22}} 
We use   \eqref{eq:PS30} and the property
\begin{equation}\label{eq:minimalV}
  F(|x|< R)\chi_+(t^{\epsilon_1}M_t)=0\text{  for all large } t,
\end{equation}  estimating  for \eqref{eq:decayInx2}   as 
\begin{align*}
  \lim_{t\to \infty}\, \norm{F(|x|< R)\psi(t)}&= \lim_{t\to \infty}\, \norm{F(|x|< R)\parb{\chi^2_-(t^{\epsilon_1}M_t)+\chi^2_+(t^{\epsilon_1}M_t)}\psi(t)}\\
  &\leq   \lim_{t\to \infty}\, \norm{\chi_-(t^{\epsilon_1}M_t)\psi(t)}=0.
\end{align*} 

Since (\ref{eq:minimalV}) holds with $R=t^\lambda/8$ and $\lambda$ can be taken arbitrarily close to $1$ (possibly requiring $\epsilon_1>0$ taken small), also \eqref{eq:decayInx22} holds.
\qed

   \section{Localization to  the strictly outgoing  region}\label{sec:Localization to  outgoing  region}
   We recall that the constructions of Subsection \ref{subsec:
  Homogeneous Yafaev type functions} depend on a small positive
parameter $\epsilon$. As in Section \ref{sec:Preliminary localization results} this is fixed, and we are going to use the
(derived) quantities  (\ref{eq:fes}) and (\ref{eq:Mtilde}) as well as the quadratic form $B^*_tB_t$ of  (\ref{eq:Heis2}). In addition we will use the function $\bar f$ introduced in \eqref{eq:barf} (it will appear similarly).

   \begin{corollary}\label{cor:smoothEst9}  
     Let $g\in C^\infty_{\c}(\T\setminus \vT_\p)$ and  $\chi\in \vF_+$   be given. Then,  using the notation
 $\psi(t)=U(t)g(U(1))\psi$ for  any $\psi\in \vH$,  the following bounds hold.
 \begin{subequations}
 \begin{align}\label{eq:Smoothbnd19}\begin{split}
 \forall \,\psi\in \vH&:\quad \int_{1}^\infty
\norm[\big]{Q_{1} \psi(t)}^2\,\d t \leq C\|\psi\|^2 ;\\
Q_{1}=&Q_{1,t}=2\sqrt{\wt M_t\chi\parb{\wt M_t}}\,\bar f\big(\wcH_t\big)
        \,m_t^{-1/2}.     
        \end{split}
\end{align} 
\begin{align}\label{eq:Smoothbnd29}
  \begin{split}
  \forall \,\psi\in \vH&:\quad \int_{1}^\infty
\norm[\big]{Q_{2}\psi(t)}^2\,\d t \leq C\|\psi\|^2 ;\\
Q_{2}=Q_{2,t}&=B_t\,\tilde{f}(\wcH_t)\sqrt{\chi'}\parb{\wt M_t}f\parb{\wcH_t}.  
  \end{split}
\end{align}  
 \end{subequations}
\end{corollary}
\begin{proof} By density and by using Remark
  \ref{remark:bound-incom-regi} \ref{item:1psi} and    Lebesgue's
  dominated and monotone convergence theorems  we can for the
  bounds  \eqref{eq:Smoothbnd19} and \eqref{eq:Smoothbnd29}  assume that $\psi\in \vS(\bX)$.  (Alternatively we may proceed as in the proof of Corollary \ref{cor:smoothEst} given in Subsection \ref{subsec:Preliminary integral estimates}.)

  Introducing  
 the  uniformly  bounded family of observables
\begin{subequations}
\begin{equation}
  \label{eq:PS19}
  \Psi=\Psi_t=f(\wcH_t)\chi\parb{\wt M_t}f(\wcH_t),\q t\geq 1,
\end{equation}  we can compute
\begin{align}
  \label{eq:posBND9}
   \exists \, \varepsilon >0:\qq \bD \Psi= \delta Q_{1}^*Q_{1}+ Q_{2}^*Q_{2}+\vO(t^{-1-\varepsilon})+ t^{-2\varepsilon}\vO_{\unif}(r^{-1}).
\end{align} 
\end{subequations}

In fact \eqref{eq:posBND9} follows by mimicking the proof of Lemma \ref{lemmaNegForbidden}  (noting also  Remark \ref{remark:bound-incom-regi} \ref{item:2psi}). In the present case, comparing with  (\ref{eq:PS2}), $\epsilon_1=\epsilon_2=0$ and we have replaced the product of the factors of $\chi_-$ by $\chi$,  clearly  simplifying  the computation, cf. Remark \ref{remark:bound-incom-regi} \ref{item:1psi}. We leave out the details of this computation.
Note  that in comparison with \eqref{eq:negHeis} the Heisenberg derivative \eqref{eq:posBND9}  is essentially positive (while  in \eqref{eq:negHeis} it is essentially  negative).

Clearly the bounds \eqref{eq:Smoothbnd19} and \eqref{eq:Smoothbnd29}
  for  $\psi\in \vS(\bX)$  follow by integration of $\tfrac{\d}{\d
    t}\inp{\Psi}_{\psi(t)}=\inp{\bD \Psi}_{\psi(t)}$  using the
  boundedness of $\Psi$, \eqref{eq:posBND9} and \eqref{eq:LocalDecay}.
  \end{proof}

  The factor $g\in C^\infty_{\c}(\T\setminus \vT_\p)$ in Corollary
  \ref{cor:smoothEst9} was
  used to treat the fourth term of \eqref{eq:posBND9}  by 
\eqref{eq:LocalDecay}.  However even without this localization there is still some non-trivial application of \eqref{eq:posBND9}. This is the  following  smoothing type result (recall that $Q_2$ is not bounded).
\begin{corollary}
  \label{cor:local-strictly-genG} Let  $\chi\in \vF_+$   be given.  There exists $C>0$ such that for any  subinterval $[a,b]\subseteq [1,\infty)$   the following bounds hold for  $\varphi(t)=U(t)\varphi$ (and with $Q_{j}=Q_{j,t}$, $j=1,2$,  given as in \eqref{eq:posBND9}).
  \begin{subequations}
 \begin{align}\label{eq:Smoothbnd19G2}
   \forall \,\varphi\in \vH&:\quad \delta \int_{a}^b
                           \norm[\big]{Q_{1} \varphi(t)}^2\,\d t \leq \parb{C(b-a)+\sup \chi}\|\varphi\|^2 ,\\
   \forall \,\varphi\in \vH&:\quad \int_{a}^b
\norm[\big]{Q_{2}\varphi(t)}^2\,\d t \leq \parb{C(b-a)+\sup \chi}\|\varphi\|^2 .\label{eq:Smoothbnd29G2}
        \end{align} 
 \end{subequations}
  \end{corollary}
  \begin{proof} From \eqref{eq:posBND9}  it follows that
    \begin{align}\label{eq:bndQ12}
   \bD \Psi\geq  \delta Q_{1}^*Q_{1}+ Q_{2}^*Q_{2}-C.
    \end{align} Since $0\leq {\Psi}\leq \sup \chi$,  integrating $\inp{\bD\Psi_t}_{\varphi(t)}$ yields the bounds  with this $C$.
    \end{proof}
\begin{lemma}\label{lemma:asympOBser}
For any $h\in C _\c (\R)$ and $\psi\in \vH_{\ac}$ there exists the   limit  
\begin{equation}
  \label{eq:stro}
  h(M^+) \psi:=\vHlim_{t\to \infty} \,U(t)^{-1}h\parb{\wt M_t}\psi(t) \in\vH_{\ac};\q \psi(t)=U(t)\psi.
\end{equation} We can write
\begin{equation}\label{eq:resolution0}
  \inp{\psi, h(M^+)\psi}=\int_{\R}\, h(x)\,\vE^+_\psi(\d  x)
\end{equation} in terms of  a (unique) resolution of the identity $\vE^+(\cdot)$ on  $\R$ (taking values in the orthogonal projections on $\vH_{\rm ac}$). For any Borel subset  $B\subseteq \R$ the orthogonal projection $\vE^+(B)$ commutes with the mo\-nodromy operator   
$U(1)$. The support $\supp (\vE^+)\subseteq[0,\infty)$, and on the subspace
$\vH^+:= \vE^+([0,\infty))\vH_{\rm ac}$   one obtains  a   positive operator $M^+$  by taking $h(\lambda)=\lambda$  in (\ref{eq:resolution0}) (hence $\vE^+(\cdot)$ is the spectral resolution of resolution of $M^+$). Finally  this  subspace $\vH^+$ includes  $\vH^+_{\ener}$.
\end{lemma}
\begin{proof}  First we show that there exists the  limit
  \begin{equation}
  \label{eq:stro0}
  h(M^+) \psi=\vHlim_{t\to \infty} \,U(t)^{*}h\parb{\wt M_t}\psi(t).
\end{equation}
This will be in three steps.
\Step {I} We reduce the problem to showing the well-definedness of the below limit \eqref{eq:stro2}. By density we can assume that $\psi=g(U(1))\br\psi$  with  
$g\in C^\infty_{\c}(\T\setminus \vT_\p)$ and $\br\psi\in \vH$ (or possibly $\br\psi\in \vS(\bX)$).  Similarly we can assume that  $h$ is 
  real-valued, smooth and constant in a neighbourhood of zero. We
  decompose 
  \begin{equation*}
    h-h(0)=\parb{h-h(0)}1_{(-\infty, 0)}  +\parb{h-h(0)}1_{[0,\infty)}=:h_-+h_+.
  \end{equation*}
By \eqref{eq:PS3} it suffices to show the existence of the limit
\begin{equation*}
  h_+(M^+) \psi=\vHlim_{t\to \infty} \, U(t)^{*}h_+\parb{\wt M_t}\psi(t).
\end{equation*}
  By invoking Remark
\ref{remarks:chiSMooTH} \ref{item:invarScatsmooth} it  suffices in turn to show the existence of the limit
\begin{equation*}
 \vHlim_{t\to \infty} \,U(t)^{*}\chi\parb{\wt M_t} \psi(t) ; \q \chi\in \vF_+.
\end{equation*} In terms of the observable $\Psi_t$ from (\ref{eq:PS19}) we need to show, thanks to  Lemma \ref{lemmaHighEnergy} and a   commutation (cf. \eqref{eq:ExpBreveH}), that there exists
\begin{align}\label{eq:stro2}
  \chi(M^+) \psi=\lim_{t\to \infty}\,
    U(t)^{*}\Psi_t \psi(t).
\end{align}

\Step {II} So let $\psi$ be given as above and let $\chi\in \vF_+$. 
Pick    functions $\tilde{g}, g_1\in C^\infty_{\c}(\T\setminus
\vT_\p)$ with $\t
  g\succ g_1 \succ g$.
  We will now   prove that
  \begin{align}\label{eq:BBB00}
    \lim_{t\to \infty}\,\norm[\big]{\parb{1-\tilde{g}(U(1))}U(t)^{*}\Psi_t \psi(t) }=0.
  \end{align}

  Recalling  \eqref{eq:groupLike}
  it  suffices to show
  that 
  \begin{align}\label{eq:BBB0}
    \lim_{t\to \infty}\,\norm[\Big]{
   \parbb{U(t-[t])^*\Psi_t U(t-[t])g_1(U(1))- g_1(U(1))\Psi_t }\psi([t])}=0.
  \end{align}

 We give   a proof of (\ref{eq:BBB0}) using for $g_1$ only the property that $g_1\in C(\T)$. In fact 
by the
  Stone--Weierstrass theorem
  we can  replace the factors of  $g_1(U(1))$ by an arbitrary integer power $U(1)^k$. Letting  $\sigma:=t-[t]+k$ we write (using in the last step that $H(s)=H(s+[t]-k)$) 
\begin{align*}
  U(1)^{-k}&U(t-[t])^*\Psi_t U(t-[t])U(1)^k-\Psi_t \\
  &= U(\sigma)^*\Psi_{\sigma+[t]-k} U(\sigma)-\Psi_{[t]-k}    +\parb{\Psi_{[t]-k} -\Psi_t}\\
  &= \int_0^\sigma\, U(s)^*{(\bD \Psi)_{s+[t]-k}}U(s)\,\d \s +\vO(t^{-0_+}).
\end{align*}   We need to show that
\begin{align*}
    \lim_{t\to \infty}\,\norm[\Big]{
   \int_0^\sigma\, U(s)^*{(\bD \Psi)_{s+[t]-k}}U(s)\psi([t])\,\d \s} =0.
\end{align*} By \eqref{eq:posBND9}  it suffices to show that
\begin{align}\label{eq:BBB03}
    \lim_{t\to \infty}\,\norm[\Big]{
   \int_0^\sigma\, U(s)^*Q^*_{j, s+[t]-k}Q_{j, s+[t]-k}\psi(s+[t])\,\d \s} =0;\q j=1,2.
\end{align} We can bound the integral by first taking inner product with any $\varphi\in \vH$ and then apply  \cas and  Corollaries \ref{cor:smoothEst9} and \ref{cor:local-strictly-genG} (as well as   affine  changes of variables). We can bound the integral this way as $o(t^0)\norm{\varphi}$ as $t\to \infty$. This  proves  (\ref{eq:BBB03}) and  therefore also \eqref{eq:BBB00}.

  \Step {III} Now  we can  show the existence of the limit \eqref{eq:stro2}. 
    Thanks to  Step II we are left with proving 
the existence of the limit
\begin{align}\label{eq:varphiLimi5}
  \psi_\infty=\lim_{t\to \infty}\,
    \t g(U(1))U(t)^{*}\Psi_t \psi(t);\q \psi(t)=U(t)\psi.
\end{align} Recall that $\psi=g(U(1))\br\psi$  with  
$\br\psi\in \vH$.

We consider 
 the evolution
$\wt\varphi(t)=U(t)\t g(U(1))\varphi$ for  any given $\varphi \in \vS(\bX)$. We 
verify 
the  Cauchy condition,   first  writing  for any big $t_2>t_1>1$
\begin{align*}
  \inp[\big]{\wt\varphi(t_2), \Psi_{t_2}\psi(t_2)}- \inp[\big]{\wt\varphi(t_1),
  \Psi_{t_1}\psi(t_1)}=\int^{t_2}_{t_1}\,
  \tfrac{\d} {\d t}\inp{\wt\varphi(t), \Psi_{t}\psi(t)}\,\d t.
\end{align*} By \eqref{eq:posBND9}, \caS, Corollary \ref{cor:smoothEst9} and \eqref{eq:LocalDecay}  
 we can bound the integral to the right 
as $o(t_1^0)\norm{\varphi}$ as $t_1\to \infty$, uniformly in $t_2>t_1$
and  
$\varphi\in \vS(\bX)$. This yields the well-definedness of the limiting vector $\psi_\infty\in \vH$ in 
\eqref{eq:varphiLimi5}. We have justified \eqref{eq:stro0}, and by using \eqref{eq:BBB00}   also \eqref{eq:stro} follows.
\Step {IV}  
From the well-definedness of the  construction \eqref{eq:stro} it follows easily  that  the  map $ C _\c (\R)\ni h \to h(M^+)\in \vL(\vH_{\ac})$ is a continuous $*$-homomorphism and therefore it can be represented  (cf. [Ru, Theorem 12.22]) as
\begin{equation}\label{eq:resolution}
  \inp{\psi, h(M^+)\psi}=\int_{\R}\, h(x)\,\d \vE^+_\psi(x)
\end{equation} by  a unique resolution of the identity $\vE^+(\cdot)$ on  $\R$.

The support $\supp (\vE^+)\subseteq[0,\infty)$ thanks to Lemma \ref{lemmaNegForbidden}, and  hence 
$\vE^+(\cdot)$  is the spectral resolution of a positive  operator $M^+$ on $\vH^+= \vE^+([0,\infty))\vH_{\rm ac}$. 

\Step {V} We show that $\vH^+_{\ener}\subseteq \vH^+$, i.e. that 
the projection $\vE^+(\R)$ acts as  the identity operator on $\vH^+_{\ener}$.
  Consider the increasing sequence of functions $ h_n\in  C _\c (\R)$ given by $h_n(s)=\chi_-(|s|/n)$. We need to show that
\begin{equation}\label{eq:nconv}
  \inp{\psi, h_n(M^+)\psi}\to \norm{\psi}^2.
\end{equation} We compute and estimate as 
\begin{align*}
  \inp[\big]{\psi,  \parb{I-h_n(M^+)}\psi}&= \lim_{t\to \infty}\,\inp[\big]{U(t)\psi,  \parb{I-h_n \parb{\wt M_t}}U(t)\psi} \\
  &\leq  \limsup_{t\to \infty}\,\abs[\big]{\inp[\big]{F(p^2>E)U(t)\psi,  \parb{I-h_n \parb{\wt M_t}}U(t)\psi}} \\& \qq +\limsup_{t\to \infty}\,\abs[\big]{\inp[\big]{F(p^2< E)U(t)\psi,  \parb{I-h_n \parb{\wt M_t}}U(t)\psi}}. 
\end{align*} The first term can be taken arbitrarily small  uniformly in $n$ by picking a big $E$.  For fixed $E$ we note the uniform bounds 
\begin{equation*}
  \norm [\big]{F(p^2< E)\parb{I-h_n \parb{\wt M_t}}}\leq n^{-1}\norm [\big]{F(p^2< E)\wt M_t)}\leq n^{-1}C\norm [\big]{F(p^2< E){\inp{p}}},
\end{equation*} leading to 
\begin{equation*}
  \limsup_{t\to \infty}\norm [\big]{F(p^2< E)\parb{I-h_n \parb{\wt M_t}}}\to 0\text{ for }n\to \infty.
\end{equation*} We have shown (\ref{eq:nconv}).

\Step {VI} 
We show that the orthogonal projections $\vE^+(B)$ commute  with the mo\-nodromy operator   
$U(1)$. By the relation \eqref{eq:resolution} it suffices to show that $U(1)$ commute  with the operators $h(M^+)$, and therefore in turn that for any $\chi\in \vF_+$ and $\psi\in \vH_{\ac}$
\begin{equation}\label{eq:comU}
  [\chi(M^+),U(1)]\psi  =\lim_{\N \ni t\to \infty}\,
    [U(t)^{*}\Psi_t U(t),U(1)]\psi =0.
\end{equation}
 Noting   the validity of \eqref{eq:BBB0} with $g_1(z)=z$ and for $\psi(t)=U(t)\psi$, the result \eqref{eq:comU} follows.  
\end{proof}

\begin{lemma}\label{lemma:vanishingKernel} The positive operator $M^+ $ has  vanishes kernel, i.e.  $\vE^+(\set{0})={0}$.
\end{lemma}
\begin{proof}
Consider any $\psi=g(U(1))\br\psi$  with  
$g\in C^\infty_{\c}(\T\setminus \vT_\p)$ and $\br\psi\in \ker (M^+)=\ran(\vE^+(\set{0})$. We need  to show that $\psi=0$. Since also $\psi\in \ker (M^+)$, it follows  that 
for any $R>1$
 \begin{equation}
   \label{eq:intInf}
   0-\inp{\chi_+\parb{ R\wt M_t}}_{f(\wcH_t)\psi(t)}=\int_t^\infty\, \tfrac{\d}{\d s}\inp{\chi_+\parb{R \wt M_s}}_{f(\wcH_s))\psi(s)}\,\d s;\q \psi(t)=U(t)\psi.
 \end{equation} Here the time-derivative to the right conforms with \eqref{eq:posBND9} for almost all $s\geq 1$,
   and by the proof of \eqref{eq:posBND9} with $\chi=\chi_+(R\cdot)$ it follows that the integral has a lower bound on the form
 $-CR^3 t^{-\varepsilon}$. Taking  $R=t^{\epsilon_1}$, $3\epsilon_1 < \varepsilon$, yields
 \begin{equation}\label{eq:IS}
     \lim_{t\to \infty} \,\norm{\chi_-\parb{t^{\epsilon_1}\wt M_t}f(\wcH_t)\psi(t)}^2=\lim_{t\to \infty} \,\norm{f(\wcH_t)\psi(t)}^2=\norm{\psi}^2.
   \end{equation} The combination of (\ref{eq:PS3}) with this $\epsilon_1$ (possibly taken smaller if needed)  and  
   \eqref{eq:IS} leads to the conclusion
   that  $\psi=0$.\end{proof}

 Recall that $M_0=M_1=(M_t)_{|t=1}$.

 \begin{corollary}\label{cor:OutEst9}  
For any $\psi\in \vH_{\ac}$ the evolution $\psi(t)=U(t)\psi$ obeys
  \begin{subequations}
\begin{align}
  \label{eq:out1}
  \lim_{\kappa\searrow 0}\,\limsup_{t\to \infty }\,\norm{\chi_-( M_0/\kappa)\psi(t)}&=0,\\
  \label{eq:out2}
  \lim_{\kappa\searrow 0}\,\limsup_{t\to \infty }\,\norm{\parb{I-\chi_+( M_0/\kappa)} \psi(t)}&=0.
\end{align}   
  \end{subequations}
 \end{corollary}
 \begin{proof}
   It suffices to show \eqref{eq:out1}.
   \Step {I} We show \eqref{eq:out1} in the  case $\psi\in \vH^+$.
   Thanks to Lemma \ref{lemma:vanishingKernel} we can assume that $\psi\in \ran (\vE^+(\R_+))$ and therefore that $\psi=h( M^+)\br\psi$ for some  $h\in C^\infty_{\c}(\R_+)$ and some $\br\psi\in \vH_{\ac}$. A priori the operator $ M^+$ depends on a parameter $\lambda<1$ close to $1$. However, it actually does not, which is a consequence of the fact that 
   \begin{equation}\label{eq:similarLOC}
     \lim_{t\to \infty}\,\norm{h\parb{ \wt M_t}\br\psi(t)- h( M_0)\br\psi(t)}=0,\q \br\psi(t)=U(t)\br\psi.
   \end{equation} The latter is in turn a consequence of  \eqref{eq:decayInx22} as follows: We can in \eqref{eq:similarLOC}  replace $\wt M_t$ by $ M_t$, cf.     \eqref{eq:CONs}, \eqref{eq:T1I} and \eqref{eq:T1II}, and then we use    that
     \begin{align*}
       \lim_{t\to \infty}\,\norm[\big]{\parb{I-\chi_+\parb{|x|/t^{\lambda}}}\br\psi(t)}&=0,\\
       \parb{h(M_t)- h( M_0)}\chi_+\parb{|x|/t^{\lambda}}&=\vO(t^{-0_+}),
     \end{align*} where the second assertion  follows from a   represention as in  \eqref{eq:T1II} and  commutation using that
     \begin{equation*}
       \parb{ M_t-  M_0}\chi_+\parb{|x|/t^{\lambda}}=0.
     \end{equation*}
     
     Since the function $h$ in \eqref{eq:similarLOC} is supported in $\R_+$,  it follows that  for  all $\kappa>0$ small enough   $\chi_-( \cdot/\kappa)h=0$ and therefore also that  
     \begin{equation*}
       \lim_{t\to \infty }\,\norm{\chi_-( M_0/\kappa)\psi(t)}=0;\q \psi=h( M^+)\br\psi.
     \end{equation*} We have proven \eqref{eq:out1}.

     \Step {II} We show \eqref{eq:out1} in the general case,  $\psi\in \vH_{\ac}$. The arguments in Step I yield that 
     \begin{equation*}
       \lim_{t\to \infty }\,\norm{\chi_-( M_0/\kappa)\chi_+( \wt M_t)\psi(t)}=0;\q \kappa<1/2,\,\psi(t) =U(t)\psi.
     \end{equation*} By Lemma \ref{lemmaNegForbidden} we can then assume that $\psi=h( M^+)\psi$ for some  $h\in C^\infty_{\c}(\R)$. In particular this  $\psi\in \vH^+$,  and hence Step I applies.
\end{proof}
\begin{remarks}\label{remarks:f-outg-regi} In the next section and throughout the paper we are going to use \eqref{eq:out2} as well as a slightly different auxiliary Hamiltonian than $\wcH_t= \wcH_{t,\delta,\lambda}$ (the one given in \eqref{eq:widecheckH2}). This is denoted by $\wcH(t)=\wcH_\delta(t)$ and defined by  \eqref{eq:widecheckH}. By commutation it follows easily that for any
   $\lambda_1\in (\lambda, 1)$
     \begin{align*}
       \parbb{f(\wcH_t)- f\parb{\wcH(t)}}\chi_+\parb{|x|/t^{\lambda_1}}=\vO(t^{-0_+}).
     \end{align*} Consequently, using again \eqref{eq:decayInx22}, we can freely replace the localization factor  $f(\wcH_t)$ by the factor $f\parb{\wcH(t)}$, yielding in particular  the following refinement of \eqref{eq:CONs} and \eqref{eq:out2}.
     \begin{equation}
  \label{eq:out2b}
  \forall \psi\in \vH_{\ac}:\q\lim_{\kappa\searrow 0}\,\limsup_{t\to \infty }\,\norm[\Big]{\parbb{I-\chi_+( M_0/\kappa)f\parb{\wcH(t)}} U(t)\psi}=0. 
\end{equation}

The bound (\ref{eq:out2b})  will be the outset for further analysis in this  paper. In particular we will   henceforth  not need  the time-scaled quantities of (\ref{eq:mtime}).
\end{remarks}

\section{Preliminary propagation estimates}\label{sec:Proof of {eq:Strength}}

We recall the quantities 
$m$ and  $M=M_0$ given by
\eqref{eq:ml}  and \eqref{eq:M0l}. In this section and henceforth we
shall not use the more general quantities $m_t$ and  $M=M_t$ from
(\ref{eq:mtime}). Rather    we will   only need the time one  versions 
\begin{subequations}
 \begin{align}\label{eq:9m}
   m(x)&= f_{\rm cvx}(m_{a_0}(x)),\\
   \label{eq:92}
  M&= M_0=2\Re (\grad m\cdot p),
\end{align} where   as before $f_{\rm cvx}\in C^\infty(\R_+)$ is taken  convex such that $f_{\rm cvx}(s)=1$ for $s\leq  1/4$ and $f_{\rm cvx}(s)=s+2/3$ for $s\geq
1/2$. We recall that the small positive
parameter $\epsilon$  in the construction of $m_{a_0}$ is considered as fixed.
\end{subequations}

We consider the auxiliary Hamiltonian
\begin{equation}\label{eq:widecheckH}
  \wcH=\wcH(t)=\wcH_\delta(t) =m^{-\delta}H(t)m^{-\delta},\q \delta\in (1/2,3/5),
\end{equation}
realized as a self-adjoint operator by the Friedrich extension
procedure. This is  not the same as the auxiliary Hamiltonian $\wcH_t= \wcH_{t,\delta, \lambda}$ given in \eqref{eq:widecheckH2}, but clearly similar. 
We prefer to omit the dependence of $\wcH_\delta(t) $ on 
$\delta$ and  $t$,  and just  use the shorthand  notation $\wcH$. Our analysis,   henceforth not using the operator $\wcH_t$,  will  entirely be  based on (\ref{eq:9m})-(\ref{eq:widecheckH}) and the localization result \eqref{eq:out2b}. We will freely use the analogues of 
(\ref{eq:fes}) and \eqref{eq:barf}, viz  real-valued  time-independent functions $f,\tilde f, \wc f , \bar f\in C_\c^\infty (\R)$ obeying 
\begin{align}\label{eq:fesb}
  \begin{split}
 f\parb{\wcH}=\chi_-\parb{\wcH},&\q \tilde{ f}\parb{\wcH}=\chi_-\parb{\wcH/2},\q \wc f\parb{\wcH}=
  \tilde{ f}\parb{\wcH} \wcH \tilde{ f}\parb{\wcH},\\
  &\bar f\parb{\wcH}=\sq{-\chi_-'\parb{\wcH}{\wcH}\chi_-\parb{\wcH}}.   
  \end{split}
\end{align} 

\subsection{Integral estimates}\label{subsec:Preliminary integral estimates}
Our  proof of   Lemma \ref{lemmaNegForbidden} reveals   useful smoothness estimates similar to  Corollary \ref{cor:smoothEst9}, now to be stated  entirely (hence somewhat  simplified)  in terms of the quantities  (\ref{eq:9m})-(\ref{eq:fesb}). The following result is  somewhat supplementary to the localization result \eqref{eq:out2b}.
\begin{corollary}\label{cor:smoothEst}  
 Let $g\in C^\infty_{\c}(\T\setminus \vT_\p)$ and  $\kappa>0$ be given. Then,  using the notation
 $\psi(t)=U(t)g(U(1))\psi$ for  any $\psi\in \vH$,  the following bounds hold.
 \begin{subequations}
 \begin{align}\label{eq:Smoothbnd1}\begin{split}
 \forall \,\psi\in \vH&:\quad \int_{-\infty}^\infty
\norm[\big]{Q_{1,\kappa} \psi(t)}^2\,\d t \leq C\|\psi\|^2 ;\\
Q_{1,\kappa}=&Q_{1,\kappa}(t)=2\sqrt{M\chi^2_+(M/\kappa)}\,\bar f\big(\wcH\big)
  \,m^{-1/2}.     
                        \end{split}
\end{align} 
\begin{align}\label{eq:Smoothbnd2}
  \begin{split}
  \forall \,\psi\in \vH&:\quad \int_{-\infty}^\infty
\norm[\big]{Q_{2,\kappa}\psi(t)}^2\,\d t \leq C\|\psi\|^2 ;\\
Q_{2,\kappa}=Q_{2,\kappa}(t)&=2\sqrt{p\cdot {\nabla^2m(x)} p}\,\tilde{f}(\wcH)\sqrt{(\chi^2_+)^{(1)}(M/\kappa)}f\parb{\wcH}.  
  \end{split}
\end{align}  
 \end{subequations}
\end{corollary}
\begin{remarks}\label{remarks:intepretation}
  While it is easy to see that $ Q_{1,\kappa}$ is bounded, an
  interpretation is needed for $Q_{2,\kappa}$. The proofs of  Lemma \ref{lemmaNegForbidden} and Corollary \ref{cor:smoothEst9} 
   already involve such interpretation.  The proof of Corollary
   \ref{cor:smoothEst} given below   involves the same interpretation in that we 
  consider $p\cdot \nabla^2m(x) p$ as the  limiting
  form of the increasing regularized forms 
  $p\cdot\parb{\chi^2_-(r/n)\nabla^2m}p $, $n\in \N$.
   Now 
  $4\inp[\big]{p\cdot \parb{\chi^2_-(r/n)\nabla^2m}
      p}_{\wt\psi}$ with  $\wt \psi:=\tilde{f}(\wcH)\sqrt{(\chi^2_+)^{(1)}(M/\kappa)}f(\wcH)
  \psi(t)$
  is a well-defined expression. The proof  shows that the bound \eqref{eq:Smoothbnd2} is
  correct 
  with this replacement uniformly in $n$, so by   letting  $n\to \infty$ using the   Lebesgue monotone convergence
  theorem  we 
 conclude
  finiteness of the integrand  $\norm[\big]{Q_{2,\kappa}\psi(t)}^2$
  for almost all $t\in \R$. We will later find it convenient  to regularize  $Q_{2,\kappa}$
  by the weight 
  $W=\inp{p}^{1/3}$, thus  using the fact that $Q_{2,\kappa}W^{-1}$ is bounded,
  see Corollary \ref{cor:smoothEst2b}.
\end{remarks} 
\begin{proof}[\bf Proof of Corollary \ref{cor:smoothEst}]
  We  first compute the
  time-derivative of the bounded function 
  \begin{equation}\label{eq:phiSigma}
    \phi(t)=\phi_\kappa(t)=\inp[\big]{f\big(\wcH (t)\big)\chi^2_+(M/\kappa)f\big(\wcH(t)\big)}_{\psi(t)},\q t\in\R,
  \end{equation} by mimicking   the proof of Corollary \ref{cor:smoothEst9},    then  expand   and
  check  integrability as before. More precisely, we proceed this way for the regularization
\begin{equation*}
    \phi_n(t)=\phi_{\kappa,n}(t)=\inp[\big]{f\big(\wcH
      \big)\chi_-(r/n)\chi^2_+(M/\kappa)\chi_-(r/n)f\big(\wcH\big)}_{\psi(t)},\q
    n\in \N,
  \end{equation*}  claiming that as a result of integrating the derivative on any
  interval $[-T, T]$: 
  \begin{align*}
    \lim_{n\to \infty}\int_{-T}^T
\parbb{\delta\norm[\big]{\chi_-(r/n)Q_{1,\kappa} \psi(t)}^2+\kappa^{-1}\norm[\big]{\chi_-(r/n)Q_{2,\kappa}(t)\psi(t)}^2}\,\d t \leq C\|\psi\|^2 
  \end{align*} with the constant $C$ being independent of $T>0$.
  
In fact 
most  terms  either vanish in the limit  $n\to \infty$  or they are
handled by 
  \eqref{eq:LocalDecay}. There are only two terms of the expanded
  derivative not treatably this
  way, but they add exactly up to the above  non-negative
  integrand. This proves the above bound, from which two applications
  of the   Lebesgue monotone convergence
  theorem complete the proof of the lemma. 

Note that for the limiting integrand 
$\delta \norm[\big]{Q_{1,\kappa}(t)\psi(t)}^2+
  \kappa^{-1}\norm[\big]{Q_{2,\kappa}(t)\psi(t)}^2$  the first
  term is well-defined for all $t$, while the second one is
  well-defined only for almost all $t$, cf. Remarks
  \ref{remarks:intepretation}. Note also that the  limiting procedure  shows  that $\phi$ is absolutely continuous.
\end{proof}

For states in $\vH^+_{\ener}$ (to be considered in Section \ref{sec:Asymptotic clustering}) we can supplement the localization result \eqref{eq:out2b} by   cutoffs in energy of the form
  $\chi_-(M/K)$ with  large  $K>0$. This motivates the following bounds similar  to (and generalizing) (\ref{eq:Smoothbnd2}).

\begin{corollary}\label{cor:smoothEst2}  Let $g\in
  C^\infty_{\c}(\T\setminus \vT_\p)$,  $\chi\in \vF_+$ and $K>2\kappa>0$  be given. Then, using the notation
 $\psi(t)=U(t)g(U(1))\psi$ for  any $\psi\in \vH$,  the following bounds hold.
 \begin{subequations}
\begin{align}\label{eq:Smoothbnd30}
  \begin{split}
  \forall \,\psi\in \vH&:\quad \int_{-\infty}^\infty
\norm[\big]{Q_{2,\chi}\psi(t)}^2\,\d t \leq C\|\psi\|^2 ;\\
Q_{2,\chi}=&Q_{2,\chi}(t)=2\sqrt{p\cdot \nabla^2m(x) p}\,\tilde{f}\big(\wcH\big)\sq{\chi'}(M)f\parb{\wcH},  
  \end{split}
\end{align}  and hence in particular,
\begin{align}\label{eq:Smoothbnd3}
  \begin{split}
  \forall \,\psi&\in \vH:\quad \int_{-\infty}^\infty
\norm[\big]{Q_{2,K}\psi(t)}^2\,\d t \leq C\|\psi\|^2 ;\\
Q_{2,K}&=Q_{2,K}(t)=2\sqrt{p\cdot \nabla^2m(x)
  p}\,\tilde{f}\big(\wcH\big)\sqrt{-(\chi^2_-)^{(1)}(M/K)}f\parb{\wcH}. 
  \end{split}
\end{align}  
\begin{align}\label{eq:Smoothbnd4}
  \begin{split}
  \forall \,\psi\in \vH&:\quad \int_{-\infty}^\infty
\norm[\big]{Q_{2,\kappa,K}\psi(t)}^2\,\d t \leq C\|\psi\|^2 ;\\
Q_{2,\kappa,K}&=Q_{2,\kappa,K}(t)=2\sqrt{p\cdot \nabla^2m(x) p}\,\tilde{f}\big(\wcH\big)\chi_{\kappa,K}(M)f\parb{\wcH}.  
  \end{split}
\end{align}  
 \end{subequations}
\end{corollary}

\begin{proof}
Consider   
the  bounded function
\begin{subequations}
  \begin{align}\label{eq:BaspsI}
    \phi(t)=\inp[\big]{\Psi}_{\psi(t)};\q \Psi= f\big(\wcH (t)\big)
    \chi(M)f\big(\wcH(t)\big),\q t\in\R.
  \end{align} The form is similar to the function $\phi$ of
  \eqref{eq:phiSigma}, and we can mimic the proof of Corollary
  \ref{cor:smoothEst},  regularizing and concluding  that  $\phi$ is absolutely continuous with  integrable derivative $\phi'$. We compute
   $\phi'=\inp[\big]{\bD\Psi}_{\psi(t)}$ with 
    \begin{align}\label{eq:BasicComput}
      \begin{split}
      \bD \Psi&=   \delta Q_{1}^*Q_{1}+ Q_{2}^*Q_{2}+\vO(r^{-1_+});\\
      Q_{1}&=Q_{1,\chi}=2\sqrt{M\chi(M)}\,\bar f\big(\wcH\big) \,m^{-1/2},\q  Q_{2}=Q_{2,\chi}, 
      \end{split}
\end{align}  
\end{subequations} 
In particular  the derivative  is to leading order  a sum  af two
  non-negative terms, and one of them is $
\norm[\big]{Q_{2,\chi}\psi(t)}^2$. This  shows \eqref{eq:Smoothbnd30}.

Since $ K\chi^2_+(\cdot/K)\in \vF_+$, clearly \eqref{eq:Smoothbnd3} is a
special case of \eqref{eq:Smoothbnd30}. (Alternatively, note that  \eqref{eq:Smoothbnd3} is the same as \eqref{eq:Smoothbnd2}
with $\kappa=K$.)

 For \eqref{eq:Smoothbnd4} we
apply \eqref{eq:Smoothbnd30} to $\chi\in \vF_+$ defined
by $ \chi 
(t)=\int_{-\infty}^t\,\chi^2_{\kappa,K}(s)\, \d s$.
\end{proof}

  In the proof of Corollary \ref{cor:smoothEst2}  the  factor $g\in C^\infty_{\c}(\T\setminus \vT_\p)$ is needed  to treat the third  term  of \eqref{eq:BasicComput}  by 
\eqref{eq:LocalDecay}.  Without this localization there is still some non-trivial application of (\ref{eq:BasicComput}). This is the  following  smoothing type result, completely similar to Corollary \ref{cor:local-strictly-genG} (with similar proof).
\begin{corollary}
  \label{cor:local-strictly-genG5} Let  $\chi\in \vF_+$   be given.  There exists $C>0$ such that for any real numbers  $a<b$   the following bounds hold for  $\varphi(t)=U(t)\varphi$, $\varphi\in \vH$
  and with $Q_{j}$, $j=1,2$,  given as in \eqref{eq:BasicComput}.  
  \begin{subequations}
 \begin{align}\label{eq:Smoothbnd19G25}
   \forall \,\varphi\in \vH&:\quad \delta \int_{a}^b
                           \norm[\big]{Q_{1} \varphi(t)}^2\,\d t \leq \parb{C(b-a)+\sup \chi}\|\varphi\|^2 ,\\
   \forall \,\varphi\in \vH&:\quad \int_{a}^b
\norm[\big]{Q_{2}\varphi(t)}^2\,\d t \leq \parb{C(b-a)+\sup \chi}\|\varphi\|^2 .\label{eq:Smoothbnd29G25}
        \end{align} 
 \end{subequations}
\end{corollary}
\begin{remark}
  \label{remark:integral-estimates6} Since the above $Q_1$  is bounded, the most important assertion is the second bound (\ref{eq:Smoothbnd29G25}). In fact we will only need    the following  slightly distorted version of (\ref{eq:Smoothbnd29G25}): Suppose that $\chi(M)$ is replaced by  $\chi(M_1)$, where $M_1$ is constructed as  $M$ but in terms of a possibly different $m$, say denoted $m_1$. Define (with $ \wcH =m^{-\delta}H(t)m^{-\delta}$ as before)
  \begin{equation*}
    Q_2=Q_{2,\chi,m,m_1}=2\sqrt{p\cdot \nabla^2m_1(x) p}\,\tilde{f}\big(\wcH\big)\sq{\chi'}(M_1)f\parb{\wcH}.
    \end{equation*} 
    Then for some $C>0$ it holds that
    \begin{equation}
      \label{eq:GenQ}
      \forall \,\varphi\in \vH:\quad \int_{a}^b
\norm[\big]{Q_{2,\chi,m,m_1}\varphi(t)}^2\,\d t \leq \parb{C(b-a)+\sup \chi}\|\varphi\|^2 .
\end{equation}   
This follows from using $\Psi= f\big(\wcH (t)\big)
\chi(M_1)f\big(\wcH(t)\big)$, leading to the  simplified analogue of (\ref{eq:bndQ12}) stated as $\bD \Psi\geq  Q_{2}^*Q_{2}-C$, and we can proceed as in the proof of Corollary \ref{cor:local-strictly-genG}.
\end{remark}

\subsection{A square root problem, an alternative proof of
  \eqref{eq:Smoothbnd4}}\label{subsec:A complicated proof}

The above proof of \eqref{eq:Smoothbnd4}  does not generalize to the
extent that we will need later. In this subsection  we give  a
different proof. Although it is considerably more complicated, its
scheme 
 will be useful for generalizations. In general it is difficult to
 commute operators through the square root of a positive 
 operator, this may be called the 'square root singularity
problem'. In  our case it may appear useful  to commute a concrete 
operator (presently of the form $h(M)$ for some $h\in
C_\c^\infty (\R)$) through 
$\sqrt{4p\cdot \nabla^2m(x) p}$ (or possibly alternatively, through $2\sqrt{\nabla^2
      m(x) }p$)  and estimate
the commutator, however this
does not seem doable. Rather we shall regularize this square root
before computing and estimating the commutator. Our regularization
takes  advantage of the full strength
of the integral local decay estimate (\ref{eq:LocalDecay}), viz
(\ref{eq:LocalDecay}) with the explicitly included smoothing factor. The same method
will be used in Subsection \ref{subsec:integral
  estimates} for other examples.

We introduce the auxiliary operator
\begin{align}\label{eq:aux}
  \begin{split}
  P=P_{\rho,W}(t) =\tilde{f}\big(\wcH\big) W^{-1}&\parb{p\cdot \nabla^2m(x)p}W^{-1}\tilde{f}\big(\wcH\big)+r^{-2\rho};\\ &W=\inp{p}^{1/3},\q\rho\in [\delta, (6-3\delta)/7).  
  \end{split}
\end{align} By interpolating the analogue of (\ref{eq:bndM}) for
$\wcH$  and using that $\delta<  3/5$ it follows  that $P$ is a positive  bounded 
operator. See Lemma \ref{lemma:Pweights} for more refined properties.  Like 
 for $\delta$, also the
parameter $\rho$ can  henceforth be  considered  as fixed. For later use,
note that
\begin{equation}\label{eq:constntLimit}
  7\rho+3\delta<6.
\end{equation} The reader may   consider  the parameters
$\delta$ and $\rho$ as constants chosen `sufficiently close' to $1/2$,
but larger. In fact it suffices throughout the paper to take  $\rho=\delta=4/7$.

 We need the following version of \eqref{eq:Smoothbnd30}.

\begin{corollary}\label{cor:smoothEst2b}
  Let $g\in C^\infty_{\c}(\T\setminus \vT_\p)$ and  
  $\chi\in \vF_+$ be given. Then, using the notation
 $\psi(t)=U(t)g(U(1))\psi$ for any  $\psi\in \vH$,  the following bound holds.
\begin{align}\label{eq:Smoothbnd30b}
  \begin{split}
  \forall \,\psi\in \vH&:\quad \int_{-\infty}^\infty
\norm[\big]{Q_{2,\chi, \rho, W}\,\psi(t)}^2\,\d t \leq C\|\psi\|^2 ;\\
Q_{2,\chi, \rho,W}&=Q_{2,\chi, \rho, W}(t)=2\sqrt{P}\,\sq{\chi'}(M)f\parb{\wcH}W.  
  \end{split}
\end{align}
\end{corollary}
\begin{proof} By \eqref{eq:Smoothbnd30}
  \begin{align*}
  \forall \,\psi\in \vH&:\quad \int_{-\infty}^\infty
\norm[\big]{Q_{2,\chi}W^{-1} W\psi(t)}^2\,\d t \leq C\|\psi\|^2 ;\\
Q_{2,\chi}=&Q_{2,\chi}(t)=2\sqrt{p\cdot \nabla^2m(x) p}\,\tilde{f}\big(\wcH\big)\sq{\chi'}(M)f\parb{\wcH}.  
  \end{align*} We use  (\ref{eq:LocalDecay}) to treat the term $r^{-2\rho}$
    in $P$, and it remains  to show that
  \begin{align*}
    \int_{-\infty}^\infty
\norm[\big]{\sqrt{p\cdot \nabla^2m(x) p}\,W^{-1} \tilde{f}\big(\wcH\big)\sq{\chi'}(M)f\parb{\wcH}W\psi(t)}^2\,\d t \leq C\|\psi\|^2. 
  \end{align*} Hence by commutation  it suffices to show that
  \begin{align*}
    \int_{-\infty}^\infty
&\norm[\big]{S_\chi W\psi(t)}^2\,\d t \leq C\|\psi\|^2;\\& S_\chi=S_\chi(t)=\sqrt{p\cdot \nabla^2m(x) p}\,\big [W^{-1} ,\tilde{f}\big(\wcH\big)\sq{\chi'}(M)f\parb{\wcH}\big ]. 
    \end{align*}  By  \eqref{82a0} and the calculus of PsDOs we deduce  that  $[W^{-1},\sq{\chi'}(M)]=\vO(r^{-1})$. Since also $[W^{-1},\tilde{f}\big(\wcH\big)]=\big(\wcH-\i \big)^{-1}\vO(r^{-1})$,  it follows that  
    \begin{equation}\label{eq:1TstarT}
       S_\chi^*S_\chi=  \vO_\unif(r^{2\delta-1-2})=\vO_\unif(r^{-1_+}).
    \end{equation} We  conclude by (\ref{eq:LocalDecay}).
 \end{proof}

 \begin{remark}\label{remark:square-root-problem}
   We showed  above that   \eqref{eq:Smoothbnd30} implies
   \eqref{eq:Smoothbnd30b} by using  (\ref{eq:LocalDecay})  with the
   factor $\inp{p}^r:=
   W=\inp{p}^{1/3}$. Thanks to \eqref{eq:1TstarT} in fact  \eqref{eq:Smoothbnd30} and 
   \eqref{eq:Smoothbnd30b} may be considered as  equivalent bounds. 
   Since the commutators $[V,p_j]$  do not possess   a global weight,  proving   \eqref{eq:Smoothbnd30b} 
   more directly  does not seem
   possible.
 \end{remark}

 Using    the square root formula for $P$,
\begin{align}\label{eq:sqrt}
 \sqrt P=\pi^{-1}\int_0^\infty\, \tau^{-1/2} P(P+\tau)^{-1}\, \d \tau, 
\end{align} we now embark on studying certain  commutators with
$\sqrt{P}$.

\begin{lemma}\label{lemma:Pweights} The operator $P$ fulfils the following properties for all $s,\sigma\in \R$,  uniformly in $t\in \R$ and $\tau>0$. 
  \begin{subequations}
  \begin{align}\label{eq:Pwei0} P&=\vO_\unif(r^{4\delta/3-1}),\\\label{eq:Pwei1}
    \quad r^{s-\sigma}Pr^{-s}&=r^{-\sigma}P +\vO_\unif(r^{\delta/3-2-\sigma}),\\
(P+\tau)^{-1}&=\vO_\unif(r^{2\rho} ).
               \label{eq:Pwei4}\\
\quad r^{s-\sigma}(P+\tau)^{-1}r^{-s}&=r^{-\sigma}(P+\tau)^{-1}+\vO_\unif(r^{4\rho+\delta/3-2-\sigma} ),\label{eq:Pwei44}\\
     \sqrt P &=\vO_\unif(r^{2\delta/3-1/2}),\label{eq:Pwei33}\\
    \quad r^{s-\sigma}\sqrt Pr^{-s}&=r^{-\sigma}\sqrt P
+\vO_\unif(r^{\delta/3+2\rho-2 -\sigma}). 
\label{eq:Pwei2}
    \end{align}
 \end{subequations}
  \end{lemma}
  \begin{proof} \Step{I} We reduce the proof to  showing
    \eqref{eq:Pwei44} and \eqref{eq:Pwei2}. First note  that $\big [r^s, p\cdot \parb{\chi^2_+(\abs{x})\nabla^2m(x)}p\big ]=0$ and that uniformly in $t\in \R$
    \begin{equation*}
      \big [r^s, \tilde{f}(\wcH) \big ] =\vO_\unif(r^{s-1-\delta}),
    \end{equation*} cf. \eqref{eq:ExpsBreveH}. Combined with the
    calculus of PsDOs this yields   \eqref{eq:Pwei0} and  
    \eqref{eq:Pwei1}.

    \begin{subequations} 
 By  \eqref{eq:Pwei0}
   and  the familiar 
    $T^*T$-argument 
    \begin{align} \label{eq:B18} 
      r^{1/2-2\delta/3}\sqrt P  \text{  is bounded uniformly in }t,
    \end{align} and similarly we can record that 
    \begin{align} 
      \label{eq:B1}
      &\sup_{\tau>0}\,\tau\norm {(P+\tau)^{-1} }\leq 1,\\
     \label{eq:B20}
      &\sup_{\tau>0}\,\norm
        {\sqrt P(P+\tau)^{-1/2}}\leq 1,\\\label{eq:B2}
      &\sup_{\tau>0}\,\norm
        {\parb{r^{-2\rho}+\tau}^{1/2}(P+\tau)^{-1/2}}\leq 1,\\
&\sup_{\tau>0}\,\norm {r^{-\rho}(P+\tau)^{-1}r^{-\rho}}\leq 1.\label{eq:B3}
    \end{align} 
\end{subequations}

The assertion \eqref{eq:Pwei33} then follows from  \eqref{eq:B18}  and
\eqref{eq:Pwei2}.
Similarly the assertion \eqref{eq:Pwei4} follows  from \eqref{eq:B3} and \eqref{eq:Pwei44}.
    Thus it remains to show   \eqref{eq:Pwei44} and \eqref{eq:Pwei2}.

    \Step{II} We here show  \eqref{eq:Pwei44}. Let
    $\eta=2-2\rho-\delta/3$, and introduce for   all
    $s\in \R$   the operators  $B_s=r^{s-\rho}(P+\tau)^{-1}
    r^{-s-\rho}$. 
It suffices to show that
    \begin{equation}
      T_s:=B_s-B_0=\vO_\unif(r^{-\eta} ),
    \end{equation}
    In turn it suffices to show   that for all  $s\in \R$ the
    operators $r^{\eta}T_s$ are 
   bounded (uniformly in $t$ and $\tau$).
     Obviously we can assume that $s\neq 0$.

 For the case $s>0$ we  compute
\begin{align}\label{eq:resolvenComm}
  \begin{split}
   B_s&=r^{-\rho} (P+\tau)^{-1}
    r^{-\rho}\parb{ I +r^{\rho}\big
                          [P,r^{s}\big ](P+\tau)^{-1} r^{-s-\rho}}
                                              \\
&=r^{-\rho} (P+\tau)^{-1}
    r^{-\rho} \parb{
  I+\vO_{\unif}(r^{\rho +\delta/3-2 +s}) (P+\tau)^{-1} r^{-s-\rho}}. 
  \end{split}
\end{align}
\begin{subequations}
  Since $\rho +\delta/3-2 +s=s-\rho-\eta$ we can iterate this,
gaining at each step a factor $r^{-\eta}$, and hence
deduce (using \eqref{eq:B3}) that $B_s$ is (uniformly) bounded and therefore in turn  that also 
\begin{equation}\label{eq:unifB1}
  r^{\eta}T_s=r^{\eta}r^{-\rho} (P+\tau)^{-1}
  r^{-\rho} \vO_{\unif}(r^{-\eta})B_s=B_{\eta}\vO_{\unif}(r^{0})B_s
\end{equation}
is   (uniformly) bounded for $s>0$, as wanted. (Here we used 
that also 
  $B_{\eta}$ is   bounded.)

The case  $s<0$ is similar. In \eqref{eq:resolvenComm} we  pull the factor $r^{-s}$ to the left  (rather than pulling $r^{s}$ to the right), yielding again that  $B_s$ is (uniformly)  bounded. In particular we know at this point \eqref{eq:Pwei4}. 
The analogue of \eqref{eq:unifB1} then reads
\begin{equation}\label{eq:unifB2}
 r^{\eta}T_s=\parb{r^{\eta}B_s \vO_{\unif}(r^{-\eta})}\parb{r^{-\rho} (P+\tau)^{-1}
  r^{-\rho}}=\vO_{\unif}(r^{0}),
\end{equation}
  and therefore it follows that  $r^{\eta}T_s$ is   (uniformly) bounded for $s<0$ also. We have  shown  \eqref{eq:Pwei44}.
\end{subequations}

\Step{III}
We  finally show  \eqref{eq:Pwei2}.   We can assume that
$\sigma=0$,  and letting
\begin{equation*}
  T_s=r^s\sqrt Pr^{-s}-\sqrt P,
\end{equation*}
 it suffices to show that for all $s\in \R$ 
\begin{equation}\label{eq:Tbnddd}
  T_sr^{\eta} \text{ is (uniformly) bounded}; \quad \eta= 2-2\rho-\delta/3.
\end{equation}
 Using \eqref{eq:sqrt}  we  compute    
\begin{align}\label{eq:exSqrt}
  \begin{split}
  T_sr^{\eta}&= \int_0^\infty\, \tfrac{\tau ^{1/2}}\pi (P+\tau)^{-1}\big [r^s,P\big ](P+\tau)^{-1} r^{\eta-s}\, \d \tau\\
 &=\int_0^\infty\, \tfrac{\tau ^{1/2}}\pi
   (P+\tau)^{-1}\vO_{\unif}(r^{\delta/3-2 +s}) (P+\tau)^{-1}
   r^{\eta-s}\, \d \tau;\\&\quad\quad\quad\quad 
  \end{split}
\end{align} 

Let us split the integral in \eqref{eq:exSqrt} as $\int_0^1\cdots  \d \tau+\int_1^\infty\cdots  \d \tau$. 
 We first estimate 
\begin{equation}\label{eq:elem}
  \norm {\tau ^{1/2}(P+\tau)^{-1}}\leq 
  \tau^{-1/2}, 
\end{equation} and then thanks to  \eqref{eq:Pwei4}
\begin{align*}
  \norm {
     r^{\delta/3-2 +s}(P+\tau)^{-1} r^{\eta-s}} = \norm {
     r^{-2\rho-\eta+s}(P+\tau)^{-1} r^{\eta-s}}\leq C_1 ,
\end{align*}   yielding  the finiteness of the first integral. By \eqref{eq:Pwei4} and at most  one  
 commutation like in \eqref{eq:resolvenComm}
\begin{align*}
  \norm {
     r^{\delta/3-2 +s}(P+\tau)^{-1} r^{\eta-s}} = \norm {
     r^{\delta/3-2 +s}(P+\tau)^{-1} r^{2-\delta/3 -s-2\rho}} \leq C_2
   \tau ^{-1},
\end{align*}  yielding with \eqref{eq:elem} the finiteness of the
second  integral.

We have  shown \eqref{eq:Tbnddd} and therefore  also  
  \eqref{eq:Pwei2}.
\end{proof} 

 Below we introduce   a more `robust' procedure  for proving   
   \eqref{eq:Smoothbnd4} than the one
   used  before.
\begin{proof}[\textbf {Alternative proof of \eqref{eq:Smoothbnd4}}]
  \Step {I} We set up a new scheme of proof. 
   Recalling the notation 
$\psi(t)=U(t)g(U(1))\psi$ for $\psi\in \vS(\bX)$ (or for $\psi\in \vH$), we first remark that  
the
  time-derivative of the  function 
  \begin{align*}
    \phi(t)=\inp[\big]{f\big(\wcH (t)\big)\chi^2_{\kappa,K}(M)f\big(\wcH(t)\big)}_{\psi(t)}
  \end{align*} is integrable. In fact using the notation
  \eqref{eq:phiSigma}, it follows that $\phi(t)=\phi_\kappa(t)-\phi_K(t),$
   and we already showed that $\phi'_\kappa,\phi'_K\in
  L^1(\R,\d t)$ in agreement with the bounds
  \eqref{eq:Smoothbnd1}
  and  \eqref{eq:Smoothbnd2}. 

Now      we consider  the
  time-derivative of the bounded function 
\begin{align}\label{eq:pro7}
  \begin{split}
   &\br{ \phi}(t)=\inp{\Psi}_{\psi(t)};\\
  \Psi&=f\parb{\wcH}M\chi^2_{\kappa,K}(M)f\parb{\wcH}\\&=f\parb{\wcH}\chi_+(M/\kappa)\chi_-(M/K)M\chi_-(M/K)\chi_+(M/\kappa)f\parb{\wcH},  
  \end{split}
\end{align} compute, expand   and check the  integrability of   terms. The extension from $\psi\in \vS(\bX)$ to an arbitrary $\psi\in \vH$ is explained in the proof of Corollary \ref{cor:smoothEst9} (alternatively the more direct scheme of the proof of Corollary \ref{cor:smoothEst} applies).  We claim that after  expansion and
  symmetrization (to be elaborated on in Steps II-VI) all terms except a single one are treatably   by
  \eqref{eq:Smoothbnd1} and \eqref{eq:Smoothbnd2}
   as well as by   \eqref{eq:Smoothbnd3}, cf. the above
  remark. Essentially  only one new term arises, and this amounts to   the
  contribution from the commutator with the appearing  middle factor of
  $M$ in ${\Psi}$. However $\i [p^2,M]$ has a  `content of positivity'  given by $4{
   p\cdot \nabla^2m(x) p}$  and
 leads to  the `new term' $  \norm[\big]{Q_{2, \kappa,K}(t)\psi(t)}^2$  in the
 derivative of $\br{\phi}$. Consequently \eqref{eq:Smoothbnd4} would follow.

 \Step {II}
We  check that indeed the computation, expansion and
symmetrization procedure outlined in Step I is doable. Using the scheme of proof of Lemma \ref{lemmaNegForbidden},  the contribution  from $ \tfrac {\d}{\d
  t}f\parb{\wcH}$ is estimated by \eqref{eq:LocalDecay}  and the
commutator  $\i [H,\Psi]$ can  be replaced by $\i [m^\delta\wc
f\parb{\wcH} m^\delta,\Psi]=\i [m^\delta\tilde{ f}(\wcH)\wcH \tilde{
  f}(\wcH)m^\delta,\Psi]$.  By the product rule there are seven terms
in the expansion of the commutator, commuting separately through the seven factors  of   $\Psi$. 
The contribution from the middle factor 
\begin{align*}
  S_4=f\parb{\wcH}\chi_{\kappa,K}(M)\i \big[m^\delta\wc f\parb{\wcH} m^\delta,M\big]\chi_{\kappa,K}(M)f\parb{\wcH}
\end{align*} is the wanted 'new term' mentioned in Step I. This does not need
symmetrization.

The sum of the other terms takes the form
\begin{equation*}
  S_1+S_2+S_3+ S_5+S_6+S_7= (S_1+\text{hc})+(S_2+\text{hc}) +(S_3+\text{hc}),
\end{equation*} where $S_1+\text{hc}=S_1+S_1^*=S_1+S_7$ and similarly
for $S_2$ and $S_3$.

In Steps III-V below we treat the contribution from $S_2+\text{hc}$ to
the time-integral,
and in Step VI  we treat the remaining cases  defined for
$S_3+\text{hc}$ (doable in a simlarly way) and  $S_1+\text{hc}$ (doable in a much simpler  way).

\Step {III} 
\noindent\underline{$S_2+\text{hc}$:}
 By definition 
\begin{align*}
  S_2+\text{hc}=f\parb{\wcH}\i \big[m^\delta\wc f\parb{\wcH}
  m^\delta,\chi_+(M/\kappa)\big]M\chi^2_-(M/K)\chi_+(M/\kappa)f\parb{\wcH} +\text{hc}.
\end{align*} After  computations  similar to some appearing in 
the proofs of Lemma \ref{lemmaNegForbidden} and Corollary
\ref{cor:smoothEst2b}   the main part is identified, to be proven
below, as
\begin{align}\label{eq:Qintro1}
  \begin{split}
 \tfrac 2 \kappa Wf\parb{\wcH}&B_\kappa\tilde{f}\big(\wcH\big)\,W^{-1} \parb{p\cdot \nabla^2m(x) p}W^{-1}\tilde{f}\big(\wcH\big)B_\kappa h(M)
                                                                                                                                                             f\parb{\wcH}W+\text{hc};\\
    & B_\kappa=\sqrt{(\chi^2_+)^{(1)}(M/\kappa)},\q\q h(s)=\chi_+(2s/\kappa)s\chi^2_-(s/K).    
  \end{split} 
\end{align} In terms of the positive operator $P=P_{\rho,W}$ from
\eqref{eq:aux} this expression is on the form
\begin{align*}
 \tfrac 2 \kappa
 W\parb{f\parb{\wcH}B_\kappa \,P_{\rho,W}\,B_\kappa h(M)f\parb{\wcH}+\text{hc} +\vO_\unif(r^{-1_+})}W.
\end{align*}

First, it follows by commutation (skipping here elements from Step VI in the
proof of  Lemma \ref{lemmaNegForbidden}) 
\begin{align*}
  \begin{split}
 S_2+\text{hc}=\tfrac 2 \kappa f\parb{\wcH}B_\kappa\tilde{f}\big(\wcH\big)\,\parb{p\cdot \nabla^2m(x) p}\tilde{f}\big(\wcH\big)B_\kappa h(M)
f\parb{\wcH}+\text{hc}+\vO_\unif(r^{-1_+}).
  \end{split} 
\end{align*}
Second,  we argue
that  the essential  part of this expression in turn is identified as 
\begin{align}\label{eq:Qintro2}
  \begin{split}
 \tfrac 2 \kappa Wf\parb{\wcH}B_\kappa\tilde{f}\big(\wcH\big)\,W^{-1} \parb{p\cdot \nabla^2m(x) p}\tilde{f}\big(\wcH\big)B_\kappa h(M)
f\parb{\wcH}+\text{hc}.
  \end{split} 
\end{align} This amounts to estimating 
\begin{align*}
 \tfrac 2 \kappa W\big
 [f\parb{\wcH}B_\kappa \tilde{f}\big(\wcH\big),\,W^{-1}\big
 ] \parb{p\cdot \nabla^2m(x) p}\tilde{f}\big(\wcH\big)B_\kappa h(M)
f\parb{\wcH}+\text{hc},
\end{align*} and therefore by commutation (cf. Lemma
\ref{lemma:expansion})   and   using  notation from
\eqref{eq:Smoothbnd2},
\begin{align*}
  \begin{split}
 &\tfrac 2 \kappa W\big
 [f\parb{\wcH}B_\kappa \tilde{f}\big(\wcH\big),\,W^{-1}\big
 ] h(M)\tilde{f}\big(\wcH\big) \sq {p\cdot \nabla^2m(x) p}\\&\q\q\q\q \q\q\q\q\q\q\q\q \sq {p\cdot \nabla^2m(x) p}\tilde{f}\big(\wcH\big)B_\kappa 
f\parb{\wcH}+\text{hc}
\\&=  \kappa^{1} W S^*_{\chi_+^2,h} Q_{2, \kappa}
+\text{hc};\\&\q\q\q
 S_{\chi_+^2,h}=\sq {p\cdot \nabla^2m(x)
  p}\,\tilde{f}\big(\wcH\big)h(M)  \big [W^{-1} ,\tilde{f}\big(\wcH\big)B_\kappa f\parb{\wcH}\big ].
  \end{split} 
\end{align*} The last expression  is treatable by \eqref{eq:Smoothbnd2},
a slight modification  of \eqref{eq:1TstarT} and by (\ref{eq:LocalDecay}),  yielding the bound
$C\norm {\psi}^2$ of  its  contribution  to the 
time-integral, so indeed the essential  part is   given by
\eqref{eq:Qintro2}. 

Third, to obtain \eqref{eq:Qintro1} from \eqref{eq:Qintro2}, it
suffices to estimate 
\begin{align}\label{eq:Qintro3}
  \begin{split}
 \tfrac 2 \kappa
 Wf\parb{\wcH}&B_\kappa \tilde{f}\big(\wcH\big)\,W^{-1}
 \sq
 { p\cdot \nabla^2m(x) p}\\&\q\q\q \sq
 { p\cdot \nabla^2m(x) p}\Big [ W^{-1}, \tilde{f}\big(\wcH\big)B_\kappa h(M)
f\parb{\wcH}\Big ]W +\text{hc}.
  \end{split} 
\end{align} By  repeating the above  argument,  using
\eqref{eq:Smoothbnd30b} and 
   a slight modification  of \eqref{eq:1TstarT}, indeed
 it follows that 
 the time-integral of the  contribution from \eqref{eq:Qintro3} 
 is bounded by $C\norm {\psi}
^2$. We have derived the main  term \eqref{eq:Qintro1}  as wanted.

\Step {IV}  We continue our study of $S_2+\text{hc}$. 
In terms of the positive operator $P=P_{\rho,W}$ from
\eqref{eq:aux} and its square root $Q=Q_{\rho,W}=\sqrt{ P_{\rho,W}}$
 the expression \eqref{eq:Qintro1}  can be written as $\tfrac 2 \kappa
 WTW$, where 
\begin{align}\label{eq:Tsquared}
  \begin{split}
  T=
  &f\parb{\wcH}B_\kappa \,Q^2\,B_\kappa h(M)f\parb{\wcH}+\text{hc} +\vO_\unif(r^{-1_+})\\
=&2
  f\parb{\wcH}B_\kappa \,Qh(M)Q B_\kappa f\parb{\wcH}+
  \parbb{f\parb{\wcH}B_\kappa \,QS
  +\text{hc}}
           +\vO_\unif(r^{-1_+});\\
&  \qq\qq  S=S_{\kappa, \rho,W}=\big [Q B_\kappa \,,h(M)\big
       ]f\parb{\wcH}=\big [Q \,,h(M)\big
       ]B_\kappa f\parb{\wcH}.  
  \end{split}
\end{align} 
  Obviously the first term  is treatable  by Corollary \ref{cor:smoothEst2b}, and
 we only need to control the middle  term involving the operator $S$. We decompose 
\begin{align*}
  S&=Q^{-1}R_1B_\kappa f\parb{\wcH}+Q^{-1}R_2B_\kappa f\parb{\wcH};\\
&R_1=2^{-1} \big [P,M\big ] h'(M),\qq R_2=Q\big [Q,h(M)\big
       ]-2^{-1} \big [P,M\big ] h'(M).
\end{align*} The contribution from $Q^{-1}R_1B_\kappa f\parb{\wcH}$ is
easy: We substitute it 
 in the middle  term of \eqref{eq:Tsquared} 
and use 
the cancellation $ QQ^{-1}=I$. This  leads  to the expression 
\begin{align*}
  2^{-1}
  f\parb{\wcH}&B_\kappa \, \parbb{\big
           [P,M\big ] h'(M)-h'(M)\big [P,M \big ]}
          B_\kappa f\parb{\wcH}, 
\end{align*} which is  on the
form $\vO_\unif(r^{4\delta/3 -3})=\vO_\unif(r^{-1_+})$, cf. Lemma
\ref{lemma:expansion} (and its proof).

For the  contribution from  $Q^{-1}R_2B_\kappa f\parb{\wcH}$ in \eqref{eq:Tsquared} it suffices to show
that 
\begin{align}\label{eq:Rtwo}
  \forall \,\psi\in \vH:\quad \int_{-\infty}^\infty
\norm[\big]{r^\rho R_2 B_\kappa f\parb{\wcH}W\psi(t)}^2\,\d t \leq C\|\psi\|^2. 
\end{align}
\begin{subequations}

   Now to address  \eqref{eq:Rtwo} 
   we  compute using  \eqref{eq:sqrt}
   \begin{align}\label{eq:Tcomh}
    \begin{split}
   Q\big [Q,h(M)\big
     ]&= 2^{-1}[P,h(M)] \\& \qq +\int_0^\infty\, \tfrac{\tau ^{1/2}}{\pi} \sqrt
    P (P+\tau)^{-2}\ad^2_P(h(M))(P+\tau)^{-1}\, \d \tau.   
    \end{split}
   \end{align} Note that for any $\varepsilon>0$  
   \begin{align*}
 \pi^{-1}\int_0^\infty\, \tau^{1/2} F(P>\varepsilon)\sqrt P(P+\tau)^{-2}\, \d \tau=2^{-1}F(P>\varepsilon), 
   \end{align*} which implies that (\ref{eq:Tcomh}) is correct when multiplied by the projection  $F(P>\varepsilon)$ from the left. Since $\ker (P)=\set{0}$,  we then obtain (\ref{eq:Tcomh}) by letting $\varepsilon \to 0$.

   In turn   
    using  that
  $\ad^k_M(P)=\vO_\unif(r^{4\delta/3-k-1})$, $k=0,1,2$,  it follows  
  that 
  \begin{equation}\label{eq:Pcomh}
    2^{-1}[P,h(M)]  
= 2^{-1}\big [P,M\big ] h'(M)+\vO_\unif(r^{4\delta/3-3}).
  \end{equation}  
\end{subequations}  By substituting   \eqref{eq:Pcomh} in
\eqref{eq:Tcomh} we arrive at the formula 
\begin{align*}
  R_2= \int_0^\infty\, \tfrac{\tau ^{1/2}}{\pi} \sqrt
  P (P+\tau)^{-2}\ad^2_P(h(M))(P+\tau)^{-1}\, \d \tau+\vO_\unif(r^{4\delta/3-3}).
\end{align*} The second  term is on the form $\vO_\unif(r^{-2\rho}) $. Writing similarly  the first  term   on the form
$r^{-\rho}Rr^{-\rho}$,  we compute 
 \begin{equation*}
 R=\int_0^\infty\, {\tau ^{1/2}} r^{\rho}\sqrt
  P (P+\tau)^{-2}\vO_\unif(r^{8\delta/3-4})(P+\tau)^{-1}r^{\rho}\, \d \tau,
\end{equation*} and  hence  it remains for \eqref{eq:Rtwo}   to show that 
\begin{equation}\label{eq:intsqrtB}
 \sup_{t\in \R}\,\norm[\Big]{\int_0^\infty\, {\tau ^{1/2}} r^{\rho}\sqrt
  P (P+\tau)^{-2}\vO_\unif(r^{8\delta/3-4})(P+\tau)^{-1}r^{\rho}\, \d
  \tau}_{\vL (\vH)}< \infty.
\end{equation}

\Step {V} We prove
\eqref{eq:intsqrtB}, finishing our treatment of the term $S_2+\text{hc}$. Our proof is  based on Lemma \ref{lemma:Pweights} and its proof.
If we  `ignore' the first term of $P$ (considering  $P\approx r^{-2\rho}$) and hence also
ignore commutation errors,  the bound
\eqref{eq:intsqrtB} should be a `consequence' of the simplified bound
\begin{equation*}
  \sup_{r\geq 1}\, \int_0^\infty\, {\tau ^{1/2}} r^{2\rho+8\delta/3-4}(r^{-2\rho}+\tau)^{-5/2}\, \d \tau<\infty,
\end{equation*} which   in turn  is valid since by \eqref{eq:constntLimit} the number 
$\eta:=4-4\rho-8\delta/3>0$.

To  proceed rigorously we need to worry about certain
commutations. Let us start with examining the `leading contribution'
arising by replacing   $r^{\rho}\sqrt P$ by $\sqrt P r^{\rho}$ (the
error is  controlled below by \eqref{eq:Pwei2}).   Then we split the integral
 as $\int_0^1\cdots  \d \tau+\int_1^\infty\cdots  \d \tau$ and  first check  the
 (uniform) integrability on $(0,1)$ of 
\begin{align*}
   \phi(\tau):=\tau ^{1/2}\norm { \sqrt P r^{\rho} 
(P+\tau)^{-2}\vO_\unif (r^{8\delta/3-4})(P+\tau)^{-1}r^{\rho} }.
\end{align*}  
We observe  the following versions  of \eqref{eq:resolvenComm}.
\begin{subequations}
\begin{align}\label{eq:resolvenComm2}
  \begin{split}
   &r^{\rho}(P+\tau)^{-1} r^{-\rho}\\&=(P+\tau)^{-1}
   \parb{ I +\big
                          [P,r^{\rho}\big ](P+\tau)^{-1} r^{-\rho}}
                                              \\
&=(P+\tau)^{-1}
  \parb{ I+\vO_\unif(r^{\rho +\delta/3-2 }) (P+\tau)^{-1} r^{-\rho}}\\
    &=:(P+\tau)^{-1}
    \parb{ I+A},
  \end{split}
\end{align}
\begin{align}\label{eq:resolvenComm23}
  \begin{split}
   r^{-\rho}(P+\tau)^{-1} r^{\rho}&=
                                    \parb{ I+r^{-\rho}(P+\tau)^{-1} \vO_\unif(r^{\rho +\delta/3-2 }) }(P+\tau)^{-1}\\
    &=
    \parb{ I+A^*}(P+\tau)^{-1},
  \end{split}
\end{align}
 \end{subequations}
yielding first  with  \eqref{eq:B20} and \eqref{eq:B2},  and in the last
step  with \eqref{eq:Pwei4}, 
\begin{align*}
   &\phi(\tau)\\&\leq C_1 \tau ^{1/2}\norm {\parb{r^{-2\rho}+\tau}^{-1/2} \parb{ I+A}r^{\rho}
  (P+\tau)^{-1}\vO_\unif(r^{8\delta/3-4+\rho})\parb{ I+A^*}(P+\tau)^{-1}}\\
  &\leq C_2 \norm {\parb{r^{-2\rho}+\tau}^{-1/2} \parb{ I+A}r^{\rho}
    (P+\tau)^{-1}\vO_\unif(r^{8\delta/3-4+\rho})\parb{ I+A^*}\parb{r^{-2\rho}+\tau}^{-1/2}}\\
  &\leq C_3 \norm {r^{\rho}\parb{ I+A}r^{\rho}
    (P+\tau)^{-1}\vO_\unif(r^{8\delta/3-4+\rho})\parb{ I+A^*}r^{\rho}}\\
  &\leq C_4 \norm {r^{\rho}\parb{ I+A}r^{-\rho}\parbb{r^{2\rho}
    (P+\tau)^{-1}\vO_\unif(r^{8\delta/3-4+2\rho})}r^{-\rho}\parb{ I+A^*}r^{\rho}}\\
   &\leq C_5 \norm {  r^{2\rho}
     (P+\tau)^{-1}\vO_\unif(r^{-2\rho-\eta})};\q \eta=4-4\rho-8\delta/3.
  \end{align*}   By using 
  \eqref{eq:resolvenComm2} with $\rho$ replaced by $2\rho$ and
  \eqref{eq:Pwei4} (along the same lines),  and in addition an interpolation of
  \eqref{eq:Pwei4} and \eqref{eq:B1} (using that $0<\eta< 2\rho$),  we can  in turn  bound
  \begin{equation}
    \label{eq:kappaB}
    \norm {  r^{2\rho} (P+\tau)^{-1}\vO_\unif(r^{-2\rho-\eta})}\leq C_1\norm {  P+\tau)^{-1}r^{-\eta}}\leq C_2 \tau^{-(1-\eta/2\rho)}.
  \end{equation} Clearly  \eqref{eq:kappaB}  yields  that
  $\int_0^1\,{\phi(\tau)}\,  \d \tau <\infty$  and hence a  
   finite contribution from  $\int_0^1\cdots  \d \tau$,  as wanted.  By  the second estimation of $\phi(\tau)$  given above
  (and again by using 
  \eqref{eq:Pwei4} and 
  \eqref{eq:resolvenComm2})  
  we can also estimate
  \begin{equation*}
    \int_1^\infty\,{\phi(\tau)}\,  \d \tau \leq \int_1^\infty\,C\tau^{-1/2}\tau^{-1}\tau^{-1/2}\,  \d \tau =C<\infty.
  \end{equation*}

 It remains to check the above replacement of $r^{\rho}\sqrt P$ by
 $\sqrt P r^{\rho}$. By \eqref{eq:Pwei2} (applied with  $s=\rho$ and
 $\sigma=0$) we need to show the  integrability of 
\begin{align*}
   \br \phi(\tau):=\tau ^{1/2}\norm { r^{\br\rho} 
(P+\tau)^{-2}\vO_\unif (r^{8\delta/3-4})(P+\tau)^{-1}r^{\rho} };\q \br\rho=\delta/3+3\rho-2.
\end{align*}  By using
  \eqref{eq:resolvenComm2} with $\rho$ replaced by $\br
  \rho$ we estimate as above (with a similar $A$)
\begin{align*}
   \br \phi(\tau)&\leq C_1 \norm {(P+\tau)^{-1} \parb{ I+A}r^{\br\rho}
  (P+\tau)^{-1}\vO_\unif(r^{8\delta/3-4+\rho})\parb{
                   I+A^*}\parb{r^{-2\rho}+\tau}^{-1/2}}\\
&\leq C_2 \norm {(P+\tau)^{-1} \vO_\unif(r^0) 
  (P+\tau)^{-1}r^{\br\rho+8\delta/3-4+2\rho}}\\
  &= C_2\norm {(P+\tau)^{-1} \vO_\unif(r^0) 
    (P+\tau)^{-1}r^{-2\rho}r^{7\rho+3\delta-6}}\\
  &= \norm {(P+\tau)^{-1} r^{-\br \eta}\vO_\unif(r^0)};\q \br \eta=6-3\delta-7\rho.
  \end{align*} Arguing as in \eqref{eq:kappaB}  (using  that $0<\br \eta< 2\rho$, cf. \eqref{eq:constntLimit}) yields 
  a finite contribution from  $\int_0^1\cdots  \d \tau$, and for $\tau>1$ the  estimation shows that $\br \phi (\tau)\leq C \tau^{-2}$ yielding also    a finite contribution from  $\int_1^\infty\cdots  \d \tau$.

   We have verified \eqref{eq:intsqrtB} and therefore proven \eqref{eq:Rtwo}.

\Step {VI}  We treat the contributions from 
$S_3+\text{hc}$ and  $S_1+\text{hc}$ to
the time-integral. 
\noindent\underline{$S_3+\text{hc}$:} By definition 
\begin{align*}
  S_3+\text{hc}=f\parb{\wcH}\chi_+(M/\kappa)\i \big[m^\delta\wc f\parb{\wcH}
  m^\delta,\chi_-(M/K)\big]M\chi_{\kappa,K}(M)f\parb{\wcH} +\text{hc}.
\end{align*}  This term can be treated similarly to $S_2+\text{hc}$,
now using \eqref{eq:Smoothbnd3} rather than  \eqref{eq:Smoothbnd2}
after commutation. Obviously the difficult term involving the commutator
\begin{equation*}
  \Big [Q\sqrt{-(\chi^2_-)^{(1)}(M/K)},\br h(M) \Big
       ]f\parb{\wcH},\q Q=Q_{\rho,W},\q \br h(s)=\chi_-(s/2K)s\chi^2_+(s/\kappa),
\end{equation*}
 can be  treated as before.

\

\noindent\underline{$S_1+\text{hc}$:} We compute, cf. the proof of
Corollary  \ref{cor:smoothEst}, 
\begin{align*}
   &S_1+\text{hc}=\i \big [ m^\delta\wc f\parb{\wcH}
  m^\delta,f\parb{\wcH}\big ] M\chi^2_{\kappa,K}(M)f\parb{\wcH}
                  +\text{hc}\\
&=4\delta  m^{-1/2}\bar f\big(\wcH\big) \,{M^2\chi^2_{\kappa,K}(M)}\,\bar f\big(\wcH\big) m^{-1/2}+\vO_\unif(r^{-1_+}).
\end{align*} Obviously \eqref{eq:Smoothbnd1} applies.
\end{proof}

\subsection{Generalized integral
  estimates}\label{subsec:integral
  estimates}

Recalling  that $m$ and $M$ depend on a small $\epsilon>0$, we  let 
$\epsilon_0=\epsilon$ and  $m_0=m$. This is consistent with our previous convention $M_0=M$.  Suppose now
we are additionally  given
a finite number of   small positive numbers $\epsilon_1,\dots, \epsilon_I$ and  corresponding functions   $m_1,\dots, m_I$ and 
operators  $M_1,\dots, M_I$  defined as for  $m_0$ and $M_0$ just by
replacing $\epsilon_0$ by   $\epsilon_1,\dots, \epsilon_I$ in the constructions.  In Subsection \ref{subsec:Radiation condition
  integral estimates} we introduce  a concrete example  of this setup for which  the below Corollary  \ref{cor:smoothEst3}  (in the disguised form of  Lemma  \ref{lemma:concretLem}) will be  important for us in Subsection \ref{subsec:Proof
  of Proposition}.

Introduce  then 
 the operator  
\begin{subequations}
  \begin{equation}\label{eq:minus1}   E_{-1}=f\parb{\wcH},  
  \end{equation}
  and let for $i=0, 1,\dots, I$ and given 
$K>2\kappa>0$   
\begin{align}\label{eq:tildeT}
    E_i&=\chi_{i}(\bM)f\parb{\wcH};\\
    \chi_i(\bM)&=\chi_{\kappa,K}(M_i)\chi_{\kappa,K}(M_{i-1})\cdots \cdots\chi_{\kappa,K}(M_0).\label{eq:barM2}
  \end{align}  
\end{subequations} Note that $\wcH$ and $f\parb{\wcH}$ are  defined by
\eqref{eq:widecheckH} and \eqref{eq:fesb} in terms of $m_0=m$ (as  considered in \eqref{eq:Smoothbnd4} in 
combination with the factor $\chi_0(\bM)=\chi_{\kappa,K}(M_0)$).

\begin{corollary}\label{cor:smoothEst3}
  Let $g\in C^\infty_{\c}(\T\setminus \vT_\p)$,  $\chi\in
  \vF_+$,  $K>2\kappa>0$   and  any small
 $\epsilon_0, \epsilon_1\dots ,\epsilon_I >0$  be given. Then, using the notation
 $\psi(t)=U(t)g(U(1))\psi$ for  any $\psi\in \vH$,  the following
 bounds  hold.
 \begin{subequations}
 \begin{align}\label{eq:Smoothbnd50}
  \begin{split}
  \forall \psi\in \vH\,\, \forall &i=0,\dots, I:\q\int_{-\infty}^\infty
\norm[\big]{Q_{2,i}\psi(t)}^2\,\d t \leq C\|\psi\|^2 ;\\
Q_{2,i}&=Q_{2,i}(t)=2\sqrt{p\cdot \nabla^2m_i(x)
  p} )\tilde{f}\big(\wcH\big)\sq{\chi'}(M_i)  E_{i-1}.   
  \end{split} 
\end{align}  
\begin{align}\label{eq:Smoothbnd5}
  \begin{split}
  \forall \psi\in \vH\,\, \forall &i=0,\dots, I:\q\int_{-\infty}^\infty
\norm[\big]{Q'_{2,i}\psi(t)}^2\,\d t \leq C\|\psi\|^2 ;\\
&\qq\qq Q'_{2,i}=Q'_{2,i}(t)=2\sqrt{p\cdot \nabla^2m_i(x)
  p} \tilde{f}\big(\wcH\big)  E_i.  
  \end{split}
\end{align}  
 \end{subequations}  
\end{corollary}
\begin{proof}[Outline of  proof] The bounds \eqref{eq:Smoothbnd5} is a special case of the bounds
  \eqref{eq:Smoothbnd50}, cf.  the proof of Corollary
  \ref{cor:smoothEst2}.  

The proof of \eqref{eq:Smoothbnd50} is by induction in $i$. By
  \eqref{eq:Smoothbnd30}  we know \eqref{eq:Smoothbnd50} for $i=0$.
So suppose $i\in \set{1\dots, I}$ and that  \eqref{eq:Smoothbnd50} is  known for
the indices $0,\dots, i-1$. Then we need  to verify \eqref{eq:Smoothbnd50} 
for this   $i$.

Motivated by the proof of 
\eqref{eq:Smoothbnd30}  and our `Alternative proof of
  \eqref{eq:Smoothbnd4}'  we   consider the   time-derivative of  the
bounded function
\begin{equation} \label{eq:barM_i}
   \phi_i(t)=\inp[\big]{f\big(\wcH (t)\big)\chi_{i-1}(\bM)^*\chi(M_{i})\chi_{i-1}(\bM)f\big(\wcH(t)\big)}_{\psi(t)}.
   \end{equation}
The main contribution from the middle factor $\chi(M_{i})$, 
\begin{align*}
  \inp[\big]{f\big(\wcH \big)\chi_{i-1}(\bM)^*\sq{\chi'}(M_i)\i \big[m^\delta\wc
  f\parb{\wcH}
  m^\delta,M_i\big]\sq{\chi'}(M_i)\chi_{i-1}(\bM)f\big(\wcH\big)}_{\psi(t)},
\end{align*} is the wanted term. The contributions  from commutation
with the other
factors   are treated by
induction after 
symmetrization. We circumvent   the (dangerous) square root singularity 
problem   by using versions of the operator $ P=P_{\rho,W}(t) $ used
in the  alternative proof of \eqref{eq:Smoothbnd4}. Note that the bound \eqref{eq:Smoothbnd50} with $i$ replaced
by any $j$ fulfilling $0\leq
j\leq i-1$  is  equivalent to
the following bound:
\begin{align}\label{eq:altBnd}
\begin{split}
  &\int_{-\infty}^\infty
\norm[\big]{Q_{2,j,\rho,W}\psi(t)}^2\,\d t \leq C\|\psi\|^2 ;\\
Q_{2,j,\chi, \rho,W}&=Q_{2,j,\chi, \rho,W}(t)=2
\sqrt{P_j}\sq{\chi'}(M_j)  E_{j-1}W,\\
P_j=P_{j,\rho,W}(t) &=\tilde{f}\big(\wcH\big) W^{-1}\parb{p\cdot 
\nabla^2m_j(x)p}W^{-1}\tilde{f}\big(\wcH\big)+r^{-2\rho}.   
  \end{split}
\end{align} Here 
$P_0$ and $Q_{2,0,\chi, \rho,W}$ agree  with \eqref{eq:aux} and
$Q_{2,\chi, \rho,W}$ from  in Corollary \ref{cor:smoothEst2b},
respectively. Hence the  stated equivalence for $j=0$ follows from 
Remark \ref{remark:square-root-problem}. The general case $0\leq
j\leq i-1$ follows
similarly by commutation.

The induction step mimic our     alternative proof
of \eqref{eq:Smoothbnd4} using \eqref{eq:altBnd} and commutation.  Note
that previously  the complication arose
from dealing with the commutator $[\sq P_0 ,h(M_0)]$. Now we need 
similarly to deal with the commutators 
\begin{align}\label{eq:sqrjjj}
  \big [\sq P_j , \chi_{\kappa,K}(M_{j+1}))\cdots
  \chi_{\kappa,K}(M_{i-1})\chi(M_i)\chi_{\kappa,K}(M_{i-1})\cdots
  \chi_{\kappa,K}(M_{j+1}) \big ]
\end{align} for $j\leq i-2$ and 
 the commutator $ \big [\sq P_j , \chi(M_i)\big ]$
  for $j= i-1$.  Note that this reduction comes about by  mimicking the previous proof using  the induction
  hypothesis and 
  \eqref{eq:altBnd} with  $\chi=\chi_1=\chi^2_+(\cdot/\kappa)$ and 
  with $\chi=\chi_2=\chi^2_+(\cdot/K)$ (since $(\chi^2_{\kappa,K})' =\chi_1'-\chi_2'$). These  bounds are   analogues  of \eqref{eq:Smoothbnd2} and \eqref{eq:Smoothbnd30b}, respectively. Again the needed momentum-localization is provided by the
factors of $f\big(\wcH \big)$, and in fact we can deal with (\ref{eq:sqrjjj}) as in  the alternative proof
of \eqref{eq:Smoothbnd4}. Hence the pattern of proof is the same. We leave out  the details of proof.
\end{proof}

\section{Asymptotic clustering of states in
  $\vH^+_{\ener}$}\label{sec:Asymptotic clustering}
 The main result of this section reads as follows.
\begin{proposition}\label{prop:negB} For 
  any $\psi\in \vH^+_{\ener}$  and   $\varepsilon >0$ there exists a family of states
  $\set[[big]{\varphi_{a,\varepsilon}\in \vH\mid a\in \vA_1}$ such that 
   \begin{equation}
\label{eq:scatdeepsilon}
      \limsup _{t\to \infty}\, \norm[\Big]{U(t)\psi-\sum_{a\in \vA_1}
        U_a(t)\varphi_{a,\varepsilon}}\leq \varepsilon .
   \end{equation} 
   \end{proposition}

   \begin{remark}\label{remark:asymptClust} The constructed
     vectors $\varphi_{a,\varepsilon}\in \vH^+_{\ener}(H_a)$:
   From  \eqref{eq:varphiLimi} it follows  that $\varphi_{a,\varepsilon}\in
\vH_{\ac}(H_a)$,  and by \eqref{eq:BBB}  and \eqref{eq:varphiLimi} that 
\begin{align}\label{eq:varphiLimi2}
  \varphi_{a,\varepsilon}= \lim_{t\to \infty}\,
    U_a(t)^*S_a(t)U(t)\psi.
\end{align} Here $S_a(t)$ is  defined in \eqref{eq:approxINT} for a
fixed  $\psi$ from a dense subset of $ \vH^+_{\ener}$. 
The representation
 \eqref{eq:varphiLimi2} yields that indeed $\varphi_{a,\varepsilon}\in
\vH^+_{\ener}(H_a)$, for example seen by using the commutation  property $[\chi(p^2/E), S_a(t)]=\vO_{\unif}(r^{-0_+})$. 
     \end{remark}

The property \eqref{eq:scatdeepsilon} may be referred to as
\emph{asymptotic clustering}, which is consistent with this terminology  used in time-independent $N$-body 
scattering theory, see for example \cite[Section 6.11]{DG}. The
proof is given in Subsection \ref{subsec:Proof
  of Proposition}. In Subsection \ref{subsec:Radiation condition
  integral estimates}  we derive various  integral
  estimates using Subsections \ref{subsec: Geometric considerations
    for a0=a{max}} and \ref{subsec:integral
  estimates}. In Subsection  \ref{subsec:A prior  localization}  we derive initial localization of states in
  $\vH^+_{\ener}$.
\subsection{Radiation condition  bounds}\label{subsec:Radiation condition integral estimates}
As in Subsections \ref{subsec: Geometric considerations
    for a0=a{max}} and \ref{subsec:integral
  estimates}  we first fix $\epsilon=\epsilon_0>0$. We recall  that for each $a\in\vA_1$ we picked in Subsection
\ref{subsubsec: Controlling commutators} certain small positive numbers $\epsilon_1,\dots,
\epsilon_J$ for some $J=J(a)\in \N$. Next, upon also varying  $a\in\vA_1$,  this leads overall to 
 families  of functions $m_i$'s and    operators $M_i$'s conforming  with the outset of Subsection \ref{subsec:integral
  estimates}. In particular we use the same notation $\epsilon_0, \epsilon_1,\dots, \epsilon_I$, $m_0, m_1,\dots, m_I$ and $M_0, M_1,\dots, M_I$, i.e. we use  the  index $i\in \set{0,\dots, I}$. Although first to  be used in Subsection \ref{subsec:A prior  localization}, let us note that
\begin{equation}\label{eq:DECO}
  M_0\chi_+(|x|)=\sum_{a\in \vA_1}\,M_a\chi_+(|x|),
\end{equation}
 where the operators $M_a$ are given as in Subsection \ref{subsec: Geometric considerations
   for a0=a{max}} (in terms of $\epsilon$).

We will apply Corollary
\ref{cor:smoothEst3} to  the above  setting. In particular this involves the constructions 
(\ref{eq:minus1})-\eqref{eq:barM2}  for fixed parameters
$K>2\kappa>0$.

 If $(a,j)$ with the above  convention is given the
label $i\in \set{1,\dots, I}$,  it   follows  from \eqref{eq:Hes_est}   (recalling that  
$Q(a,j)=\xi_jG_{b_j}$ is defined previously in Subsection \ref{subsubsec: Controlling commutators})  that
\begin{align*}
  E_i^*\tilde{f}\big(\wcH\big) Q(a,j)^*Q(a,j)\tilde{f}\big(\wcH\big)  E_i 
\leq 4^{-1}  Q'_{2,i}(t)^*Q'_{2,i}(t),
\end{align*}  or  stated alternatively
in terms of the abbreviation  $ Q_i:=Q(a,j)$ as 
\begin{align}\label{eq:Hes_estb} 
  E_i^*\tilde{f}\big(\wcH\big) Q_i^*Q_i\tilde{f}\big(\wcH\big)  E_i 
\leq 4^{-1}  Q'_{2,i}(t)^*Q'_{2,i}(t).
\end{align} 
\begin{subequations}
  
We recall the setting of \eqref{eq:Smoothbnd5}  in Corollary
\ref{cor:smoothEst3}:  For $\psi\in \vH$  and fixed  $g\in
C^\infty_{\c}(\T\setminus \vT_\p)$ we consider the evolution $\psi(t)=U(t)g(U(1))\psi$.

\begin{lemma}\label{lemma:concretLem} Under the conditions of
  \eqref{eq:Smoothbnd5} with  $M_0, M_1,\dots, M_I$ chosen as
  explained above (in particular fixing $ E_1, \dots, E_I$), it follows 
  with the  labelling convention of \eqref{eq:Hes_estb} 
  that for any $\psi\in \vH $ and  $ i=1,\dots, I$:
    \begin{equation}\label{eq:Smoothbnd6}
  \int_{-\infty}^\infty
\norm[\big]{Q_i\tilde{f}\big(\wcH\big) E_i \psi(t)}^2\,\d t \leq C\|\psi\|^2.  
\end{equation}  
\begin{equation}\label{eq:Smoothbnd7}
  \int_{-\infty}^\infty
\norm[\big]{Q_i\tilde{f}\big(\wcH\big) E_I\psi(t)}^2\,\d t \leq C\|\psi\|^2.  
\end{equation}  
  \end{lemma}
\end{subequations}  
  \begin{proof}
    The bounds \eqref{eq:Smoothbnd6} follow immediately from \eqref{eq:Smoothbnd5}
    and \eqref{eq:Hes_estb}.

For \eqref{eq:Smoothbnd7}, the result is a  trivial consequence of
\eqref{eq:Smoothbnd6} if $i=I$. If $i< I$  we write 
\begin{equation*}
  E_I= \chi_{\kappa,K}(M_I)\cdots \cdots\chi_{\kappa,K}(M_{i+1})E_i,
\end{equation*} and then 
\begin{align*}
  Q_i\tilde{f}\big(\wcH\big) E_I&= \parb{\chi_{\kappa,K}(M_I)\cdots
                                        \cdots\chi_{\kappa,K}(M_{i+1})}Q_i\tilde{f}\big(\wcH\big)
                                        E_i\\
 &\qq \qq + \big [ Q_i\tilde{f}\big(\wcH\big), \chi_{\kappa,K}(M_I)\cdots
   \cdots\chi_{\kappa,K}(M_{i+1}) \big ] E_i\\
 &= \vO(r^{0})Q_i\tilde{f}\big(\wcH\big)
                                        E_i  + \vO(r^{\delta- 3/2}).
\end{align*} We conclude by 
\eqref{eq:Smoothbnd6} and \eqref{eq:LocalDecay}. 
  \end{proof}

\subsection{A priori  localization of states in
  $\vH^+_{\ener}$}\label{subsec:A prior  localization}
Consider in the following an arbitrary   $\psi\in \vH^+_{\ener}$ obeying $\psi=g(U(1))\psi$ for some  $g\in C^\infty_{\c}(\T\setminus
\vT_\p)$  (henceforth fixed). Since the set of such 
vectors is  dense
in  $\vH^+_{\ener}$, we are  lead   to `localize' $\psi$. Our final result is  stated  below  in Lemma \ref{lemma:locIN} for the corresponding evolution 
$\psi(t)=U(t)\psi$.

We choose $K>1$ large and $\kappa>0$
 small allowing us, thanks to (\ref{eq:out2b}) (localizing  $M_0$ as well as  $M_1,\dots, M_I$ applied similarly),   to localize $\psi(t)$ uniformly in large   $t>1$ 
 as $\psi(t)\approx E_I^*E_I\psi(t)$  (using  here 
 the factor $E_I$ appearing in 
\eqref{eq:Smoothbnd7}). In addition to (\ref{eq:out2b}) we used  that $\psi\in \vH^+_{\ener}$ to  obtain
 the   localization by the factors of $\chi_-(M_i/K)$, $i=0,\dots, I$. More precisely,  for any  given
$\varepsilon>0$, it follows  that 
for all  small $\kappa> 0$  and
  all large $K>1$ 
   \begin{equation*}
  \limsup_{t\to \infty}\,\norm[\big]{\psi(t)- E_I^*E_I\psi(t)}\leq \varepsilon.
\end{equation*} We will henceforth (including  here Section \ref{sec:Existence of
  channel wave operators}) use a similar  meaning of the
symbol $\psi(t) \approx \psi_1(t)$ for given $\vH$-valued functions  $\psi(\cdot)$ and $\psi_1(\cdot)$.

We need a further  initial  localization in
terms of a  
`channel phase space partition of unity' allowing us to simplify the
evolution. Note that a  factor of $\chi_+(2M_0/\kappa)$ can freely be
inserted (after a commutation) as $E_I^*E_I\to E_I^*
\chi_+(2M_0/\kappa)E_I $. Recalling  \eqref{eq:DECO} we are then lead to
write
\begin{equation}\label{eq:DECPm}
  \chi_+(2M_0/\kappa)\chi_+(|x|)= \sum_{a\in \vA_1}\,h_\kappa(M_0)M_a\chi_+(|x|);\q h_\kappa(t)= \chi_+(2t/\kappa)t^{-1 }.
\end{equation}

The formula (\ref{eq:DECPm}) will be the key to construct a proper
 phase space partition of unity.  The   appearance  of the  factors $M_a$, $a\in \vA_1$,   leads  to the possibility of replacing   $H$ by
$H_a$ on the corresponding parts of the phase space.   To be used below (as well as in Section \ref{sec:Existence of
  channel wave operators})  their   position-space
localization is    conveniently  made explicit as follows.
 Choose for each $a\in\vA_1$ a function  $\hat\xi_a\in
C^\infty(\mathbf{S}^{a_0})$  such that $\hat \xi_a=1$ on
$\mathbf{S}^{a_0}\cap\mathbf  Y_a(\epsilon^{d})\cap \set[\big]{|\omega^a|\leq
\sqrt{6\ \epsilon^{d_a}}}$ and $\hat\xi_a=0$ on
$(\mathbf{S}^{a_0}\setminus  \mathbf Y_a(\epsilon^{d}/2))\cup \set[\big]{|\omega^a|>
2\sqrt{6\ \epsilon^{d_a}}}$.  
 Let  $\xi_a=\xi_a(x)=\hat\xi_a(\hat
x)\chi_+(4\abs{x})$.  It then follows from Lemma
\ref{lemma:ma1} \ref{item:13a}  and  \ref{item:14a} that the operator $M_a$ fulfills   
\begin{align}\label{eq:partM}
  M_a= M_a\xi_a= \xi_aM_a.
\end{align} 

Let us record the  property  
\begin{align}\label{82a09}
 \forall a\in \vA_1:\q f(\wcH)\xi_a-\xi_a f\parb{\wcH_a}=\vO(r^{-1-\min\set{2\mu,\delta}}).
\end{align} 
In fact,  by \eqref{82a0}
\begin{align*}
  &f(\wcH)\xi_a-\xi_a f\parb{\wcH_a} 
=
  \int _{\C}\parb{(\wcH-z)^{-1}\xi_a-\xi_a(\wcH_a-z)^{-1}}\,\mathrm d\mu_f(z)\\
  &\, =
    \int _{\C} (\wcH-z)^{-1}\parb{\xi_a\wcH_a-\wcH\xi_a}(\wcH_a-z)^{-1}\,\mathrm d\mu_f(z)
    \\
  &\, =
    \int _{\C} (\wcH-z)^{-1}\vO(r^{-1-\min\set{2\mu,\delta}})(\wcH_a-z)^{-1}\,\mathrm d\mu_f(z)=\vO(r^{-1-\min\set{2\mu,\delta}}).
\end{align*}

We will below use \eqref{eq:DECPm}-\eqref{82a09} in combination.
 However, since  the  factors
$M_a$ are not bounded  we need to  regularize them before we 
 can implement  \eqref{eq:DECPm}. 

\subsubsection{Channel phase-space partition for $\psi(t)$}\label{subsubsec:A
  partition of unity H}
We introduce the function
\begin{equation*}
  \chi_{-L,L}(t)=
\chi_-(-t/L)\chi_-(t/L);\q  L>1.
\end{equation*} 
 To  regularize 
$M_a $ we first consider the further
localization 
\begin{equation}
  \label{eq:dec1}
  \psi(t)\approx E_I^*  \chi_+(2M_0/\kappa)\tilde{f}\big(\wcH\big)  \chi_{-L,L}(M_{a_1})\cdots \chi_{-L,L}(M_{a_{\#\vA_1}})E_I\psi(t) . 
\end{equation} In the last product we have labelled the elements of
$\vA_1$ in an arbitrary fashion as $a_1, \dots ,a_{\#\vA_1}$,
and we have  chosen
$L>1$ large to meet the requirement of our use of  the
symbol $\approx$ (using again that $\psi\in \vH^+_{\ener}$).  

Next we insert \eqref{eq:DECPm} in the expression to the right in 
\eqref{eq:dec1}, leading to 
\begin{equation*}
  \psi(t)\approx \sum_{a\in \vA_1}\, E_I^*  h_\kappa(M_0) M_a\tilde{f}\big(\wcH\big)\chi_{-L,L}(M_{a_1})\cdots \chi_{-L,L}(M_{a_{\#\vA_1}})\tilde{f}\big(\wcH\big) E_I\psi(t). 
\end{equation*}  We have omitted  factors of  $\chi_+(|x|)$, which is legitimate thanks to Proposition \ref{lemma:scatDef22}. For each $a\in \vA_1$ the corresponding term
to the right  contains the factor $\chi_{-L,L}(M_{a})$, which
after a commutation makes  the appearing factor $M_a$   bounded. We
have proven that 
\begin{align*} 
  \psi(t)&\approx \sum_{a\in \vA_1}\, E_I^* R_aE_I\psi(t);\\ R_a=R_a(t)&=h_\kappa(M_0) \parb{M_a\chi_{-L,L}(M_{a})}\tilde{f}\big(\wcH\big)\prod_{ \vA_1\ni b\neq a}\chi_{-L,L}(M_b)  
\end{align*} with the same product  as before  except
that the factor  $\chi_{-L,L}(M_{a})$ is moved to the left `to meet
$M_a$'. 
By Proposition \ref{lemma:scatDef22} the commutation error vanishes as $t\to \infty$. 

Next we 
replace the  factor ${f}\big(\wcH\big) $ in the   factor $ E^*_I
$ by
$f\parb{\wcH_a}$ by first writing $M_a=\xi_aM_a$ (cf.\eqref{eq:partM})
and then moving  the factor $\xi_a$
to the left  using eventually \eqref{82a09}. Hence
\begin{subequations}
\begin{align} \label{eq:approxINT}
  \psi(t)\approx \sum_{a\in \vA_1}\,S_a\psi(t);\,\, S_a=  S_a(t)={f}\big(\wcH_a\big)
  \chi_{I}(\bM)^*R_a \chi_{I}(\bM)f\parb{\wcH} .
\end{align} 

Recalling  that $g\in C^\infty_{\c}(\T\setminus
\vT_\p)$ we  pick any   function $\tilde{g} \in C^\infty_{\c}(\T\setminus
\vT_\p)$ with $\t
g\succ g$. We can then 
refine this  localization  as follows (see the details of proof below).  
\begin{align} \label{eq:psidecPh}
  \psi(t)\approx \sum_{a\in \vA_1}\,U_a(t-[t])\t
                      g(U_a(1))U_a(t-[t])^*S_a(t)  U(t)g(U(1))\psi.
\end{align}  
\end{subequations}
Hence we are lead to consider for $\kappa> 0$ small and $K,L>1$  large the `partition of unity'
\begin{align}\label{eq:ParUnit}
  \sum_{a\in \vA_1}\, U_a(t-[t])\t g(U_a(1))U_a(t-[t])^*S_a(t).
  \end{align} 
We have almost shown the following result.
\begin{subequations} 
\begin{lemma}\label{lemma:locIN} Let  (as above)  $g,\t g \in C^\infty_{\c}(\T\setminus
\vT_\p)$ with $\t g\succ
g$ and let  $\psi
=g(U(1))\psi \in \vH^+_{\ener}$. Then for any  $\varepsilon >0$ 
 there exist a small $\kappa> 0$  and
  large $K,L>1$    such that for the corresponding  operators $S_a(t)$ (given as in \eqref{eq:approxINT}) and with
  $\psi(t)=U(t)\psi$,
\begin{align}
  \label{eq:dec2}
  \limsup_{t\to \infty}\,\norm[\Big]{\psi(t)- \sum_{a\in \vA_1}\,
    U_a(t-[t])\t g(U_a(1))U_a(t-[t])^*S_a(t)\psi(t)}\leq \varepsilon,
  \\
\limsup_{t\to \infty}\,\norm[\Big]{\psi(t)- \sum_{a\in \vA_1}\,
    U_a(t-[t])\t g(U_a(1))S_a([t])\psi([t])}\leq \varepsilon.  \label{eq:dec22}
\end{align} 
\end{lemma}  
\end{subequations}
\begin{proof} The
  assertion \eqref{eq:dec22} will follow easily from \eqref{eq:dec2}, which
  in turn is  just a  repetition of the missing
  part  \eqref{eq:psidecPh}. 
 
  \Step {I}  Using \eqref{eq:approxINT}  we  show the refinement
  \eqref{eq:dec2}. Essentially this is done by   mimicking  Step II of   the proof of Lemma \ref{lemma:asympOBser}.
Pick  $g_1 \in C^\infty_{\c}(\T\setminus
\vT_\p)$ such that  $\t
  g\succ g_1 \succ g$. Then $(\t g -1) g_1=0$, and recalling  that 
  \begin{align}\label{eq:reduc}
    \begin{split}
     U(t)&=U(t-[t])U([t])=U(t-[t])U(1)^{[t]},\\  U_a(t)&=U_a(t-[t])U_a([t])=U_a(t-[t])U_a(1)^{[t]},  
    \end{split}
  \end{align}
 it  suffices (thanks to \eqref{eq:approxINT}) to show
  that for each $a\in \vA_1$
  \begin{align}\label{eq:BBB}
    \lim_{t\to \infty}\,\norm[\Big]{
   \parbb{U_a(t-[t])^*S_a(t)U(t-[t])g_1(U(1))- g_1(U_a(1))S_a(t)}\psi([t])}=0.
  \end{align} 
By the
  Stone--Weierstrass theorem
  we can  replace the factors $g_1(U(1))$ and  $g_1(U_a(1))$ by 
  arbitrary integer powers $U(1)^k$ and $U_a(1)^k$, respectively. With
  this done and with $\sigma:=t-[t]+k$  we write parallel  to the proof of Lemma \ref{lemma:asympOBser} 
  \begin{align*}
    U_a(\sigma)^*&S_a(t)U(\sigma)-S_a(t)\\
  &= U_a(\sigma)^*S_a({\sigma+[t]-k}) U(\sigma)-S_a({[t]-k})   +\parb{S_a({[t]-k}) -S_a(t)}\\
                 &= \int_0^\sigma\, U_a(s)^*{(\bD^a_1 S_a)({s+[t]-k}})U(s)\,\d \s +\vO_{\unif}(r^{-0_+});\\
    &\q \bD^a_1S_a=\pa_tS_a+\bD^a_2S_a,\q \bD^a_2 S_a=\i(H_aS_a -S_a H).
\end{align*}   Thanks to Proposition \ref{lemma:scatDef22} it suffices  to show that
\begin{align}\label{eq:zeroINF}
    \lim_{t\to \infty}\,\norm[\Big]{
   \int_0^\sigma\, U_a(s)^*{(\bD^1_a S_a)_{s+[t]-k}}\psi(s+[t])\,\d \s} =0.
\end{align} Since $\pa_tS_a=\vO_{\unif}(r^{-1_+})=\vO_{\unif}(r^{-0_+})$ we can here replace $\bD_1^a S_a$  by $\bD_2^a S_a$, in fact by $\i[H,S_a]$.
Here we use the  localization properties of $M_a$ in the appearing factor of $M_a\chi_{-L,L}(M_{a})$ (cf. (\ref{eq:partM}) and (\ref{82a09}) and the surrounding discussion). For the same reason we can in $\i[H,S_a]$ replace the left factor ${f}\big(\wcH_a\big)$ of $S_a$ by ${f}\big(\wcH\big)$. (This modification of  $S_a$ will be used tacitly  below.)

We can compute   this commutator  along the lines of Step VI in the proof of Lemma \ref{lemmaNegForbidden}. In particular the starting point will be the replacement of $H$ by $m^\delta\wc f\parb{\wcH} m^\delta$. As before  the computation is based on the product rule and   involves several terms for which
we need various already established integral bounds, including (\ref{eq:Smoothbnd19G25})-\eqref{eq:GenQ}.  We  take  inner product with any $\varphi\in \vH$, expand  the commutator  and then apply  \cas (as in Step II of   the proof of Lemma \ref{lemma:asympOBser}). For terms  in  the   expansion of $\i[H,S_a]$ (more precisely of $\i[m^\delta\wc f\parb{\wcH} m^\delta,S_a]$) not directly treatably by Proposition \ref{lemma:scatDef22} one  needs a  factorization  of the form $Q_1^*T_1$ or $Q_2^*T_2$ with $Q_1$ and $Q_2$ indeed treatable by   Corollary \ref{cor:local-strictly-genG5} and Remark \ref{remark:integral-estimates6}  (we give some elaboration below).  Similarly the factors  $T_1$ or $T_2$ need expansions into terms of the form $BQ$, where $B$ is bounded and the involved $Q$-operators are treatable by  Corollaries
\ref{cor:smoothEst} and  \ref{cor:smoothEst3} as well as by   Lemma 
\ref{lemma:concretLem}. In the end the above integral is estimated  this way as $o(t^0)\norm{\varphi}$ as $t\to \infty$, proving  (\ref{eq:zeroINF})  and therefore also \eqref{eq:BBB}. Consequently   \eqref{eq:dec2} follows from this scheme  of proof.

It remains to see that the outlined factorization procedure is doable. 
Let us first recall that  in the proof of Lemma \ref{lemma:asympOBser} the exact computation  \eqref{eq:posBND9} lead to 
\begin{align*}
    \lim_{t\to \infty}\,\norm[\Big]{
   \int_0^\sigma\, U(s)^*Q^*_{j, s+[t]-k}Q_{j, s+[t]-k}\psi(s+[t])\,\d \s} =0;\q j=1,2.
\end{align*} Here the  factorization $Q^*_{j, s+[t]-k}Q_{j, s+[t]-k}$ (coming from \eqref{eq:posBND9}) was treated  by a simplified version of the discussed  scheme.
Since the present observable $S_a$ has a more complicated structure (containing a wider mix of  factors) the factorization for $\i[H,S_a]$ is more involved. In fact we also have to deal with the  `square root problem' encountered  in Subsection~\ref{subsec:A complicated proof}. The latter comes about when dealing with contributions from $\i[H,h_\kappa(M_0)]$ and the factors of $\i[H,\chi_{\kappa,K}(M_i)]$ of $\chi_I(\bM)$ and  $\chi_I(\bM)^*$, $i=0,\dots, I$. In all cases  the `dangerous factor' $4\parb{p\cdot 
  \nabla^2m_i(x)p}$ enters. We need again the regularization by $W=\inp{p}^{1/3}$, and using the  notation
\begin{align*}
  Q_{2,i,\chi, \rho,W}&=Q_{2,i,\chi, \rho,W}(t)=2
\sqrt{P_i}\sq{\chi'}(M_i)  E_{i-1}W,\\
P_i=P_{i,\rho,W}(t) &=\tilde{f}\big(\wcH\big) W^{-1}\parb{p\cdot 
\nabla^2m_i(x)p}W^{-1}\tilde{f}\big(\wcH\big)+r^{-2\rho}
\end{align*}
of \eqref{eq:altBnd},  we obtain after commutation to the right a factor of the form
\begin{align*}
  Q_{2,i,\chi, \rho,W}(s+[t]-k)\psi(s+[t])=Q_{2,i,\chi, \rho,W}(s+[t])\psi(s+[t]).
\end{align*} We can then use that  
\begin{align*}
\begin{split}
  \int_{0}^\sigma &
\norm[\big]{Q_{2,i,\chi, \rho,W}(s+[t]-k)\psi(s+[t])}^2\,\d s \\
  &=\int_{[t]}^{\sigma+[t]}
\norm[\big]{Q_{2,i,\chi, \rho,W}(s)U(s)\psi}^2\,\d s \to 0 \text{ for } t\to \infty.   
\end{split}
\end{align*} It remains in this case to pull a  factor $(\sq{\chi'}(M_i) {f}\big(\wcH\big)W)^* \sqrt{P_i}$ to the left to combine with $U_a
(s)\varphi$ using \eqref{eq:GenQ}, and for that we need a version of the alternative proof of \eqref{eq:Smoothbnd4} given in Subsection \ref{subsec:A complicated proof} to work out the commutation (we comment on this issue below). Once placed to the left we use the bound
\begin{align}\label{eq:commmm}
  \int_{0}^\sigma \norm[\big]{ \sqrt{P_i} \sq{\chi'}(M_i) {f}\big(\wcH\big)W \xi_a\varphi_a(s)}^2\,\d s \leq C\|\varphi\|^2 ;\q \varphi_a(s)=U_a(s)\varphi.
\end{align} Here $\xi_a$ is the localization function from (\ref{eq:partM}) and (\ref{82a09}  (which clearly can be moved to the left too). It allows us to replace  any appearance of $\wcH$ in (\ref{eq:commmm}) by $\wcH_a$. Then (\ref{eq:commmm}) is a consequence of \eqref{eq:GenQ} (applied to  $H_a$ rather than to $H$)  and a commutation with $W^{-1}$ as in  the proof of Corollary \ref{cor:smoothEst2b} (using now that $W{f}\big(\wcH_a\big)=\vO(r^{\delta/3})$).

For the commutation error we can largely mimic Steps IV and V in  our  alternative proof
of \eqref{eq:Smoothbnd4}. However presently we have the problem that 
the factor ${f}\big(\wcH\big)W=\vO(r^{\delta/3})$  is not bounded. On the other hand we now only need \eqref{eq:LocalDecay} one-sided, which yields an extra factor  $r^{-\rho}$.
Viewed  this way the fact  that $r^{\delta/3-\rho}$ is bounded   allows us to mimic the procedure  of  commutation given there.

We conclude  the desired bound  $o(t^0)\norm{\varphi}$ for the discussed  contributions to the commutator of $S_a$.

        The  contribution from  the two appearances of $\i[H,{f}\big(\wcH\big)]$  is  dealt with by (\ref{eq:Smoothbnd1}) and \eqref{eq:Smoothbnd19G25}, while the contribution from  $\i[H,\tilde{f}\big(\wcH\big)]$  is treated by \eqref{eq:ExpsBreveH}.
        
        The third  class of contributions comes from   the commutators $\i[H,{M_a\chi_{-L,L}(M_{a})}]$ and $\i[H,{\chi_{-L,L}(M_{b})}]$, $b\neq a$. For either case we use \eqref{eq:Hes0}, writing 
        \begin{align*}
  \begin{split}
 p\cdot\parb{\chi^2_+(\abs{x})\nabla^2m_a(x)} p
=\sum_{j\leq J}\, \,Q(a,j)^*\vG_j Q(a,j);\quad 
                                                                  \vG_j=\vG_j(x_{b_j})\text{ 
                                                                  bounded}.   
  \end{split}
        \end{align*} After commutation we apply Lemma 
        \ref{lemma:concretLem} to treat to integral involving $\psi$ (after we used \caS).  For the  integral involving $\varphi $ we first use (\ref{eq:Hes_est}) and then \eqref{eq:GenQ}  on the form
    \begin{equation}
      \label{eq:GenQ9}
      \int_{0}^\sigma
\norm[\big]{2\sqrt{p\cdot \nabla^2m_i(x) p}\,\tilde{f}\big(\wcH_a\big){\chi_{\kappa,K}}(M_i)f\parb{\wcH_a} \varphi_a(s)}^2\,\d s \leq C \|\varphi\|^2 .
\end{equation} Note at this point that commuting with $\,Q(a,j)$ is rather straightforward (there is no `square root problem' here).
        
   Up to  missing 
 details on bounding various commutation errors (for convenience omitted)  we have proven  
 \eqref{eq:BBB}. Consequently   \eqref{eq:dec2} follows. 
 
\Step {II}  We show \eqref{eq:dec22}. Since
$S_a(t)=S_a([t])+\vO(r^{-2\delta})$, it suffices (thanks to
\eqref{eq:dec2}) to show that 
\begin{align}\label{eq:BBBb}
    \lim_{t\to \infty}\,\norm[\Big]{
   \parbb{U_a(t-[t])^*S_a(t)U(t-[t])- S_a(t)}\psi([t])}=0.
  \end{align}  Clearly
  \eqref{eq:BBBb} amounts to \eqref{eq:BBB} for the case  $g_1=1$ (corresponding to taking $k=0$ in the above proof).
\end{proof}

\subsection{Proof of Proposition \ref{prop:negB}}\label{subsec:Proof
  of Proposition}
By density it suffices to consider any  given state $\psi
=g(U(1))\psi \in \vH^+_{\ener}$ (as in Lemma \ref{lemma:locIN}). Let  also $\varepsilon >0$ be
given. Under the conditions of Lemma \ref{lemma:locIN} (with the given 
$\psi$ and $\varepsilon$) we  introduce for each
$a\in \vA_1$
\begin{subequations}
\begin{equation}\label{eq:1varphi}
  \varphi_{a,\varepsilon}=\lim_{t\to \infty}\,
    U_a(t)^*U_a(t-[t])\t g(U_a(1))S_a([t])\psi([t]) \in \vH.
\end{equation} 
  Thanks to \eqref{eq:dec22} it suffices to show that  
 $\varphi_{a,\varepsilon}$ is a well-defined vector.   
 Using \eqref{eq:reduc} this amounts  to showing  the existence of the limit
\begin{align*}
  \lim_{t\to \infty}\,
    U_a([t])^*\t g(U_a(1))S_a([t])\psi([t])=\lim_{t\to \infty}\,
    \t g(U_a(1))U_a([t])^*S_a([t])\psi([t]).
\end{align*} Hence    
(obviously) it suffices to show
the existence of the limit
\begin{align}\label{eq:varphiLimi}
  \varphi_{a,\varepsilon}= \lim_{t\to \infty}\,
    \t g(U_a(1))U_a(t)^*S_a(t)\psi(t).
\end{align} 
\end{subequations}

 For the latter  goal  let us introduce   the evolution
$\wt\varphi_a(t)=U_a(t)\t g(U_a(1))\varphi$, $\varphi\in \vH$. To
verify 
the  Cauchy condition,  mimicking Step III of  the proof of Lemma \ref{lemma:asympOBser}, we write for big $t_2>t_1>1$
\begin{align*}
  \inp[\big]{\wt\varphi_a(t_2), S_a(t_2)\psi(t_2)}- \inp[\big]{\wt\varphi_a(t_1),
  S_a(t_1)\psi(t_1)}=\int^{t_2}_{t_1}\,
  \tfrac{\d} {\d t}\inp{\wt\varphi_a(t), S_a(t)\psi(t)}\,\d t.
\end{align*} 
Next we compute the derivative as in the proof of Lemma \ref{lemma:locIN}, expand  (using in particular \eqref{eq:Hes0})  and
then estimate by  Corollaries
\ref{cor:smoothEst} and  \ref{cor:smoothEst3} as well as  by  Lemma 
 \ref{lemma:concretLem}. We apply the bounds for $H$  (as explicitly stated) as well as similar  bounds 
for $H_a$. (The latter bounds  are  just  special cases.) Indeed (more specifically) the contribution from the Heisenberg derivative of the factors of  functions of $M_a$ and $M_b$, $b\neq a$,  and the factor of $h_\kappa(M_0)$ appearing in  the expression $R_a$ are dealt with by \eqref{eq:Hes0} and  \eqref{eq:Smoothbnd7}, respectively \eqref{eq:Smoothbnd5} (applied with $i=0$). The contribution from the  factors of $f(\wcH)$ and $f(\wcH_a)$ are dealt with by \eqref{eq:Smoothbnd1}, while the contribution from the  factors of $\chi_{I}(\bM)$ are treated by \eqref{eq:Smoothbnd50}. As before the contribution from  $\i[H,\tilde{f}\big(\wcH\big)]$  is treated by \eqref{eq:ExpsBreveH}. Of course these 
applications need various commutation arguments. In particular the mentioned application of \eqref{eq:Smoothbnd50} needs commutation using the scheme from our alternative proof
of \eqref{eq:Smoothbnd4} in Subsection \ref{subsec:A complicated proof}. As in the proof of Lemma \ref{lemma:locIN} we    skip the
details on bounding commutation errors.

Thanks
to the standard Kato type form of our bounds  the above integral is bounded this way
as $o(t_1^0)\norm{\varphi}$ as $t_1\to \infty$, uniformly in $t_2>t_1$
and  
$\varphi\in \vH$. This yields the existence of the limit
   \eqref{eq:varphiLimi}.
 \qed

\section{Proof of Theorem \ref{thm:ACn-body-short1}}\label{sec:Existence of
  channel wave operators}  We prove the existence and orthogonality of the forward  channel wave
operators. For simplicity we will only consider this case (the backward case can be treated similarly, or by the time-reversal property).

\Step {I} We  prove the existence of the forward  channel wave
operators. This is done essentially along the lines of the proofs of Lemma \ref{lemma:locIN} and Proposition \ref{prop:negB}.

Recall that for any  \emph{channel}
$\beta =(b,\lambda^\beta, u^\beta)$, i.e.   $b\in\vA_1$, $\lambda^\beta\in [0,2\pi)$ and  
$U^b(1)u^\beta = \e^{-\i \lambda^\beta }u^\beta$ for a normalized $u^\beta\in
\mathcal H^b$,   the corresponding forward  {channel wave operator} is
given by 
\begin{subequations}
\begin{equation}\label{eq:wave_op3}
  W_\beta^{+}\psi_b=\lim_{t\to \infty}U(0,t)\parbb{\parb{\e^{-\i \lambda^\beta [t]}U^b(t- [t])u^\beta} \otimes
\parb{\e^{-\i
   tp_b^2} \psi_b}};\quad \psi_b\in L^2(\bX_b).
 \end{equation}  Alternatively written 
\begin{align}\label{eq:wave_op4}
  W_\beta^{+}\psi_b=\lim_{t\to \infty}U(0,t)U_b(t)\parb{u^\beta \otimes
\psi_b};\quad \psi_b\in L^2(\bX_b).
 \end{align}  
\end{subequations}

By density it suffices to show the existence of the latter limit  for any Schwartz
function, viz   $\psi_b\in \vS(\bX_b)$.  In fact we can  assume that
its Fourier
 transform $\widehat \psi_b =F_b\psi_b \in C^\infty_\c(\bX_b)$,  
 that the vector $\varphi_b:={u^\beta \otimes
\psi_b}$ satisfies $\varphi_b=g(U_b(1))\varphi_b$ for some 
$g\in C^\infty_{\c}(\T\setminus
\vT_\p)$ and that for some (small) $\epsilon'>0$ 
\begin{equation}\label{eq:collision}
\supp \widehat \psi_b
\subseteq \bX_b\cap  \bY_b(\epsilon')=(\bX_b\setminus\set{0})\setminus \cup_{c\not\leq b,\,c\in\vA_1} \,\bX_c(\epsilon').
\end{equation}
Denoting
$\varphi_b(t)=U_b(t)\varphi_b$ we need to show the existence of the limit 
\begin{equation*}
 \lim_{t\to \infty}U(0,t)U_b(t)\varphi_b=\lim_{t\to \infty}U(0,t)\varphi_b(t).
 \end{equation*} 
As in the proof of Proposition \ref{prop:negB} we will verify the
  Cauchy condition.

First we introduce  a
version of 
 the `partition of unity' $\sum_{a\in \vA_1}\,S_a(t)$ of
 \eqref{eq:approxINT}, more precisely on the form $\sum_{a\in \vA_1}\,S^b_a(t)$,
 where  essentially given  by the same `building blocks'  as in
  \eqref{eq:approxINT} we let  
  \begin{align*}
    S^b_a(t)&={f}\big(\wcH\big)
  \chi_{I}(\bM)^*R^b_a(t) \chi_{I}(\bM)f\parb{\wcH_b};\\
 R^b_a(t)&=h_\kappa(M_0) \parb{M_a\chi_{-L,L}(M_{a})}\tilde{f}\big(\wcH_b\big)\prod_{ \vA_1\ni c\neq a}\chi_{-L,L}(M_c).  
  \end{align*}  Note in particular that this construction involves
   small parameters $\epsilon=\epsilon_0$,  $\epsilon_1, \dots, \epsilon_I$, $\kappa> 0$  and
  large  parameters $K,L>1$, which can be adjusted for  verifying the
  Cauchy condition.  We use 
  \eqref{eq:collision}   for an    $\epsilon<<\epsilon'$
  (in fact $2 \sq{6\epsilon}< \sq{1-(1-\epsilon')^2}$ suffices). Then
    Lemma \ref{lemma:ma1}
  \ref{item:13a} and \ref{item:14a} (applied with this $\epsilon$ and  implemented as in
  \eqref{eq:partM} and \eqref{82a09})  and stationary phase analysis
  of the state $\e^{-\i
   tp_b^2} \psi_b$  yield (leaving out the details of proof) 
   that 
  \begin{align*}
    \varphi_b(t)\approx \sum_{a\in \vA_1}\,S^b_a(t)\varphi_b(t)\approx S^b_b(t)\varphi_b(t).
  \end{align*} 
  Here we  used  the symbol $\approx$ in the same meaning as in Subsection \ref{subsec:A prior  localization}.
  
Next, arguing as in the proof of Lemma \ref{lemma:locIN}, 
   it follows that for any $\tilde{g}\in C^\infty_{\c}(\T\setminus
\vT_\p)$   with  $\t g\succ g$
\begin{align*}
    \varphi_b(t)\approx U(t-[t])\t g(U(1))S^b_b([t])\varphi_b([t]),
  \end{align*} and it remains to show the  existence of the limit 
\begin{equation*}
 \lim_{t\to \infty}U(0,[t])\t g(U(1))S^b_b([t])\varphi_b([t])=\lim_{t\to \infty}\t g(U(1))U(0,t)S^b_b(t)\varphi_b(t).
 \end{equation*} 

 Note that this reduction is completely similar to the reduction to   the existence of the limit
   \eqref{eq:varphiLimi}  in the proof of Proposition
 \ref{prop:negB}. From here we can complete the proof in the same way as in the proof of Proposition
 \ref{prop:negB}, recalling that  the involved 
 integral estimates appeared on a symmetric form. 
 Hence the first part of  Theorem
 \ref{thm:ACn-body-short1} follows. 

  \Step {II}  We prove the second  part of  Theorem
 \ref{thm:ACn-body-short1}, i.e. the  orthogonality
  of channels (for the forward case). The assertion  for the
  time-independent case  appears as \cite[Theorem XI.36 (b)]{RS}. It can  essentially  be proven similarly for
  smooth potentials, to be demonstrated below.

  Consider two different channels $\alpha
  =(a,\lambda^\alpha, u^\alpha)$ and $\alpha =(b,\lambda^\beta,
  u^\beta)$. The orthogonality is obvious if  $a= b$. So we consider
  only the case  $a\neq b$. Let 
  \begin{align*}
    \varphi_\alpha(\sigma)= \parb{U^a(\sigma)u^\alpha} \otimes
\psi_a \mand \varphi_\beta(\sigma)= \parb{U^b(\sigma)u^\beta} \otimes
\psi_b ;\q \sigma=t-[t].
  \end{align*}
We need  to show that  
\begin{equation}\label{eq:bnVanish}
  \inp[\big]{\varphi_\alpha(\sigma), \e^{-\i t (p_b^2-p_a^2)}
  \varphi_\beta(\sigma)}\to 0\text{ for }t\to \infty.
\end{equation} 

By approximation we can assume that $u^\alpha \in
\vS(\bX^a)$ and  $\psi_a\in \vS(\bX_a)$, and similarly when $\alpha$
is replaced by $\beta$.
By the standard convolution technique and familiar Duhamel type computations (cf. (\ref{eq:Duhamel})) we can furthermore assume that
 Condition \ref{cond:smooth2wea3n12} is valid with all `pair potentials'
 being smooth in $x$ with uniformly bounded derivatives. With  these
 assumptions it is easy to check by  Duhamel type computations as in (\ref{eq:Duhamel})
 that
 \begin{equation}
   \label{eq:kbnd} \forall s\in \R:\q\sup_{\sigma\in [0,1)}\,
   \norm{\varphi_\alpha(\sigma)}_{L^2_s(\bX)}< \infty\mand \sup_{\sigma\in [0,1)}\,
   \norm{\varphi_\beta(\sigma)}_{L^2_s(\bX)}< \infty.
 \end{equation} We decompose $\bX=\bY\oplus \bZ$, where $\bZ$ is
 determined by first writing $p_b^2-p_a^2=\Sigma_{1\leq j\leq n}\, c_j
 p_j^2$ with all  $c_j\neq 0$,  corresponding to an orthonormal basis
 in a certain $n$-dimensional subspace $\bZ$. By using \eqref{eq:kbnd} for any $s> n/2$ and
 \caS, it follows 
 that 
\begin{equation}
   \label{eq:kbnd2} \sup_{\sigma\in [0,1)}\,
   \norm{\varphi_\alpha(\sigma)}_{L^2(\bY, L^1(\bZ))}< \infty\mand \sup_{\sigma\in [0,1)}\,
   \norm{\varphi_\beta(\sigma)}_{L^2(\bY, L^1(\bZ))}< \infty.
 \end{equation} Writing  
\begin{align*}
  \inp[\big]{\varphi_\alpha(\sigma), \e^{-\i t (p_b^2-p_a^2)}
  \varphi_\beta(\sigma)}=\int_{\bY}\, \inp[\Big]{\varphi_\alpha(\sigma,y, \cdot),  \prod_{1\leq j\leq n}\,\e^{-\i t c_j
 p_j^2}
  \varphi_\beta(\sigma,y, \cdot)}_{L^2(\bf Z)}\, \d y,
\end{align*} we  estimate the right-hand side,   
first by  
\eqref{eq:kbnd2} and then by appying \caS. Note that each one-dimensional  free dynamics
contributes  by a
factor $t^{-1/2} $, yielding the total bound $\vO(t^{-n/2})$. In
particular \eqref{eq:bnVanish} follows.
\qed
\section{Proof of Theorem
  \ref{thm:ACn-body-short2}}\label{sec:Proof of
  Theorem (2)}   Consider for $N=2$ vectors    $\psi\in \vH$  obeying  $\psi=g(U(1))\psi$ for some $g\in C^\infty_{\c}(\T\setminus
\vT_{\p})$. The set of these  vectors
 is  dense
in  $\vH_{\ac}=\vH_{\ac}(H)$. We  will show that any such given 
$\psi\in\vH^+_{\wave}=\vH^+_{\wave}(H)$, proving Theorem
\ref{thm:ACn-body-short2}. 

Let $\psi(t)=U(t)\psi$. Pick    $\t g\succ g$ such
that  also $\tilde{g}\in C^\infty_{\c}(\T\setminus
\vT_{\p} )$. It then holds that 
\begin{equation}\label{eq:dec2233}
  \lim_{t\to \infty}\,\norm[\big]{\psi(t)- 
    U_{a_{\min}}(t-[t])\t g(U_{a_{\min}}(1))\psi([t])}=0. 
\end{equation} This  is  a simplified version of  \eqref{eq:dec22}, proven similarly although  easierly  since   the quantity 
 $\bD^a_1 S_a$ appearing in  the proof of Lemma \ref{lemma:locIN} now is replaced by $-\i V$ (which in fact may be treated directly by  Proposition \ref{lemma:scatDef22}).

Next, 
 arguing as in the proof of Proposition \ref{prop:negB},  we deduce 
 the existence of 
 the limit 
\begin{equation}\label{eq:phi2}
  \varphi=\lim_{t\to \infty}\,
    U_{a_{\min}}(t)^*U_{a_{\min}}(t-[t])\t g(U_{a_{\min}}(1))\psi([t]).
  \end{equation}

  It  follows from \eqref{eq:dec2233} and (\ref{eq:phi2})  that
\begin{equation*}
   \norm[\big]{\psi-W^+_{a_{\min}} 
    \varphi }=\lim_{t\to \infty}\,\norm[\big]{\psi(t)- 
    U_{a_{\min}}(t)\varphi }=0. 
\end{equation*} Hence $\psi\in \ran\parb{ 
    W^+_{a_{\min}}} =\vH^+_{\wave}$.
 \qed

\section{Proof of Theorem \ref{thm:ACn-body-short3} and remaining results}\label{sec:Proof
  of Theorem (N)}
We prove our main result Theorem \ref{thm:ACn-body-short3}. Viewed as a criterion for   asymptotic completeness we provide another  one,   and we   prove Propositons \ref{thm:ACn-body-short3EXC} 
and \ref {prop:asympt-compl-anoth9} (which can also be viewed  as
criteria). The principal tool is for all issues  variations  of Proposition \ref{prop:negB}.

\subsection{Proof of Theorem \ref{thm:ACn-body-short3} }\label{sec:Proof
  of Theorem (N1)}
  We proceed by induction in $N$.
  For $N=2$ the  assertion (\ref{eq:mainR}), i.e. that $\vH^+_{\wave}(H)=\vH^+_{\ener}(H)$, is a consequence of  Theorem \ref{thm:ACn-body-short2}. Suppose $N\geq 3$ and that (\ref{eq:mainR}) 
 is true for  $N'\leq N-1$, then we need to show it for
  $N'=N$. By this  induction hypothesis it follows  that $\vH^+_{\wave}(H^a)=\vH_{\ener}^+(H^a)$ for all
  $a\in \vA_3$.
   
 Let  $\psi\in \vH^+_{\ener}(H)$ be given. We need to show that $\psi\in\vH^+_{\wave}(H)$. Letting also $\varepsilon >0$ be given,  by invoking   Proposition \ref{prop:negB} and Remark \ref{remark:asymptClust} it follows that   
 there exists a family of states
  $\set[[big]{\varphi_{a,\varepsilon}\in \vH^+_{\ener}(H_a) \mid a\in \vA_1}$ such that 
   \begin{equation}
\label{eq:scatdeepsilon9}
      \limsup _{t\to \infty}\, \norm[\Big]{U(t)\psi-\sum_{a\in \vA_1}
        U_a(t)\varphi_{a,\varepsilon}}\leq \varepsilon .
   \end{equation} 
 We claim that there exist  $\vHlim_{t\to \infty} U(0,t)
   U_a(t)\varphi_{a,\varepsilon}$,
 and in fact that these limiting vectors are in
  $\vH^+_{\wave}(H)$. Since $\varepsilon>0$ is arbitrary and 
$\vH^+_{\wave}(H)$ is  closed in $\vH$, it  then follows that 
$\psi\in\vH^+_{\wave}(H)$, proving the  theorem.

For $a=a_{\min}$,  clearly 
\begin{equation*}
    \lim_{t\to \infty} U(0,t)
   U_a(t)\varphi_{a,\varepsilon}=W_{\alpha_{\min}}^+
   \varphi_{a,\varepsilon} \in \vH^+_{\wave}(H).
\end{equation*} So let us henceforth   fix any  $a\in \vA_3$. We need to prove that $\vHlim_{t\to \infty} U(0,t)
   U_a(t)\varphi_{a,\varepsilon}$ exists and represents a vector in  $\vH^+_{\wave}(H)$, completing the proof.

Writing $I=g+\bar g$, $g\in C^\infty_{\c}(\T\setminus
\vT_\p(U(1)))$, and recalling that $U_a(t)=U^a(t)\otimes \e^{-\i tp_a^2} $, 
    we  decompose 
   \begin{equation*}
   U_a(t)=\parb{U^a(t)g(U^a(1))}\otimes \e^{-\i tp_a^2}
   +\parb{U^a(t)\bar g(U^a(1))}\otimes \e^{-\i tp_a^2}  . 
   \end{equation*} We apply this formula to  a sequence of
   growing  non-negative functions $g_n\uparrow 1-1_{\vT_\p}$
(and hence that $\bar g_n\downarrow 1_{\vT_\p}$).
By the spectral theorem it follows that
\begin{equation*}
  \sup_{t\in \R}\, \norm[\Big]{\parb{U^a(t)\bar g_n(U^a(1))\otimes
      I}\varphi_{a,\varepsilon}-\sum_{\substack{\alpha =(a,\lambda^\alpha,
      u^\alpha);\\a\text{ fixed}}}\parb{\e^{-\i \lambda^\alpha [t]}U^a(t-
      [t])u^\alpha} \otimes\inp {u^\alpha, \varphi_{a,\varepsilon}} }\to 0.
\end{equation*} In particular with
$\psi_\alpha^a=\psi_{\alpha,\varepsilon}^a:=\inp {u^\alpha,\varphi_{a,\varepsilon}}\,(=\inp {u^\alpha, \varphi_{a,\varepsilon}
}_{\vH^a}\in L^2(\bX_a))$
\begin{equation*}
  \limsup_{t\to\infty}\, \norm[\Big]{U(0,t)\parb{U^a(t)\bar g_n(U^a(1))\otimes
      \e^{-\i tp_a^2} }\varphi_{a,\varepsilon}-\sum_{\alpha =(a,\lambda^\alpha,
      u^\alpha);\, a\text{ fixed}}\, W^+_\alpha \psi_\alpha^a}\to 0.
\end{equation*} For $n$ taken sufficiently large this introduces an
error of any prescribed order. 

Next we write the orthogonal projection onto $\vH^+_{\wave}(H^a)$ as 
\begin{equation*}
      \sum_{\beta=(b,\lambda^\beta,
      u^\beta),\, b\lneq a}^\oplus \,
      P^a_\beta;\q  P^a_\beta =W_\beta^+(H^a)W_\beta^+(H^a)^*.
    \end{equation*}
We
claim  for any fixed (large) $n$  
\begin{equation}\label{eq:formEN}
  \varphi^n_{a,\varepsilon}:=\parb{g_n(U^a(1))\otimes I} \varphi_{a,\varepsilon} =\sum_{\beta=(b,\lambda^\beta,
    u^\beta),\, b\lneq a}\,
      \parb{P^a_\beta g_n(U^a(1))\otimes I} \varphi_{a,\varepsilon}.  
    \end{equation} To see this we    let
    $\set{\psi_{a,1},\psi_{a,2}, \dots}$ be an orthormal basis in
    $L^2({\bX_a})$. First we prove  that for each $j\in \N$ the vector
    \begin{equation}\label{eq:Egi}
      \psi^a_j:=\inp{\psi_{a,j},\varphi^n_{a,\varepsilon}}_{L^2({\bX_a})}\in\vH^+_{\ener}(H^a).
    \end{equation} Clearly $\psi^a_j\in\vH_{\ac}(H^a)$. Moreover  
  \cas yields
\begin{align*}
     \norm{F((p^a)^2>E)U^a(t)\psi^a_j}&\leq
      \norm{\e^{-\i t p_a^2}\psi_{a,j}}\,\norm{F((p^a)^2>E)U_a(t)\varphi^n_{a,\varepsilon}}
      \\
&\leq
      \norm{F(p^2>E)U_a(t)\varphi^n_{a,\varepsilon}}.
    \end{align*}  Since $\varphi_{a,\varepsilon}\in
    \vH^+_{\ener}(H_a) $ also $\varphi^n_{a,\varepsilon}\in
    \vH^+_{\ener}(H_a) $, and therefore indeed \eqref{eq:Egi} follows.

    By combining \eqref{eq:Egi} and  the property
    $\vH^+_{\ener}(H^a)=\vH^+_{\wave}(H^a)$,  we then  deduce  that $
    \psi^a_j\in \vH^+_{\wave}(H^a)$. Consequently, when taking partial
    inner product with  the  $j$'th basis
    function on both sides of  \eqref{eq:formEN}, the resulting
    formula  
    is correct. Since this is valid for all $j\in \N$,  we  finally conclude    \eqref{eq:formEN} by the Parseval
    inversion formula.   

    By the orthogonality  of channels we can   replace the  sum by a finite
sum (up to another  error of the prescribed order), say denoted
\begin{equation*}
  \sum_{\beta,\,\text {finite}}^\oplus \,
      \parb{P^a_\beta g_n(U^a(1))\otimes I} \varphi_{a,\varepsilon} .
\end{equation*} Then we compute for each of the finitely many terms,
for some $\psi_\beta^a=\psi_{\beta,\varepsilon}^a\in
L^2(\bX_b)$
\begin{equation*}
  \lim_{t\to \infty }\, \norm[\Big]{\parb{U^a(t)P^a_\beta g_n(U^a(1))\otimes I}
    \varphi_{a,\varepsilon} -\parb{\e^{-\i \lambda^\beta [t]}U^b(t-[t])u^\beta} \otimes \parb{\parb{\e^{-\i t(p^a_b)^2} \otimes I} \psi_\beta^a}}= 0.
\end{equation*}   Indeed, this holds with
\begin{equation*}
  \psi_\beta^a=\parb{W_\beta^+(H^a)^*g_n(U^a(1))\otimes I}
  \varphi_{a,\varepsilon} \in L^2(\bX_b^a) \otimes L^2(\bX_a)=L^2(\bX_b) .
\end{equation*}
 Reintroducing the
factor $I\otimes \e^{-\i tp_a^2} $, we conclude that 
\begin{equation*}
  \lim_{t\to \infty }\, \norm[\Big]{\parb{U^a(t)P^a_\beta g_n(U^a(1))\otimes \e^{-\i tp_a^2} }
    \varphi_{a,\varepsilon} -\parb{\e^{-\i \lambda^\beta [t]}U^b(t-
      [t])u^\beta} \otimes \parb{\e^{-\i tp_b^2}  \psi^a_\beta}}= 0,
\end{equation*} and therefore that 
\begin{equation*}
  \lim_{t\to \infty }\, \norm[\Big]{U(0,t)\parb{U^a(t)P^a_\beta g_n(U^a(1))\otimes \e^{-\i tp_a^2} }
    \varphi_{a,\varepsilon} -W^+_\beta \psi^a_\beta}= 0.
\end{equation*}
 
In conclusion,  the  limiting vector  $\lim_{t\to \infty} \,U(0,t)
   U_a(t)\varphi_{a,\varepsilon}\in \vH$ exists, and   up to an error
   of  any prescribed order we have written  this    vector as a  finite sum of vectors on the form $W^+_\beta
\psi_\beta$; $b\in \vA_1$. Thanks to  the closedness of 
$\vH^+_{\wave}(H)$  it follows that the limiting  vector  is  in
$\vH^+_{\wave}(H)$, as wanted. Consequently  also    
$\psi\in\vH^+_{\wave}(H)$. \qed 

\subsection{Asymptotic completeness, another criterion}\label{subsec:Asymptotic completeness, another potential scheme}
 The above proof of Theorem \ref{thm:ACn-body-short3}
   relies on Proposition \ref{prop:negB}
 and Remark \ref{remark:asymptClust}. 
 However the proof does not need the full strength of these assertions.  Thus   the  proof
   works well if for any given 
   $\psi^a\in \vH^+_{\ener}(H^a)$ and   $\varepsilon >0$ we   know  the existence of a family of states
  $\set[[big]{\varphi^a_{b,\varepsilon}\in \vH^+_{\ener}(H^a_b) \mid b\lneq a}$ (in the  above inductive argument used with 
  $a=a_0$)  such that 
   \begin{equation}
\label{eq:scatdeepsilon92}
      \liminf _{t\to \infty}\, \norm[\Big]{U^a(t)\psi^a-\sum_{ b\lneq a}
        U^a_b(t)\varphi^a_{b,\varepsilon}}\leq \varepsilon.
   \end{equation} Here, for a general $a\in \vA_2$ and
   $ b\lneq a$,  the dynamics $U^a_b(t)$ is  generated  by $H^a_b=H^b \otimes I +I\otimes
   (p^a_b)^2$. Note that \eqref{eq:scatdeepsilon92} is slightly weaker than (\ref{eq:scatdeepsilon}) (applied to all elements of $\vA_2$).
   
 By Proposition \ref{prop:negB}  we know \eqref{eq:scatdeepsilon92}  for $\psi^a\in
   \vH^+_{\ener}(H^a)$, however it could be that the assertion is
   valid under the milder condition that  $\psi^a\in
   \vH_{\ac}(H^a)$  and say for a family of  states  $\varphi^a_{b,\varepsilon}$ just known to be elements of    $\vH^a $.  Let us   formulate
   \eqref{eq:scatdeepsilon92} this way as an hypothesis valid for any   $a\in
   \vA_2$, $\psi^a\in \vH_{\ac}(H^a)$ and  $\varepsilon >0$. 
   \begin{proposition}\label{prop:asympt-compl-anoth} The following assertions  $i)$ and $ii)$ are equivalent.
  \begin{enumerate}[i)]
  \item\label{item:1asymp4}  $\forall  a\in \vA_2\,\forall \psi^a\in
    \vH_{\ac}(H^a)\, \forall \varepsilon>0 $:
    \begin{equation*}
      \q\q\text{ \q
        The asymptotic bound }\eqref{eq:scatdeepsilon92} 
    \text{ holds for a family of  states  } \varphi^a_{b,\varepsilon}\in \vH^a ;\, b\lneq a.
    \end{equation*}
     \item \label{item:2equi5} $\forall  a\in \vA_2$: \q $\vH^+_{\wave }(H^a)=\vH_{\ac}(H^a)$. 
  \end{enumerate}
     \end{proposition}
   \begin{proof}
     Clearly \ref{item:2equi5} implies \ref{item:1asymp4}, and
      \ref{item:2equi5}  is valid for $N=2$ thanks to  Theorem
     \ref{thm:ACn-body-short2}. 

 For  $N\geq 3$
the set $\vA_3\neq \emptyset$.  Suppose $N\geq 3$ and that
$\ref{item:1asymp4}_{N'}$ is known to imply $\ref{item:2equi5}_{N'}$
for  $N'\leq N-1$, then we need to show this  implication  for
  $N'=N$. Hence (assuming
    $a_0$ has length $N$)   we know that $\vH^+_{\wave}(H^a)=\vH_{\ac}(H^a)$ for all
    $a\in \vA_3$. Using  in addition  \eqref{eq:scatdeepsilon92}  for $a=a_0$ 
    we can then 
    mimic the above proof of Theorem \ref{thm:ACn-body-short3}
  to obtain $\ref{item:2equi5}_{N}$, completing the
  induction proof.
\end{proof}

\subsection{Proof of Propositions \ref{thm:ACn-body-short3EXC} 
and \ref {prop:asympt-compl-anoth9}}\label{sec:Proof
of Theorem (N2)}

The assumption $\psi\in \vH^+_{\ener}$ of Proposition \ref{prop:negB} was used in the proof  of the proposition to provide factors of the form $h(M_i)$ or  $h(M_b)$, $h\in C_\c^\infty(\R)$,  in the definition of $S_a(t)$ in (\ref{eq:approxINT}). More precisely the assumption allowed us to freely multiply the evolution $\psi(t)$ by $S_a(t)$  (as in  \ref{eq:approxINT}). If we only know that  $\psi\in \vH_{\ac}$, then we are not a priori allowed to use (\ref{eq:approxINT}). This suggests that we need to split an arbitrarily
given  $\psi=g(U(1))\psi\in \vH_{\ac}$ into a sum of two  pieces, 
 say $\psi=\psi_{\rm ie}+\psi_{\rm fi}$, where $\psi_{\rm ie}\in \vH^+_{\rm ie}$ and that (\ref{eq:approxINT}) applies to  $\psi$ replaced by $\psi_{\rm fi}$. 
 This splitting  would mean that  we have the following version of \eqref{eq:scatdeepsilon} at disposal.

For any $\varepsilon>0$ there exists a family of states
  $\set[[big]{\varphi_{a,\varepsilon}\in \vH\mid a\in \vA_1}$ such that 
   \begin{equation}
\label{eq:scatdeepsilon97}
      \limsup _{t\to \infty}\, \norm[\Big]{U(t)\psi_{\rm fi}-\sum_{a\in \vA_1}
        U_a(t)\varphi_{a,\varepsilon}}\leq \varepsilon .
   \end{equation} 
  Note that \eqref{eq:scatdeepsilon} is the key estimate in Subsection \ref{sec:Proof
    of Theorem (N1)}. Similarly the version given by \eqref{eq:scatdeepsilon92} is a main ingredient in  Subsection \ref{subsec:Asymptotic completeness, another potential scheme}.

  \begin{subequations}
  To make such a decomposition $\psi= \psi_{\rm ie}+\psi_{\rm
    fi}$, we need to rewrite the construction of the asymptotic observable  
  $M^+$ of Lemma \ref {lemma:asympOBser}. It is now determined   by 
  \begin{equation}
  \label{eq:stro9}
  h(M^+) \psi=\vHlim_{t\to \infty} \,U(t)^{-1}h\parb{ M}\psi(t) \in\vH_{\ac};\q h\in C_\c(\R),\,\psi(t)=U(t)\psi.
\end{equation}  Note that we here used $M$ rather than $\wt M_t$ as in
\eqref{eq:stro}, which however thanks to  \eqref{eq:similarLOC}
amounts to the same limiting operator.  We recall that
\begin{equation}\label{eq:resolution05}
  \inp{\psi, h(M^+)\psi}=\int_{\R}\, h(x)\,\vE^+_\psi(\d  x),
\end{equation} where $\vE^+(\cdot)$ is the spectral resolution of resolution of the positive operator $M^+$  as defined on the subspace
$\vE^+(\R_+)\vH_{\rm ac}\subseteq \vH_{\rm ac}$.
 \end{subequations}

Let us introduce the notation $M_0^+=M^+$, $\vE^+_0=\vE^+(\R_+)$ and $\vH^+_0=\vE^+_0\vH_{\rm ac}$. For the other factors $M_1, \dots M_I$ used in the definiton of $S_a(t)$ in (\ref{eq:approxINT}) we introduce similarly positive asymptotic observables $M^+_i$ on  subspaces  
$\vH^+_i\subseteq \vH^+_0$ by 
\begin{align*}
  &h(M_i^+) \psi=\vHlim_{t\to \infty } \,U(t)^{-1}h\parb{ M_i}\psi(t) \in\vH^+_0,\q i=1,\dots, I,\, \psi \in\vH^+_0,\\
  &h(M_i^+)=\int_{\R}\, h(x)\,\vE^+_i(\d  x),\q \vH^+_i=\vE_i^+\vH^+_0,\q \vE_i^+=\vE^+_i(\R_+).
\end{align*}

The above projections $\vE_i^+$, $i=0,\dots, I$ commute. Hence their product, say denoted by $ \vE_{\bf M}^+=\Pi_i \,\vE_i^+$, is another  orthogonal projection. On  the subspace $\vH_{\bf M}^+= \vE_{\bf M}^+ \vH_{\ac}$ we can by the same recipe introduce self-adjoint asymptotic observables for  $M_b$, $b\in \vA_1$, defined on subspaces $\vH^+_b\subseteq \vH^+_{\bf M}$ by 
\begin{align*}
  &h(M_b^+) \psi=\vHlim_{t\to \infty} \,U(t)^{-1}h\parb{ M_b}\psi(t) \in \vH_{\bf M}^+ ,\q b\in \vA_1,\, \psi \in \vH_{\bf M}^+ ,\\
  &h(M_b^+)=\int_{\R}\, h(x)\,\vE^+_b(\d  x),\q \vH^+_b=\vE_b^+\vH^+_{\bf M},\q \vE_b^+=\vE^+_b(\R).
\end{align*} Their product, say denoted by $ \vE_{\vA_1}^+=\Pi_{b} \,\vE_b^+$, is another orthogonal projection, and so is $  \vE_{\rm tot}^+:=\vE_{\vA_1}^+\vE_{\bf M}^+ $ projecting onto $ \vH_{\rm tot}^+:=\vE_{\rm tot}^+\vH_{\ac}$.

Now, back to our decomposition problem $\psi= \psi_{\rm ie}+\psi_{\rm fi}$, we define $ \psi_{\rm fi}= \vE_{\rm tot}^+ \psi$ and $\psi_{\rm ie}=\psi- \ \psi_{\rm fi}$. From the very construction it follows  that indeed $\psi_{\rm ie}\in \ran( I - \vE_{\rm tot}^+ )\subseteq \vH^+_{\rm ie}$, as wanted. Here we used that $(I-\vE^+_0)\vH_{\rm ac}\subseteq \vH^+_{\rm ie}$, $(I-\vE^+_i)\vH^+ _{0}\subseteq \vH^+_{\rm ie}$  for $i=1, \dots, I$  and  also that $(I-\vE^+_b)\vH_{\bf M}^+ \subseteq \vH^+_{\rm ie}$ for $b\in \vA_1$. On the other hand the evolution of the vector $ \psi_{\rm fi}$ conforms with  (\ref{eq:approxINT}) up to an error that vanishes in the limits $\kappa\to 0$ and $K,L\to \infty$, i.e. within any  given  $\varepsilon$-tolerance. Hence indeed (\ref{eq:approxINT}) applies to  $\psi$ replaced by $\psi_{\rm fi}$, as wanted. We then conclude \eqref{eq:scatdeepsilon97} by  mimicking  our proof of \eqref{eq:scatdeepsilon} (which was based on (\ref{eq:approxINT}) for a given $\psi=g(U(1))\psi\in \vH^+_{\ener}$).
\begin{proof}[\bf Proof of the propositions] We mimic in both cases the proof of Theorem \ref{thm:ACn-body-short3}  using the input \eqref{eq:scatdeepsilon97}.

  For Propositon \ref{thm:ACn-body-short3EXC} we use the input \eqref{eq:scatdeepsilon97} (for any given   $\psi=g(U(1))\psi\in \vH_{\ac}$) in combination with  Theorem \ref{thm:ACn-body-short2}.   The proof is completed without induction. Note that indeed the argument shows that $ \psi_{\rm fi}=\psi- \psi_{\rm ie}\in \vH^+_{\wave}$. 

  For Propositon \ref {prop:asympt-compl-anoth9} we argue  similarly  (for the difficult implication), but now we use   induction  proceeding  completely similarly to Subsection \ref{subsec:Asymptotic completeness, another potential scheme} (or similarly  to Subsection \ref{sec:Proof
  of Theorem (N1)}). We  conclude that  $\psi \in \vH^+_{\wave}$, as wanted. 
  \end{proof}


\begin{thebibliography}{DoGa}

 \bibitem[Ad]{Ad} T. Adachi,  
 \emph{Asymptotic completeness for N-body quantum systems with
  long-range interactions in a time-periodic electric field}, 
Comm. Math. Phys. \textbf{275} (2007), 443--477.

\bibitem[AT]{AT} T. Adachi,\quad H. Tamura,  \emph{Asymptotic completeness for
      long-range   many-particle  systems with Stark effect. II},
  Comm. Math. Phys. \textbf{174} (1996),  537--559.

 \bibitem[BM]{BM}
A. Black,  T. Malinovitch, \emph{Scattering for Schr{\"o}dinger operators
  with conical decay},  
J.   Funct. Anal.   \textbf{288}  (2025),  37 pp.

\bibitem[De]{De}
J. Derezi\'nski, \emph{Asymptotic completeness for $N$-particle long-range quantum
    systems}, Ann. of Math.  \textbf{38}   (1993),
  427--476.


\bibitem[DG]{DG}
J. Derezi{\'n}ski, C. G{\'e}rard, \emph{Scattering theory of
 classical and quantum {$N$}-particle systems}, Texts and Monographs in
  Physics,  Berlin, Springer 1997.

\bibitem[GGM]{GGM} V. Georgescu, C. G{\'e}rard, J.S. M{\o}ller, \emph{Commutators, $C_0$-semigroups and resolvent estimates}. J. Funct. Anal. \textbf{216} (2004),  303--361. 


\bibitem[Gr]{Gra} G.M. Graf, \emph{Asymptotic completeness for
    $N$-body short-range quantum systems: a new proof},
  Commun. Math. Phys. \textbf{132} (1990), 73--101.

\bibitem[Ho]{Ho} J.S. Howland,  \emph{Scattering theory for
Hamiltonians periodic in time},   Indiana Univ. Math. J. \textbf{28}
 (1979),   471--494.


\bibitem[HMS]{HMS}I. Herbst,  J.S. M{\o}ller, E.  Skibsted, \emph{Asymptotic
    completeness for N-body Stark Hamiltonians},
  Comm. Math. Phys. \textbf{174} (1996),  509--535.

  \bibitem[HS]{HS} W. Hunziker, I. Sigal, \emph{Time-dependent
      scattering theory of $N$-body quantum systems},  Rev. Math. Phys. \textbf{12} (2000), 1033--1084.

    
 



\bibitem[H{\"o}]{H1} L. H{\"o}rmander, \emph{The analysis of linear
 partial differential operators. {I--IV}}, Berlin, Springer
 1983--85.

\bibitem[Is]{Is}
 H. Isozaki, \emph{Many-body Schr{\"o}dinger equation--Scattering Theory and Eigenfunction
Expansions}, Mathematical Physics Studies, Singapore, Springer  2023. 



\bibitem[Ka1]{Ka1} T. Kato, \emph{Wave operators and similarity for some non-selfadjoint operators},
  Math. Ann.  \textbf{162} (1966), 258--276.


\bibitem[Ka2]{Ka2} T. Kato, \emph{Smooth operators and commutators},
  Studia Math. \textbf{31} (1968), 535--546.

\bibitem[KY]{KY} H. Kitada, K. Yajima,  \emph{Bound states and scattering states for time
periodic hamiltonians}, Ann. Inst. Henri Poincar\'e \textbf{39}  (1983), 145--157.
  


  
\bibitem[MS]{MS} J.S. M\o ller, E. Skibsted,   \emph{Spectral theory of time-periodic
many-body systems},  Adv. Math. {\bf 188} 
 (2004),
137--221.


\bibitem[Na]{Na} S. Nakamura,   \emph{Asymptotic completeness for three-body Schrödinger equations with time-periodic potentials},
J. Fac. Sci. Univ. Tokyo Sect. IA Math. {\bf 33} (1986),  379--402.

\bibitem[RS]{RS}
M.~Reed, 
B.~Simon, \emph{Methods of modern mathematical physics {I}--{I\hspace{-.1em}V}},
New York, Academic Press 1972-78.

\bibitem[Ru]{Ru} W. Rudin, \emph{Functional Analysis}, New Delhi, Tata  McGraw-Hill, TMH edition 1973.

\bibitem[SS]{SS}
I. Sigal, A. Soffer, \emph{The $N$-particle scattering problem:
  asymptotic completeness for short-range systems},  Ann. of Math. (2) \textbf{126} (1987), 35--108.


\bibitem[Sk1]{Sk1}
 E. Skibsted, \emph{Stationary scattering theory: the $N$-body
   long-range case},  Commun. Math. Phys. \textbf{401} (2023),  2193--2267.
   



\bibitem[Sk2]{Sk2}
E. Skibsted, \emph{Stationary completeness: the $N$-body
  short-range case},
Adv. Math. \textbf{481} (2025), 1--50.

  \bibitem[Yaf]{Ya1} D.R. Yafaev, \emph{Radiation conditions and
    scattering theory for $N$-particle Hamiltonians},
  Commun. Math. Phys. \textbf{154}  (1993), 523--554.

\bibitem[Yaj1]{Yaj1}  K. Yajima,  \emph{Scattering theory for
Schr{\"o}dinger equations with potentials periodic in time},  
J. Math. Soc. Japan \textbf{29} (1977), 729--743.

\bibitem[Yaj2]{Yaj2}  K. Yajima,  \emph{
Boundedness of energy for $N$-body Schr{\"o}dinger  equations with time dependent small potentials}, preprint 2024,  
https://doi.org/10.48550/arXiv.2403.07400


\bibitem[Yo]{Yo}  K. Yoshida, \emph{Functional Analysis},  Berlin, Springer  1965.

\end{thebibliography}
\end{document}